\documentclass[12pt]{article}

\usepackage{epsfig,epsfig}
\usepackage{epstopdf}
\usepackage{tabu}
\usepackage{makecell}
\usepackage{booktabs}
\usepackage{amsmath}
\usepackage{mathrsfs}
\usepackage{amsfonts}
\usepackage{bm}
\usepackage{amssymb}
\usepackage{amsthm}

\usepackage{stmaryrd}
\usepackage{booktabs}
\usepackage{color}
\usepackage{multirow}
\usepackage{graphicx}
\usepackage{subfigure}
\usepackage{float}
\usepackage{appendix}
\usepackage{tikz}
\usetikzlibrary{calc}
\usetikzlibrary{matrix}
\usepackage{mathtools}
\usepackage{algorithm,algpseudocode,float}
\usepackage{lipsum}


\usepackage{lineno}
\usepackage{srcltx,graphicx}

\newtheorem{thm}{Theorem}[section]
\newtheorem{prop}[thm]{Proposition}

\newtheorem{lemma}[thm]{Lemma}

\newtheorem{example}[thm]{Example}

\newcommand{\be}{\begin{equation}}
\newcommand{\ee}{\end{equation}}

\newcommand{\ben}{\begin{equation*}}
\newcommand{\een}{\end{equation*}}

\newcommand\ba{\begin{array}}
\newcommand\ea{\end{array}}
\newcommand{\bea}{\begin{eqnarray}}
\newcommand{\eea}{\end{eqnarray}}

\newcommand{\bean}{\begin{eqnarray*}}
\newcommand{\eean}{\end{eqnarray*}}

 \def\ha{\frac 12}

 \def\bF{\mathbf{F}}
 \def\bG{\mathbf{G}}
 \def\bU{\mathbf{U}}
 \def\bV{\mathbf{V}}
 \def\bS{\mathbf{S}}

  \def \dlt{\Delta}

\usepackage{hyperref}

\numberwithin{equation}{section} \topmargin=-2.0cm \oddsidemargin=1cm
\usepackage{authblk}
\begin{document}
\title{\textbf{Well-balanced fifth-order finite volume WENO schemes with constant subtraction technique for shallow water equations}}

\author{Lidan Zhao\footnote{School of Mathematics, Jilin University, Changchun, 130012, China. Email: zhaold21@mails.jlu.edu.cn.},  ~
Zhanjing Tao\footnote{School of Mathematics, Jilin University, Changchun, 130012, China. Email: zjtao@jlu.edu.cn.},~
and Min Zhang\footnote{Corresponding author. National Engineering Laboratory for Big Data Analysis and Applications, Peking University, Beijing, 100871, China.~Chongqing Research Institute of Big Data, Peking University, Chongqing, 401121, China.  Email: minzhang@pku.edu.cn.}
}

\date{}

\maketitle
\begin{abstract}
In this paper, we propose a new well-balanced fifth-order finite volume WENO method for solving one- and two-dimensional shallow water equations with bottom topography.
The well-balanced property is crucial to the ability of a scheme to simulate perturbation waves over the ``lake-at-rest'' steady state such as waves on a lake or tsunami waves in the deep ocean.
We adopt the constant subtraction technique such that both the flux gradient and source term in the new pre-balanced form vanish at the lake-at-rest steady state, while the well-balanced WENO method by Xing and Shu [Commun. Comput. Phys., 2006] uses high-order accurate numerical discretization of the source term and makes the exact balance between the source term and the flux gradient, to achieve the well-balanced property.
The scaling positivity-preserving limiter is used for the water height near the dry areas.
The fifth-order WENO-AO reconstruction is used to construct the solution since it has better resolution than the WENO-ZQ and WENO-MR reconstructions for the perturbation of steady state flows.
Extensive one- and two-dimensional numerical examples are presented to demonstrate the well-balanced, fifth-order accuracy, non-oscillatory, and positivity-preserving properties of the proposed method.

\end{abstract}

\vspace{5pt}


\normalsize \vskip 0.2in
\newpage

\section{Introduction}
\label{sec:intro}

The shallow water equations (SWEs) with source terms play a critical role in the modeling and simulation of flows in rivers and coastal areas. They have wide applications in ocean and hydraulic engineering, such as tidal flows in estuary and coastal water regions, bore wave propagation, rivers, reservoirs, and open channel flows, etc.
{
The one-dimensional (1D) SWEs take the original form
\begin{equation}\label{swe-1d}
\frac{\partial}{\partial t}
\begin{bmatrix*}[c]
h\\
hu
\end{bmatrix*}
+\frac{\partial}{\partial x}
\begin{bmatrix*}[c]
hu\\
hu^2+\ha gh^2
\end{bmatrix*} =
\begin{bmatrix*}[c]
0\\
-ghb_x
\end{bmatrix*},
\end{equation}
}
where $h=h(x,t)\geq0$ is the water height, $u$ is the water velocity, $b=b(x)$ represents the bottom topography, and $g$ is the gravitational constant.
{ The source term in \eqref{swe-1d} only takes into account the effect of the bottom topography.}

There are some challenges in developing robust and accurate numerical methods for the system \eqref{swe-1d}.
One is that the system produces both smooth and non-smooth solutions (e.g., the hydraulic jumps/shock waves, rarefaction waves) even when the initial date is smooth. Second, we need to ensure the nonnegativity of the water height near the dry situation.  Third is that the system admits steady state solutions in which the flux gradient is balanced by the source term exactly. An important steady state solution is the so-called ``lake-at-rest":
\be \label{steady state}
h+b=C,  \quad hu=0.
\ee

Research on numerical methods for the solution of the SWEs has attracted tremendous attention in the past three decades. A significant result in computing such solutions was given
by Bermudez and Vazquez \cite{bermudez1994upwind} in $1994$. They proposed the idea of
the exact ``C-property'', which refers to the ability of the scheme to exactly preserve the lake-at-rest steady state solution \eqref{steady state}.
Such numerical methods
are often regarded as well-balanced methods.  Since then, many well-balanced methods have been developed to approximate the SWEs.
LeVeque \cite{leveque1998balancing} introduced a quasi-steady wave-propagation algorithm by designing a Riemann problem in the center of each grid cell such that the flux difference exactly cancels the source term.
Zhou et al. \cite{zhou2001SR} proposed a surface gradient method for the treatment of the source terms based on an accurate reconstruction of the conservative variables at cell interfaces.
Audusse et al. \cite{audusse2004fast} introduced the concept of hydrostatic reconstruction. Tang et al. \cite{tang2004gas} extended a kinetic flux vector splitting scheme to solve the SWEs with source terms.
Chen et al.~\cite{chen2022unified} presented a second-order accurate bottom-surface-gradient method.
Vukovic and Sopta \cite{vukovic2002eno} proposed the well-balanced finite difference ENO and WENO schemes with the decomposed source term, where the ENO and WENO reconstruction are applied to both the flux and the source term.
{ {Xing and Shu designed high-order well-balanced finite volume WENO  and discontinuous Galerkin methods \cite{xing2006new,xing2011high,xing2010positivity}.}}  Li et al. \cite{li2014high} developed a well-balanced fourth-order central WENO method for the one-dimensional SWEs.
Recently, Wang et al. \cite{WANG20201387} presented a fifth-order finite difference well-balanced multi-resolution WENO method.
Zhao and Zhang \cite{zhao2023well} proposed a well-balanced fifth-order finite difference Hermite WENO method.
More related well-balanced high-order methods include
finite difference methods \cite{huang2022high,li2012hybrid,xing2005high,xing2006high,zhang2023high,zhu2016new}, finite volume methods \cite{chen2017new,chen2022unified,noelle2007high,qian2022positivity,zhang2023structure}, and DG methods \cite{bokhove2005flooding,bunya2009wetting,du2019well,xing2010positivity,
zhang2021high,
CiCP-31-94,zhang2010positivity}.

We consider the well-balanced fifth-order finite volume WENO-AO scheme with the constant subtraction technique (CST) for SWEs in this paper, denoted as WENOAO-CST.
{
The scaling positivity-preserving limiters \cite{liu1996nonoscillatory,XU20092194,zhang2010maximum,zhang2010positivity} are used to preserve the positivity of water height which should be non-negative physically.}
The WENO-AO reconstruction \cite{BALSARA2016780} is used to construct the solution at the cell interface since the WENO-AO reconstruction has better resolution than the fifth-order WENO-ZQ and WENO-MR reconstructions for the simulation of small perturbation (cf. Figs.~\ref{figex3002}--\ref{figex3003} of Example \ref{small perturbation} in \S\ref{num_sec}).
The CST is adopted to achieve the well-balanced property of the method, such that both the flux gradient and source term { in the new pre-balanced form (denoted as the `` CST pre-balanced form")} vanish at the lake-at-rest steady state.
The CST was first proposed in \cite{yang2015well} to construct well-balanced central schemes for SWEs. Then, Du et al. \cite{du2019well} extended the CST to construct a well-balanced DG method for SWEs on unstructured meshes.
In this paper, we further explore the potential of the CST to develop a high-order well-balanced finite volume WENO method for SWEs.

Compared to the well-balanced WENO method by Xing and Shu (denoted as WENOJS-XS) \cite{xing2006new}, the well-balanced property of our WENOAO-CST method is automatically achieved since { {not only the flux gradient is balanced by the source term in the CST pre-balanced form (cf. \eqref{1D_eqa-1new}), but also}} both flux gradient and source term { {of the CST pre-balanced form}} tend to be zero when the lake-at-rest steady state is reached. However, the WENOJS-XS method uses high-order accurate numerical discretization of the source terms, which mimics the approximation of the flux term, so that the exact balance between the source term and the flux gradient can be achieved numerically at the steady state. To be specific, the WENOJS-XS method needs to use the same non-linear WENO weights computed from $h$ on $b$ to obtain corresponding reconstructions of $b$, and the numerical integration of the source term must be computed exactly by using a suitable Gauss quadrature rule.
It is worth pointing out that, unlike the WENOJS-XS method, the numerical integration { {for the  source term of the  CST pre-balanced form}} does not need to be computed exactly in the WENOAO-CST method, and we only need to use the numerical quadrature rule for the  source term to satisfy the fifth-order accuracy.
Thus, it is obvious that the procedure of the WENOAO-CST method is better implemented and the format becomes easier than the WENOJS-XS.
In addition, the WENOAO-CST method is more efficient than the WENOJS-XS method in the sense that the former leads to a smaller error than the latter for a fixed amount of CPU time (cf. Example \ref{accuracy test} and Example \ref{accuracy test-2D} in \S\ref{num_sec}).

The organization of the paper is as follows.  A well-balanced and positivity-preserving fifth-order finite volume WENOAO-CST method for the one-dimensional SWEs is developed in  \S\ref{WENO_CST_1D}, and the extension to two dimensions in a dimension-by-dimension manner is described in \S\ref{sec_2D}.
One- and two-dimensional numerical results are presented in \S\ref{num_sec} to show the accuracy, high resolution, well-balanced, and positivity-preserving properties
of the proposed WENOAO-CST method.
Finally, the conclusions and remarks are given in \S\ref{sec_con}.

\section{Well-balanced WENOAO-CST method for 1D SWEs}
\label{WENO_CST_1D}

In this section, we present a finite volume fifth-order WENOAO-CST method for 1D SWEs.
The method can preserve the lake-at-rest steady state solution (i.e., well-balanced) and the non-negativity of the water depth $h$ (i.e., positivity-preserving/PP).

The computational domain is $ \Omega=[a,b]$.
For simplicity of presentation, the uniform cells  $I_{i} = [x_{i-\ha},x_{i+\ha}],~i=1,...,N_x$ are used with mesh size $\Delta x=x_{i+\ha}-x_{i-\ha}$ and cell center $x_i=(x_{i-\ha}+x_{i+\ha})/2$.
{ {
Many researchers apply a special splitting technique for the source term (e.g., \cite{li2012hybrid,WANG20201387,xing2005high}) to design the high-order well-balanced method for SWEs.
To avoid the splitting technique for the source term, an alternative approach uses the pre-balanced form  \cite{CANESTRELLI2009834,li2014high,LIANG2009873}, and the 1D SWEs are given by the equilibrium variable $\mathbf{U}=(H=h+b,hu)^T$
\bea\label{1D_eqa-preb}
\frac{\partial}{\partial t}
\underbrace{\begin{bmatrix*}[c]
  H\\
  hu
\end{bmatrix*}}_{\mathbf{U}}
+\frac{\partial}{\partial x}
\underbrace{\begin{bmatrix*}[c]
hu\\
\frac{(hu)^2}{ H-b}+\ha gH^2-gHb
\end{bmatrix*}}_{\mathcal{F}(\mathbf{U})}
=
\underbrace{\begin{bmatrix*}[c]
  0\\
-gHb_x
\end{bmatrix*}}_{\mathcal{S}(H, b_{x})}.
\eea
Inspired by this idea, Yang et al. \cite{yang2015well} added a term $g\overline{H}b_x$ on both sides of the second equation in \eqref{1D_eqa-preb}, and get a new pre-balanced form as
\bea\label{1D_eqa-1new}
\frac{\partial}{\partial t}
\underbrace{\begin{bmatrix*}[c]
  H\\
  hu
\end{bmatrix*}}_{\mathbf{U}}
+\frac{\partial}{\partial x}
\underbrace{\begin{bmatrix*}[c]
hu\\
\frac{(hu)^2}{ H-b}+g(\overline{H}-H)b+\ha gH^2
\end{bmatrix*}}_{\bF(\mathbf{U})}
=
\underbrace{\begin{bmatrix*}[c]
  0\\
g(\overline{H}-H)b_x
\end{bmatrix*}}_{\bS(H, b_{x})},
\eea
where $\overline{H}=\overline{H}(t)=\frac{1}{|\Omega|}\int_{\Omega} H(x, t) d x$ is the global spatial average of the water surface, and independent on the spatial variable $x$.
We call the equation \eqref{1D_eqa-1new} as the `` CST pre-balanced form" for simplicity of presentation.
Obviously, the  CST pre-balanced form \eqref{1D_eqa-1new} is mathematically equivalent to the pre-balanced form \eqref{1D_eqa-preb} and the original form \eqref{swe-1d}, i.e., $\frac{\partial \bf U}{\partial t} = \bm{0} $ when $\bF(\mathbf{U})_x=\bS(H, b_{x})$. Interestingly, we have $\bF(\mathbf{U})_x=\bS(H, b_{x})=\bm{0}$ in \eqref{1D_eqa-1new} for the lake-at-rest steady state solution \eqref{steady state}
due to $H = \overline{H}$.
Let $\mathcal{M}(u)= \frac{(hu)^2}{ H-b}+g(\overline{H}-H)b+\ha gH^2$, we can rewrite \eqref{1D_eqa-1new} as
\bea\label{1D_eqa-1}
\frac{\partial}{\partial t}
\begin{bmatrix*}[c]
  H\\
  hu
\end{bmatrix*}
+\frac{\partial}{\partial x}
\begin{bmatrix*}[c]
hu\\
\mathcal{M}(u)
\end{bmatrix*}
=
\begin{bmatrix*}[c]
  0\\
g(\overline{H}-H)b_x
\end{bmatrix*}.
\eea
}}

Given the cell average of $\bU$ over each cell $I_i$
\be
\overline{\bU}_{i}(t)=\frac{1}{\Delta x} \int_{I_{i}} \bU(x, t) d x,
\ee
we integrate (\ref{1D_eqa-1}) over $ I_i$ and get the semi-discrete scheme
\begin{flalign}
&\frac{d}{d t}\overline{\bU}_{i}(t)=-\frac{1}{\Delta x}\left(\hat{\bF}_{i+\ha}-\hat{\bF}_{i-\ha}\right)+\frac{1}{\Delta x}\int_{I_i}\bS(H,b_{x})dx,\label{semi-discrete_2}
\end{flalign}
where $\hat{\bF}_{i+\ha}$ is the Lax-Friedrichs numerical fluxes
\bea\label{LF-flux}
\hat{\bF}_{i+\ha}
=
\ha\left(\begin{bmatrix*}[c]
(hu)^-_{i+\ha} \\
\mathcal{M}(u)^-_{i+\ha}
\end{bmatrix*}
+\begin{bmatrix*}[c]
(hu)^+_{i+\ha} \\
\mathcal{M}(u)^+_{i+\ha}
\end{bmatrix*}\right)
-\frac{\alpha}{2}\left(\begin{bmatrix*}[c]
h^+_{i+\ha} \\
(hu)^+_{i+\ha}
\end{bmatrix*}
-\begin{bmatrix*}[c]
h^-_{i+\ha} \\
(hu)^-_{i+\ha}
\end{bmatrix*}\right),
\eea
and
\begin{align*}
&h^{\pm}_{i+\ha} = H^{\pm}_{i+\ha}-b^{\pm}_{i+\ha},\quad\quad
\alpha=\max_i(|\overline{u}_i|+\sqrt{g\overline{h}_i}),\\
&\mathcal{M}(u)^{\pm}_{i+\ha}=\frac{\left((hu)^{\pm}_{i+\ha}\right)^2}{ h^\pm_{i+\ha}}+g\left(\overline{H}(t)-H^\pm_{i+\ha}\right)b^\pm_{i+\ha}+\ha g\left(H^\pm_{i+\ha}\right)^2.
\end{align*}
Notice that we replace $H$ by $h$ { {in the first equation of \eqref{LF-flux}}}, and it becomes easy to show the positivity-preserving property, see \S\ref{sub2.2}.

The four-point Gauss-Lobatto quadrature rule is used to approximate the source term
\be\label{source_1D}
\frac{1}{\dlt x}\int_{I_i}\bS\left(H,b_{x}\right)dx = \sum_{l=1}^4\omega_l \bS\left(H_i^l,\left(b_{x}\right)_{i}^l\right),
\ee
where $H_{i}^l = H(x_i^l)$, $\left(b_{x}\right)_{i}^{l} = b_{x} (x_i^l)$,
and the points
$\left\{x_i^l\right\}_{l=1}^4= \left\{x_{i-\ha},x_{i-\frac{\sqrt{5}}{10}},x_{i+\frac{\sqrt{5}}{10}},x_{i+\ha}\right\}$
with the quadrature weights $\omega_1 =\omega_4 = \frac{1}{12}$ and $\omega_2 =\omega_3 = \frac{5}{12}$.

The scheme \eqref{semi-discrete_2} is not well-balanced in general since the right-hand side of \eqref{semi-discrete_2} is not zero for the lake-at-rest steady state.
Specially, the numerical flux term $\hat{\bF}_{i+\ha}-\hat{\bF}_{i-\ha}$ does not equal to zero since $h_{i +\ha}^{+} \neq h_{i +\ha}^{-}$.
We follow the idea of Audusse et al. \cite{audusse2004fast} and set
\be \label{hydrostatic_reconstruction}
h^{*,\pm}_{i+\ha}=\max\left(0,H^\pm_{i+\ha}-\max \left(b^+_{i+\ha},b^-_{i+\ha} \right) \right).
\ee
{ For the lake-at-rest steady state, we can enforce $h^{*,-}_{i+\ha} = h^{*,+}_{i+\ha}$ and $h^{*,\pm}_{i+\ha} \geqslant 0$ by \eqref{hydrostatic_reconstruction}}. These properties will be used in the proof for the well-balanced and positivity-preserving properties of the method.

Replacing $h^{\pm}_{i+\ha}$ by $h^{*,\pm}_{i+\ha}$ in the numerical fluxes \eqref{LF-flux} leads to
\bea\label{LF-flux-m}
\hat{\bF}^{*}_{i+\ha}
=\begin{bmatrix*}[c]
\hat{\bF}^{*,H}_{i+\ha}\\
\hat{\bF}^{*,hu}_{i+\ha}
\end{bmatrix*}=
\ha\left(\begin{bmatrix*}[c]
h^{*,-}_{i+\ha}u^{-}_{i+\ha} \\
\mathcal{M}(u)^{*,-}_{i+\ha}
\end{bmatrix*}
+\begin{bmatrix*}[c]
h^{*,+}_{i+\ha}u^{+}_{i+\ha} \\
\mathcal{M}(u)^{*,+}_{i+\ha}
\end{bmatrix*}\right)
-\frac{\alpha}{2}\left(\begin{bmatrix*}[c]
h^{*,+}_{i+\ha} \\
h^{*,+}_{i+\ha}u^{+}_{i+\ha}
\end{bmatrix*}
-\begin{bmatrix*}[c]
h^{*,-}_{i+\ha}\\
h^{*,-}_{i+\ha}u^{-}_{i+\ha}
\end{bmatrix*}\right),
\eea
where
$$
\mathcal{M}(u)^{*,\pm}_{i+\ha}=\frac{\left(h^{*,\pm}_{i+\ha}u^{\pm}_{i+\ha}\right)^2}{ h^{*,\pm}_{i+\ha}}+g\left(\overline{H}(t)-H^\pm_{i+\ha}\right)b^\pm_{i+\ha}+\ha g\left(H^\pm_{i+\ha}\right)^2.
$$
Finally, we can obtain the actual semi-discrete scheme
\begin{flalign}
&\frac{d}{d t}\overline{\bU}_{i}(t)=\mathcal{L}(\overline{\bU}_i):=-\frac{1}{\Delta x}\left(\hat{\bF}^{*}_{i+\ha}-\hat{\bF}^{*}_{i-\ha}\right)+\frac{1}{\dlt x}\int_{I_i}\bS\left(H,b_{x}\right)dx.\label{semi-discrete-wb}
\end{flalign}
To compute the source term $\int_{I_i}\bS(H,b_x)dx$ and numerical fluxes term $\hat{\bF}_{i+\ha}-\hat{\bF}_{i-\ha}$, one needs to reconstruct  certain quantities in \eqref{source_1D} and \eqref{LF-flux-m}
$$H_{i +\ha}^{\pm},\quad (hu)_{i +\ha}^{\pm},\quad b_{i +\ha}^{\pm},\quad H_{i}^l, \quad\left(b_{x}\right)_{i}^{l}, \quad l=1,\cdots, 4,$$
based on the given data $\{ \overline {\bU}_i(t) \}_i$ and $\{ \overline {b}_i \}_i$. These quantities are obtained by the fifth-order WENO-AO reconstruction \cite{BALSARA2016780} which is given in Appendix A
and $h_{i +\ha}^{\pm} = H_{i +\ha}^{\pm}-b_{i +\ha}^{\pm}$.
Note that the WENO reconstructions for $b$ and $b_x$ are implemented only once at the initial time.

\begin{prop}
\label{wb}
The semi-discrete finite volume WENOAO-CST scheme (\ref{semi-discrete-wb}) for
1D SWEs is well-balanced, i.e., it preserves the lake-at-rest steady state $H=h+b = C,~hu=0$.
\end{prop}
\begin{proof}
For the  lake-at-rest steady state $H=h+b = C$ and $hu=0$, we have
\ben
\overline{H} = \overline{H}_i = H =  C,
\quad  (\overline{hu})_i = hu = 0,
\een
and
$$\int_{I_i}\bS(H,b_x)dx =
\begin{bmatrix*}[c]
0 \\
\int_{I_i}g(\overline{H}-H)b_x dx
\end{bmatrix*}
=\begin{bmatrix*}[c]
0 \\
0
\end{bmatrix*}.$$
From the WENO reconstruction procedures, we can obtain
\ben
H_{i +\ha}^{\pm} = C,\quad \quad (hu)_{i +\ha}^{\pm} = 0.
\een
Moreover, we have
$ {h}^{*,-}_{i+\ha}={h}^{*,+}_{i+\ha},~
  h_{i+\ha}^{*,-}u_{i+\ha}^- =h_{i+\ha}^{*,+}u_{i+\ha}^+= 0
$  based on \eqref{hydrostatic_reconstruction},
and the numerical fluxes $\hat{\bF}^*_{i+\ha}$ satisfy
 \begin{align*}
\hat{\bF}^{*,H}_{i+\ha}
 &=\frac{1}{2}\left(h_{i+\ha}^{*,-}u_{i+\ha}^- + h_{i+\ha}^{*,+}u_{i+\ha}^+\right)
 -\frac{\alpha}{2}\left(h_{i+\ha}^{*,+}-h_{i+\ha}^{*,-}\right)=0,
\\
\hat{\bF}^{*,hu}_{i+\ha}
&=
\frac{1}{2}\left(\frac{1}{2}g(H_{i+\ha}^{-})^2 + \frac{1}{2}g(H_{i+\ha}^{+})^2\right)
-\frac{\alpha}{2}\left(h_{i+\ha}^{*,+}u_{i+\ha}^+-h_{i+\ha}^{*,-}u_{i+\ha}^-\right) =  \frac{1}{2} g C^2.
 \end{align*}
 Similarly, we know that
$\hat{\bF}^*_{i-\ha}$ satisfy
 \begin{align*}
\hat{\bF}^{*,H}_{i-\ha}
&=\frac{1}{2}\left(h_{i-\ha}^{*,-}u_{i-\ha}^- + h_{i-\ha}^{*,+}u_{i-\ha}^+\right)
-\frac{\alpha}{2}\left(h_{i-\ha}^{*,+}-h_{i-\ha}^{*,-}\right)=0,\\
\hat{\bF}^{*,hu}_{i-\ha}
&
=\frac{1}{2}\left(\frac{1}{2}g(H_{i-\ha}^{-})^2+ \frac{1}{2}g(H_{i-\ha}^{+})^2\right)
-\frac{\alpha}{2}\left(h_{i-\ha}^{*,+}u_{i-\ha}^+-h_{i-\ha}^{*,-}u_{i-\ha}^-\right) =  \frac{1}{2} g C^2.
 \end{align*}
And then, we have $\hat{\bF}^{*}_{i+\ha}-\hat{\bF}^{*}_{i-\ha}=\mathbf{0}$ and $\mathcal{L}(\overline{\bU}_i) = \mathbf{0}$ . Thus, the scheme (\ref{semi-discrete-wb}) is well-balanced.
\end{proof}

\vspace{9pt}

We now list a few benefits of our WENOAO-CST method compared with the WENOJS-XS method \cite{xing2006new}.
{
Firstly,
when the lake-at-rest steady state is reached, the CST pre-balanced form can easily balance the flux gradient and source term because both the flux gradient term and source term are zero, i.e., $\bF(\mathbf{U})_x=\bS(H, b_x)=\bm{0}$.}
Secondly, we use the four-point Gauss-Lobatto quadrature rule only to satisfy the fifth-order accuracy of our WENOAO-CST method. However, the numerical integration for the source term in the WENOJS-XS method should be exact for the polynomials with a suitable Gauss quadrature rule for the well-balanced property.
Thirdly, the WENO reconstruction for  $b$ is implemented only once at the initial time in the WENOAO-CST method, while the reconstruction of $b$ needs to be implemented and updated at each time level for the WENOJS-XS method, and the same non-linear weights of $h$ in the WENO reconstruction are used for the reconstruction of $b$. It leads to our WENOAO-CST method being much simpler and cost-efficient than the WENOJS-XS method.

The third-order strong stability preserving (SSP) Runge-Kutta scheme \cite{shu1988efficient} is applied to discretize (\ref{semi-discrete-wb}) in time
\be \label{TVD}
\begin{cases}
\overline{\bU}_i^{(1)}=\overline{\bU}_i^{n}+\Delta t \mathcal{L}\left(\overline{\bU}_i^{n}\right), \\
\overline{\bU}_i^{(2)}=\frac{3}{4} \overline{\bU}_i^{n}+\frac{1}{4} \overline{\bU}_i^{(1)}+\frac{1}{4} \Delta t \mathcal{L}\left(\overline{\bU}_i^{(1)}\right),\\
\overline{\bU}_i^{n+1}=\frac{1}{3} \overline{\bU}_i^{n}+\frac{2}{3} \overline{\bU}_i^{(2)}+\frac{2}{3} \Delta t \mathcal{L}\left(\overline{\bU}_i^{(2)}\right).
\end{cases}
\ee

\subsection{The calculation of $\mathcal{L}(\overline{\bU}_i)$}
\label{sec:rec-flow}

In this section, we present the calculation of $\mathcal{L}(\overline{\bU}_i)$ in the following four steps.

\begin{itemize}
\item [\textbf{Step 1.}]
At the initial time $t^0$,
reconstruct the quantities $b_{i +\ha}^{\pm}$ and $\left(b_{x}\right)_{i}^{l}, l=1,\cdots, 4,$ in each cell $I_i$  based on the cell averages $\{\overline{b}_i \}_i$ by the fifth-order WENO-AO reconstruction.
\item [\textbf{Step 2.}]
At each time step $t^n,~n=0,1,...$, reconstruct the quantities
$$H_{i +\ha}^{\pm}, \quad (hu)_{i +\ha}^{\pm}, \quad H_{i}^{l}, \quad l=1,..., 4,$$
in each cell $I_i$ based on the cell averages $\{ \overline{H}_i^{n}\}_i $ and $\{\overline{hu}_i^{n} \}_i$ by the fifth-order WENO-AO reconstruction,
and let
$h_{i+ \ha}^{\pm}=H_{i +\ha}^{\pm}-b_{i+ \ha}^{\pm}$.

\item [\textbf{Step 3.}]
Compute $h_{i+ \ha}^{*,\pm}$ by (\ref{hydrostatic_reconstruction}).
\item [\textbf{Step 4.}]
Compute the source term $\int_{I_i}\bS(H,b_x)dx$ by \eqref{source_1D} and the flux $\hat{\bF}^*_{i+\ha}$ by (\ref{LF-flux-m}).
Finally, we can obtain $\mathcal{L}\left(\overline{\bU}_i^{n}\right)$ by (\ref{semi-discrete-wb}).
\end{itemize}

\subsection{Preservation of positivity for water height $h$}\label{sub2.2}

Another challenge in the numerical simulation for the SWEs is to preserve the nonnegativity of the water height $h$ in the computation.
The dry area often emerges in the simulation of dam breaks, flood wave flows, and run-up phenomena over shores and sea defense structures. Standard numerical methods may fail around the dry area and result in non-physical negative water height.
In this section, we show that the WENOAO-CST method with the scaling PP limiter can preserve the non-negativity of the water height $h$. {  For further details about the PP limiter, we refer to \cite{liu1996nonoscillatory,XU20092194,zhang2010maximum,zhang2010positivity}.}

At time level $t^n$, we first apply the Euler forward in time for the first equation of \eqref{semi-discrete-wb} and obtain
\begin{flalign}
&\overline{H}_{i}^{n+1}=\overline{H}_{i}^{n}-\lambda\left(\hat{\bF}^{*,h}_{i+\ha}
-\hat{\bF}^{*,h}_{i-\ha}
\right), \quad  \lambda = \frac{\Delta t}{\Delta x}.\label{discrete_1}
\end{flalign}
Due to $\overline{b}_{i}^{n+1} = \overline{b}_{i}^{n}$, \eqref{discrete_1} becomes
\begin{flalign}
&\overline{h}_{i}^{n+1}=\overline{h}_{i}^{n}-\lambda\left(\hat{\bF}^{*,h}_{i+\ha}
-\hat{\bF}^{*,h}_{i-\ha}\right).     \label{discrete_10}
\end{flalign}
Following the idea in \cite{xing2011high}, we introduce the  auxiliary variable
\begin{flalign}
\xi_{i}=\frac{\overline{h}_{i}^{n}-w_{1} h_{i-\frac{1}{2}}^{+}-w_{4} h_{i+\frac{1}{2}}^{-}}{1-w_{1}-w_{4}}, \label{Xi}
\end{flalign}
and have the following result.
\begin{prop}\label{prop_ahpha_1D}
Consider the scheme (\ref{discrete_10}) satisfied by the cell averages of the water height. If  $\overline{h}_i^n,~h_{i-\frac{1}{2}}^{\pm},~ h_{i+\frac{1}{2}}^{\pm}$ and  $\xi_{i}$ defined in (\ref{Xi}) are all non-negative, then  $\overline{h}_{i}^{n+1}$ is also non-negative under the CFL condition
\be \label{CFL_con}
\lambda \alpha \leqslant w_{1}.
\ee
\end{prop}
\begin{proof}
The detailed proof is given in Appendix B.
\end{proof}

\vspace{8pt}

Next, we present the steps of the well-balanced PP WENOAO-CST method for the 1D SWEs.
The same steps as in \S\ref{sec:rec-flow} are used except for $\textbf{Step 3}$, so we omit the same steps and replace  $\textbf{Step 3}$ by  $\textbf{Step 3a}$ and $\textbf{Step 3b}$ which are given as follows.
\begin{itemize}
\item [\textbf{Step 3a.}]
Compute
$$\overline{h}_i^n = \overline{H}_i^n - \overline{b}_i^n, \quad
\xi_{i}=\frac{\overline{h}_{i}^{n}-w_{1} h_{i-\frac{1}{2}}^{+}-w_{4} h_{i+\frac{1}{2}}^{-}}{1-w_{1}-w_{4}},$$
and setup a small number $\eta=\min\limits_i\{10^{-13},\overline h_i\}$.
Then evaluate
\be
\tilde{h}_{i-\ha}^{+}=\theta\left(h_{i-\ha}^{+}-\overline{h}_i^n\right)+\overline{h}_i^n, \qquad \tilde{h}_{i+\ha}^{-}=\theta\left(h_{i+\ha}^{-}-\overline{h}_i^n\right)+\overline{h}_i^n, \nonumber
\ee
with
$$\theta=\min\left(1,\frac{\overline{h}_i^n-\eta}{\overline{h}_i^n-m_i}\right), \qquad m_i=\min\left(h_{i - \ha}^{+},h_{i + \ha}^{-},\xi_i\right). $$
Note that this PP limiter does not maintain the well-balanced property in general,
and we apply
\bea\label{corr-b}
\tilde{b}_{i - \ha}^{+}=H_{i -\ha}^{+}-\tilde{h}_{i - \ha}^{+}, \quad \tilde{b}_{i + \ha}^{-}=H_{i + \ha}^{-}-\tilde{h}_{i + \ha}^{-}
\eea
to get the new approximation for bottom topography \cite{zhang2021high}.
{ Obviously, for the lake-at-rest steady state \eqref{steady state}, we have
\bea\label{corr-b2}
\tilde{b}_{i - \ha}^{+}+\tilde{h}_{i - \ha}^{+}=H_{i -\ha}^{+}=C, \quad \tilde{b}_{i + \ha}^{-}+\tilde{h}_{i + \ha}^{-}=H_{i + \ha}^{-}=C.
\eea
Thus, the above positivity-preserving procedure can maintain the well-balanced property of the WENOAO-CST method.}

\item [\textbf{Step 3b.}]
 Compute $\tilde{h}_{i + \ha}^{*,\pm}$ by (\ref{hydrostatic_reconstruction}) and use them instead of $h_{i + \ha}^{*,\pm}$ in the WENOAO-CST method (\ref{semi-discrete-wb}) under the CFL condition
(\ref{CFL_con}).
\end{itemize}

Since the third-order SSP Runge-Kutta scheme is a convex combination of the Euler forward method, therefore our WENOAO-CST method \eqref{TVD} with PP limiter is positivity-preserving.

\section{Well-balanced WENOAO-CST method for 2D SWEs}
\label{sec_2D}

In this section, we propose a well-balanced finite volume fifth-order WENOAO-CST method for 2D SWEs in a dimension-by-dimension manner. Then, a well-balanced PP WENOAO-CST method is also provided to ensure that no negative water height is produced during the computation.

Consider the 2D SWEs as
\begin{equation}\label{swe-2d}
\frac{\partial}{\partial t}
\begin{bmatrix*}
  h\\
  hu\\
  hv\\
\end{bmatrix*} +\frac{\partial}{\partial x}
 \begin{bmatrix*}[c]
h u\\
h u^{2}+\frac{1}{2} g h^{2}\\
h u v\\
\end{bmatrix*} +
\frac{\partial}{\partial y}
\begin{bmatrix*}[c]
h v\\
h u v \\
h v^{2}+\frac{1}{2} g h^{2}\\
\end{bmatrix*}  = \begin{bmatrix*}[c]
0\\
-g h b_{x}\\
-g h b_{y}\\
\end{bmatrix*}.
\end{equation}

Define the computational domain $ \Omega=[a,b]\times[c,d]$.
Uniform meshes are used with the mesh sizes  $\Delta x$ and  $\Delta y$ in the  $x$ and $y$ direction, respectively. Each cell of the mesh is denoted as  $ I_{i ,j}=[x_{i-\ha}, x_{i+\ha}] \times[y_{j-\ha}, y_{j+\ha}]$, $i=1,...,N_x,~j=1,...,N_y$, with its cell center  $(x_{i}, y_{j})=\left((x_{i-\ha}+x_{i+\ha})/2, ~(y_{j-\ha}+y_{j+\ha})/2\right)$.
{
Similarly as the one-dimensional, we rewrite \eqref{swe-2d} into the CST pre-balanced form with equilibrium variable $(H=h+b,hu,hv)^T$ as}
\bea\label{2D_eqa-1}
&\frac{\partial}{\partial t}
\underbrace{\begin{bmatrix*}[c]
  H\\
  hu\\
  hv
\end{bmatrix*}}_{\mathbf{U}}
+\frac{\partial}{\partial x}
\underbrace{\begin{bmatrix*}[c]
hu\\
\mathcal{M}(u)\\
\frac{(h u)(h v)}{H-b}
\end{bmatrix*}}_{\bF(\mathbf{U})}
+\frac{\partial}{\partial y}
\underbrace{\begin{bmatrix*}[c]
hv\\
\frac{(h u)(h v)}{H-b}\\
\mathcal{M}(v)
\end{bmatrix*}}_{\bG(\mathbf{U})}=
\underbrace{\begin{bmatrix*}[c]
 0 \\
S_1(H, b_x)\\
S_2(H, b_y)
\end{bmatrix*}}_{\bS(H, b)},
\eea
where
\begin{align*}
&\overline{H}=\overline{H}(t)=\frac{1}{|\Omega|}\int_{\Omega} H(x,y, t) d xdy,\\
&\mathcal{M}(u)= \frac{(hu)^2}{ H-b}+g(\overline{H}-H)b+\ha gH^2,\quad
 S_1(H, b_x)=g(\overline{H}-H)b_x, \\
&\mathcal{M}(v)= \frac{(hv)^2}{ H-b}+g(\overline{H}-H)b+\ha gH^2,\quad
S_2(H, b_y)=g(\overline{H}-H)b_y.
\end{align*}
We integrate (\ref{2D_eqa-1})  over $ I_{i,j}$ and get the semi-discrete scheme
\begin{flalign}
\frac{d \overline{\bU}_{i,j}(t)}{d t}=&
-\frac{1}{\Delta x\Delta y}
\left(\int_{y_{j-\ha}}^{y_{j+\ha}}\hat{\bF}^* \left(\bU_{i+\ha}(y)\right)dy
-\int_{y_{j-\ha}}^{y_{j+\ha}}\hat{\bF}^* \left(\bU_{i-\ha}(y)\right)dy \right) \nonumber \\
& -\frac{1}{\Delta x\Delta y}\left(\int_{x_{i-\ha}}^{x_{i+\ha}}\hat{\bG}^*\left(\bU_{j+\ha}(x) \right)dx
-\int_{x_{i-\ha}}^{x_{i+\ha}}\hat{\bG}^*\left(\bU_{j-\ha}(x) \right)dx\right) \nonumber
\\
&+\frac{1}{\Delta x\Delta y}\int_{I_{i,j}}\bS(H,b)dx dy,\label{semi-discrete_3}
\end{flalign}
where $\overline{\bU}_{i,j}(t) \approx \frac{1}{\Delta x\Delta y} \int_{I_{i,j}} \bU(x,y,t) dx dy$, $\bU_{i+\ha}(y) \approx \bU(x_{i+\ha},y)$, $\bU_{j+\ha}(x) \approx \bU(x,y_{j+\ha})$.
$\hat{\bF}^* \left(\bU_{i+\ha}(y)\right)$ and $\hat{\bG}^*\left(\bU_{j+\ha}(x) \right)$ are the Lax-Friedrichs fluxes defined on edges $(x_{i+\ha},y)$ and $(x,y_{j+\ha})$, respectively, viz.
\begin{align*}
&\hat{\bF}^* \left(\bU_{i+\ha}(y)\right)
=\frac{1}{2} \left(
 \bF\left(\bU_{i+\ha}^{*,-}(y)\right)
+\bF\left(\bU_{i+\ha}^{*,+}(y)\right)-
\alpha_1\left(\bV_{i+\ha}^{*,+}(y)-\bV_{i+\ha}^{*,-}(y)\right)
\right), \\
&\hat{\bG}^*\left(\bU_{j+\ha}(x) \right)
=\frac{1}{2}\left(
 \bG\left(\bU_{j+\ha}^{*,-}(x) \right)
+\bG\left(\bU_{j+\ha}^{*,+}(x) \right)-
\alpha_2\left(\bV_{j+\ha}^{*,+}(x)-\bV_{j+\ha}^{*,-}(x)\right)\right),\\
\nonumber
&~~~~\bU_{i+\ha}^{*,\pm}(y)=\left[\begin{array}{ccc}
H_{i+\ha}^{\pm}(y) \\
h_{i+\ha}^{*,\pm}(y)u_{i+\ha}^{\pm}(y)\\
h_{i+\ha}^{*,\pm}(y)v_{i+\ha}^{\pm}(y)
\end{array}\right],\quad
\bU_{j+\ha}^{*,\pm}(x)=\left[\begin{array}{ccc}
H_{j+\ha}^{\pm}(x) \\
h_{j+\ha}^{*,\pm}(x)u_{j+\ha}^{\pm}(x)\\
h_{j+\ha}^{*,\pm}(x)v_{j+\ha}^{\pm}(x)
\end{array}\right],
\nonumber \\
 &~~~~\bV_{i+\ha}^{*,\pm}(y)=\left[\begin{array}{ccc}
h_{i+\ha}^{*,\pm}(y) \\
h_{i+\ha}^{*,\pm}(y)u^{\pm}_{i+\ha}(y)\\
h_{i+\ha}^{*,\pm}(y)v^{\pm}_{i+\ha}(y)
\end{array}\right],\quad
\bV_{j+\ha}^{*,\pm}(x)=\left[\begin{array}{ccc}
h_{j+\ha}^{*,\pm}(x) \\
h_{j+\ha}^{*,\pm}(x)u^{\pm}_{j+\ha}(x)\\
h_{j+\ha}^{*,\pm}(x)v^{\pm}_{j+\ha}(x)
\end{array}\right],
\end{align*}
\begin{flalign}
&~~~~h_{i+\ha}^{*,\pm}(y)=\max \left(0,H_{i+\ha}^\pm(y)
-\max \left(b_{i+\ha}^+(y),b_{i+\ha}^-(y) \right) \right), \label{2D_hy_re1}\\
&~~~~h_{j+\ha}^{*,\pm}(x)=\max \left(0,H_{j+\ha}^\pm(x)
-\max \left(b_{j+\ha}^+(x),b_{j+\ha}^-(x) \right) \right),\label{2D_hy_re2}
\end{flalign}
and $\alpha_1=\max\limits_{i,j}\left(|\overline{u}_{i,j}|+\sqrt{g\overline{h}_{i,j}}\right)$,~
$\alpha_2=\max\limits_{i,j}\left(|\overline{v}_{i,j}|+\sqrt{g\overline{h}_{i,j}}\right)$.

In our computation, the line integrals in the right of \eqref{semi-discrete_3} are computed using the three-point Gauss quadrature rule as
\begin{align*}
&
\mathbf{\Phi}_{i+\ha,j} = \frac{1}{\Delta y}\int_{y_{j-\ha}}^{y_{j+\ha}}\hat{\bF}^* \left(\bU_{i+\ha}(y)\right)dy
\approx
\sum_{k=1}^3\hat{\gamma}_{k}\hat{\bF}^*\left(\bU_{i+\ha}(\hat{y}_j^k)\right),\\
&\mathbf{\Psi}_{i,j+\ha}=\frac{1}{\Delta x}\int_{x_{i-\ha}}^{x_{i+\ha}}\hat{\bG}^*\left(\bU_{j+\ha}(x) \right)dx
\approx
\sum_{k=1}^3\hat{\gamma}_{k}\hat{\bG}^*\left(\bU_{j+\ha}(\hat{x}_i^k)\right),
\end{align*}
where $\hat{\gamma}_{1} =\hat{\gamma}_{3} = \frac{5}{18}$,~$\hat{\gamma}_{2} =\frac{8}{18}$ and
\begin{align*}
&\hat{\mathcal{G}}^{x,i}=\left\{\hat{x}_i^k\right\}_{k=1}^3=\left\{x_{i-\ha\sqrt{\frac{5}{3}}},~x_{i},~x_{i+\ha\sqrt{\frac{5}{3}}}\right\},
\\&
\hat{\mathcal{G}}^{y,j}=\left\{\hat{y}_j^k\right\}_{k=1}^3=\left\{y_{j-\ha\sqrt{\frac{5}{3}}},~y_{j},~y_{j+\ha\sqrt{\frac{5}{3}}}\right\}.
\end{align*}
Let
\begin{align*}
&\mathcal{G}^{x,i}=\left\{x_i^l\right\}_{l=1}^4= \left\{x_{i-\ha},~x_{i-\frac{\sqrt{5}}{10}},~x_{i+\frac{\sqrt{5}}{10}},~x_{i+\ha}\right\},
\\&\mathcal{G}^{y,j}=\left\{y_j^l\right\}_{l=1}^4= \left\{y_{j-\ha},~y_{j-\frac{\sqrt{5}}{10}},~y_{j+\frac{\sqrt{5}}{10}},~y_{j+\ha}\right\},
\end{align*}
and the quadrature weights $\omega_1 =\omega_4 = \frac{1}{12}$ and $\omega_2 =\omega_3 = \frac{5}{12}$.
The source terms are approximately by the 2D version quadrature formula (see Fig. \ref{fig-gauss})
\begin{align*}
&\frac{1}{\Delta x\Delta y}\int_{I_{i,j}}S_1(H,b_x)dx dy \approx \sum_{k=1}^3\sum_{l=1}^4\hat{\gamma}_{k}\omega_l S_1\left(H\left(x_i^l,\hat{y}_j^k\right),b_x\left(x_i^l,\hat{y}_j^k\right)\right),
\\
&
\frac{1}{\Delta x\Delta y}\int_{I_{i,j}}S_2(H,b_y)dxdy\approx \sum_{k=1}^3\sum_{l=1}^4\hat{\gamma}_{k}\omega_l S_2\left(H\left(\hat{x}_i^k,y_j^l\right),b_y\left(\hat{x}_i^k,y_j^l\right)\right).
\end{align*}

\begin{figure}[H]
\centering
\subfigure[$ x_i^l \in \mathcal{G}^{x,i},\ \hat y_j^k \in \hat{\mathcal{G}}^{y,j}$]{
\includegraphics[width=0.47\textwidth,trim=52 0 52 10,clip]{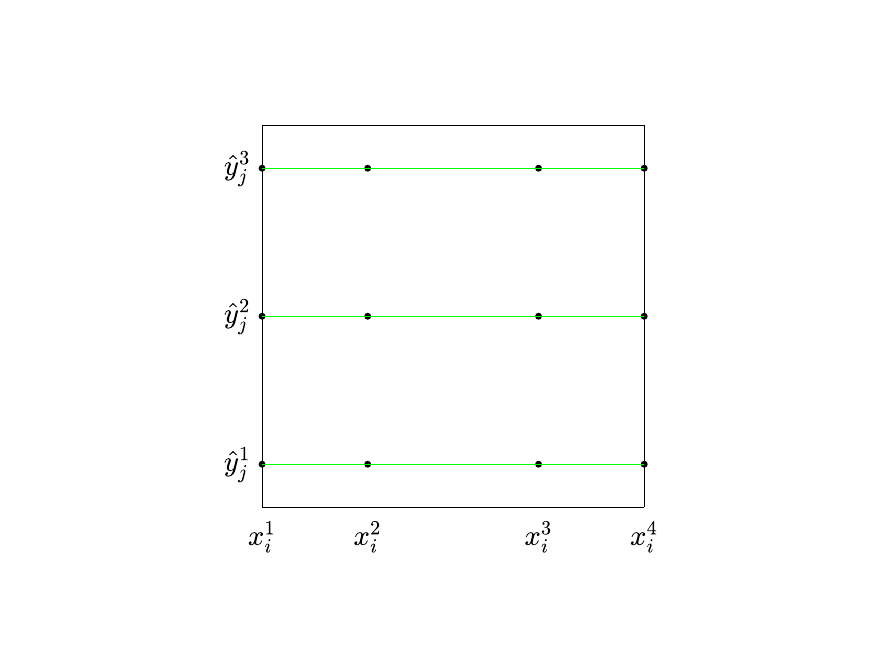}}
\subfigure[$ \hat x_i^k \in \hat{\mathcal{G}}^{x,i},\ y_j^l \in \mathcal{G}^{y,j}$]{
\includegraphics[width=0.47\textwidth,trim=52 0 52 10,clip]{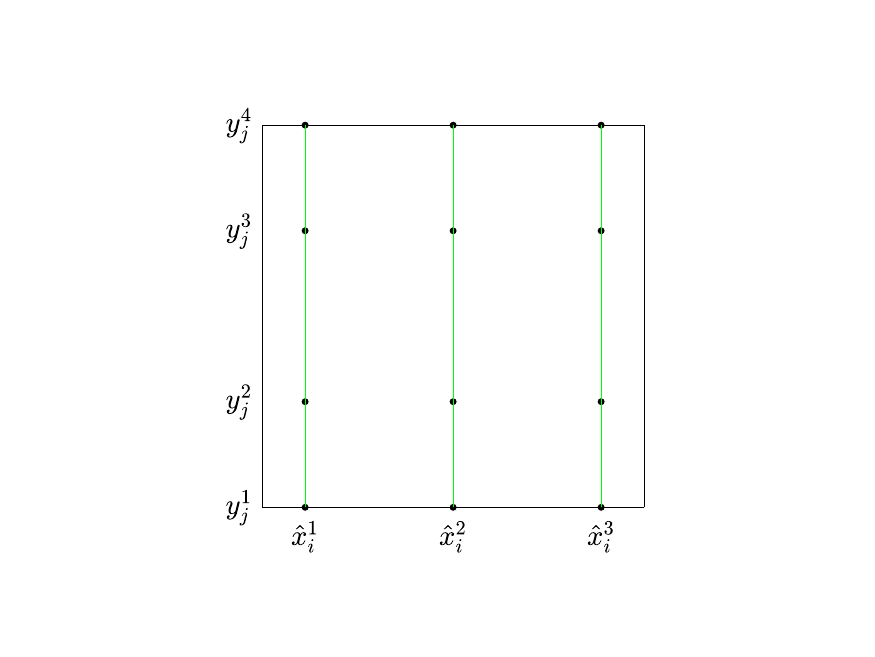}}
\caption{The quadrature points in two dimensions.  }
\label{fig-gauss}
\end{figure}

The quantities
$\bU^{\pm}(x_{i+\ha},\hat{y}_j^k)$,~
$\bU^{\pm}(\hat{x}_i^k,y_{j+\ha})$,~
$b^{\pm}(x_{i+\ha}$,~
$\hat{y}_j^k)$,~
$b^{\pm}(\hat{x}_i^k,y_{j+\ha})$,~
$H(x^{l}_i,\hat{y}_j^k)$,
$H(\hat{x}_i^k,y^{l}_j)$,
$b_x(x^{l}_i,\hat{y}_j^k)$,~
$b_y(\hat{x}_i^k,y^{l}_j)$,~
$x_i^l \in \mathcal{G}^{x,i},\ y_j^l \in \mathcal{G}^{y,j}$,~
$\hat x_i^k \in \hat{\mathcal{G}}^{x,i}$,~
$\hat y_j^k \in \hat{\mathcal{G}}^{y,j}$
are obtained by the fifth-order WENO-AO reconstruction in a dimension-by-dimension manner based on the given data $\{ \overline {\bU}_{i,j}(t) \}_{i,j}$ and $\{ \overline {b}_{i,j} \}_{i,j}$. Note that $h^{\pm} = H^{\pm}-b^{\pm}$ and the reconstructions for $b$, $b_x$, and $b_y$ are implemented only once at the initial time.

Similarly, the third-order strong stability preserving (SSP) Runge-Kutta scheme is applied to discretize (\ref{semi-discrete_3}) in time
\be \label{TVD-2d}
\begin{cases}
\overline{\bU}_{i,j}^{(1)}=\overline{\bU}_{i,j}^{n}+\Delta t \mathcal{L}\left(\overline{\bU}_{i,j}^{n}\right), \\
\overline{\bU}_{i,j}^{(2)}=\frac{3}{4} \overline{\bU}_{i,j}^{n}+\frac{1}{4} \overline{\bU}_{i,j}^{(1)}+\frac{1}{4} \Delta t \mathcal{L}\left(\overline{\bU}_{i,j}^{(1)}\right),\\
\overline{\bU}_{i,j}^{n+1}=\frac{1}{3} \overline{\bU}_{i,j}^{n}+\frac{2}{3} \overline{\bU}_{i,j}^{(2)}+\frac{2}{3} \Delta t \mathcal{L}\left(\overline{\bU}_{i,j}^{(2)}\right).
\end{cases}
\ee
where
\begin{flalign}
\mathcal{L}(\overline{\bU}_{i,j}):=&-\frac{1}{\Delta x}\left(\mathbf{\Phi}_{i+\ha,j}-\mathbf{\Phi}_{i-\ha,j}\right)-\frac{1}{\Delta y}\left(\mathbf{\Psi}_{i,j+\ha}-\mathbf{\Psi}_{i,j-\ha}\right)\nonumber\\
&+\frac{1}{\Delta x\Delta y}\int_{I_{i,j}}\bS(H,b)dx dy.\label{LU-2d}
\end{flalign}

\begin{prop}
The fully-discrete WENOAO-CST scheme (\ref{TVD-2d}) is well-balanced, i.e., it preserves the lake-at-rest steady state $H=h+b=C,~hu=0,~hv=0$.
\end{prop}
{ {
\begin{proof}
The proof is similar to the procedure in the Proposition \ref{wb}.
Specifically, for the  lake-at-rest steady state $H=C,~hu=0,~hv=0$, we have
\bean
&\overline{H}=\overline{H}_{ij}=H= C, \,\, (\overline{hu})_{ij}=hu = 0,\,\, (\overline{hv})_{ij}=hv = 0,
\eean
and
$$\int_{I_{ij}}\bS(H,b)dxdy =
\begin{bmatrix*}[c]
0 \\
\int_{I_{ij}}g(\overline{H}-H)b_x dxdy\\
\int_{I_{ij}}g(\overline{H}-H)b_y dxdy
\end{bmatrix*}
=\begin{bmatrix*}[c]
0 \\
0 \\
0
\end{bmatrix*}.$$
From the WENO reconstruction procedures, we have, for $ \forall y\in I_j=[y_{j-\ha}, y_{j+\ha}]$,
\begin{align*}
&H^{\pm}_{i-\ha}\left(y\right)
= C,\quad
(hu)^{\pm}_{i-\ha}\left(y\right) =0,\quad
(hv)^{\pm}_{i-\ha}\left(y\right) =0,\\
&H^{\pm}_{i+\ha}\left(y\right)
= C,\quad
(hu)^{\pm}_{i+\ha}\left(y\right) =0,\quad
(hv)^{\pm}_{i+\ha}\left(y\right) =0.
\end{align*}
Moreover, from \eqref{2D_hy_re1}-\eqref{2D_hy_re2} we get
\begin{align*}
&h^{*,-}_{i-\ha}\left(y\right) = h^{*,+}_{i-\ha}\left(y\right),\qquad \qquad \qquad \qquad\quad~~ h^{*,-}_{i+\ha}\left(y\right) = h^{*,+}_{i+\ha}\left(y\right),\\
&h^{*,-}_{i-\ha}\left(y\right)u^{-}_{i-\ha}\left(y\right)=h^{*,+}_{i-\ha}\left(y\right)u^{+}_{i-\ha}\left(y\right)=0,\quad
h^{*,-}_{i+\ha}\left(y\right)u^{-}_{i+\ha}\left(y\right)=h^{*,+}_{i+\ha}\left(y\right)u^{+}_{i+\ha}\left(y\right)=0, \\
&h^{*,-}_{i-\ha}\left(y\right)v^{-}_{i-\ha}\left(y\right)=h^{*,+}_{i-\ha}\left(y\right)v^{+}_{i-\ha}\left(y\right)=0,\quad
h^{*,-}_{i+\ha}\left(y\right)v^{-}_{i+\ha}\left(y\right)=h^{*,+}_{i+\ha}\left(y\right)v^{+}_{i+\ha}\left(y\right)=0.
\end{align*}
Hence, one can obtain
\begin{align*}
&\hat{\bF}^{*,H}
\left(\bU_{i+\ha}(y)\right)
=0,  \quad
\hat{\bF}^{*,hu}\left(\bU_{i+\ha}(y)\right)
= \frac{1}{2}g C^2,\quad
\hat{\bF}^{*,hv}\left(\bU_{i+\ha}(y)\right)
=0,
\\&
\hat{\bF}^{*,H}
\left(\bU_{i-\ha}(y)\right) = 0, \quad
\hat{\bF}^{*,hu}
\left(\bU_{i-\ha}(y)\right)= \frac{1}{2}g C^2, \quad
\hat{\bF}^{*,hv}
\left(\bU_{i-\ha}(y)\right) =0.
\end{align*}
This leads to
 \begin{align*}
\hat{\bF}^* \left(\bU_{i+\ha}(y)\right)-
\hat{\bF}^* \left(\bU_{i-\ha}(y)\right) = \mathbf{0},\quad \forall y\in I_j=[y_{j-\ha}, y_{j+\ha}],
\end{align*}
and
\begin{align*}
\mathbf{\Phi}_{i+\ha,j}-\mathbf{\Phi}_{i-\ha,j} = \mathbf{0}.
\end{align*}
Analogously, from the WENO reconstruction procedures, we have
 \begin{align*}
\hat{\bG}^* \left(\bU_{j+\ha}(x)\right)-
\hat{\bG}^* \left(\bU_{j-\ha}(x)\right) =\mathbf{0},
\quad \forall x\in I_i=[x_{i-\ha}, x_{i+\ha}],
\end{align*}
and
\begin{align*}
\mathbf{\Psi}_{i,j+\ha}-\mathbf{\Psi}_{i,j-\ha} = \mathbf{0}.
\end{align*}
Therefore, $\mathcal{L}(\overline{\bU}_{ij}) = \mathbf{0}$ and the scheme (\ref{LU-2d}) is well-balanced.
\end{proof}
\\
}}

\begin{prop}\label{prop_ahpha_2D}
If $\overline{h}_{i, j}^{n}$, $h^{ \pm}_{i-\frac{1}{2}} \left(\hat{y}_j^k\right)$,
$h^{ \pm}_{i + \frac{1}{2}} \left(\hat{y}_j^k\right)$, $h_{j - \frac{1}{2}}^{ \pm}\left(\hat{x}_i^k\right)$,  $h_{j + \frac{1}{2}}^{ \pm}\left(\hat{x}_i^k\right)$,  $\xi_{i,j}^{1}$ and $\xi_{i, j}^{2}$ are all non-negative in (\ref{TVD-2d}), then  $\overline{h}_{i, j}^{n+1}$ is also non-negative under the CFL condition
\be\label{CFL_con2d}
\frac{\Delta t}{\Delta x}\alpha_1+\frac{\Delta t}{\Delta y}\alpha_2 \leqslant {w}_{1} ,
\ee
where
\begin{align*}\label{xi1}
&\xi_{i,j}^1=\frac{\overline{h}_{i,j}^{n}-{w}_{1} \sum\limits_{k=1}^3\hat{\gamma}_{k}h_{i-\ha}^{+}\left(\hat{y}_j^k\right)-{w}_{4}\sum\limits_{k=1}^3\hat{\gamma}_{k} h_{i+\ha}^{-}\left(\hat{y}_j^k\right)}{1-{w}_{1}-{w}_{4}},
\\&
\xi_{i,j}^2=\frac{\overline{h}_{i,j}^{n}-{w}_{1} \sum\limits_{k=1}^3\hat{\gamma}_{k}h_{j-\ha}^{+}\left(\hat{x}_i^k\right)-{w}_{4}\sum\limits_{k=1}^3\hat{\gamma}_{k} h_{j+\ha}^{-}\left(\hat{x}_i^k\right)}{1-{w}_{1}-{w}_{4}}.
\end{align*}
\end{prop}
{ {
\begin{proof}
The proof of this result is quite similar to that given earlier for the 1D case, and is therefore omitted.
\end{proof}
}}

\vspace{10pt}

Moreover, we modify the water height $h$ on each cell $I_{i,j}$ with PP limiter as
\begin{align*}
&\tilde{h}_{i-\ha}^{+} \left(\hat{y}_j^k\right)=\theta_1 \left(h_{i-\ha}^{+} \left(\hat{y}_j^k\right)
-\overline{h}_{i,j}^n\right)+\overline{h}_{i,j}^n, \nonumber\\
&\tilde{h}_{i+\ha}^{-} \left(\hat{y}_j^k\right)=\theta_1 \left(h_{i+\ha}^{-} \left(\hat{y}_j^k\right)
-\overline{h}_{i,j}^n\right)+\overline{h}_{i,j}^n, \nonumber\\
&\tilde{h}_{j-\ha}^{+} \left(\hat{x}_i^k\right)=\theta_2 \left(h_{j-\ha}^{+} \left(\hat{x}_i^k\right)
-\overline{h}_{i,j}^n\right)+\overline{h}_{i,j}^n,\nonumber\\
&\tilde{h}_{j+\ha}^{-} \left(\hat{x}_i^k\right)=\theta_2 \left(h_{j+\ha}^{-} \left(\hat{x}_i^k\right)
-\overline{h}_{i,j}^n\right)+\overline{h}_{i,j}^n,
\end{align*}
with
\begin{align*}
\theta_1&=\min\left(1,\frac{\overline{h}_{i,j}^n-\eta}{\overline{h}_{i,j}^n-m_{i,j}^1 }\right),
\quad m_{i,j}^1=\min_{k}\min\left(h_{i-\ha}^{+}\left(\hat{y}_j^k\right),h_{i+\ha}^{-}\left(\hat{y}_j^k\right),\xi_{i,j}^1\right),
\\\theta_2&=\min\left(1,\frac{\overline{h}_{ij}^n-\eta}{\overline{h}_{i,j}^n-m_{i,j}^2 }\right),\quad m_{i,j}^2=\min_{k}\min\left(h_{j-\ha}^{+}\left(\hat{x}_i^k\right),h_{j+\ha}^{-}\left(\hat{x}_i^k\right),\xi_{i,j}^2\right),
\\ \eta&=\min\limits_{i,j}\left\{10^{-13},\overline{h}_{i,j}\right\}.
\end{align*}

Note that this PP limiter does not maintain the well-balanced property in general,
and we apply
\bean
&\tilde{b}_{i-\ha}^{+} \left(\hat{y}_j^k\right)=H_{i-\ha}^{+} \left(\hat{y}_j^k\right)-\tilde{h}_{i-\ha}^{+} \left(\hat{y}_j^k\right),
\quad \tilde{b}_{i+\ha}^{-} \left(\hat{y}_j^k\right)=H_{i+\ha}^{-} \left(\hat{y}_j^k\right)-
\tilde{h}_{i+\ha}^{-} \left(\hat{y}_j^k\right),
\\
&\tilde{b}_{j-\ha}^{+} \left(\hat{x}_i^k \right)=H_{j-\ha}^{+} \left(\hat{x}_i^k\right)-
\tilde{h}_{j-\ha}^{+} \left(\hat{x}_i^k\right),
\quad \tilde{b}_{j+\ha}^{-} \left(\hat{x}_i^k \right)=H_{j+\ha}^{-} \left(\hat{x}_i^k\right)-\tilde{h}_{j+\ha}^{-} \left(\hat{x}_i^k\right),
\eean
to get the new bottom approximation on the boundaries \cite{zhang2021high}.
{  Obviously, for the lake-at-rest steady state, we have
\bean
&\tilde{b}_{i-\ha}^{+} \left(\hat{y}_j^k\right)+\tilde{h}_{i-\ha}^{+} \left(\hat{y}_j^k\right)=H_{i-\ha}^{+} \left(\hat{y}_j^k\right)=C,
\quad \tilde{b}_{i+\ha}^{-} \left(\hat{y}_j^k\right)+
\tilde{h}_{i+\ha}^{-} \left(\hat{y}_j^k\right)=H_{i+\ha}^{-} \left(\hat{y}_j^k\right)=C,
\\
&\tilde{b}_{j-\ha}^{+} \left(\hat{x}_i^k \right)+
\tilde{h}_{j-\ha}^{+} \left(\hat{x}_i^k\right)=H_{j-\ha}^{+} \left(\hat{x}_i^k\right)=C,
\quad \tilde{b}_{j+\ha}^{-} \left(\hat{x}_i^k \right)+\tilde{h}_{j+\ha}^{-} \left(\hat{x}_i^k\right)=H_{j+\ha}^{-} \left(\hat{x}_i^k\right)=C.
\eean
Thus, the above positivity-preserving procedure can maintain the well-balanced property of the WENOAO-CST method.
}

\section{Numerical examples}
\label{num_sec}

In this section, numerous numerical examples for the 1D and 2D SWEs are presented to demonstrate the well-balanced and positivity-preserving properties of the fifth-order finite volume WENO-AO scheme based on the CST pre-balanced form. For comparisons, we mainly consider the following two methods:
\begin{itemize}
    \item {
    The proposed well-balanced WENOAO-CST method:
    uses the fifth-order finite volume WENO-AO construction \cite{BALSARA2016780} based on the CST pre-balanced form \eqref{1D_eqa-1new} in 1D and \eqref{2D_eqa-1} in 2D.}
    \item {
    The well-balanced WENOJS-XS method \cite{xing2006new,xing2011high}:
    uses the fifth-order classical finite volume WENO-JS construction \cite{jiang1996efficient} based on the original form \eqref{swe-1d} in 1D and \eqref{swe-2d} in 2D.}
\end{itemize}

The third-order SSP Runge-Kutta scheme \cite{shu1988efficient} is used for discretization in time.
The CFL number is taken to be $0.6$, unless otherwise stated.
{  Especially, the CFL number is set to be $0.6, 0.4, 0.3, 0.2, 0.1$ during the mesh refinement to ensure the spatial errors dominate in the accuracy tests such as Example \ref{accuracy test} and Example \ref{accuracy test-2D}.}
The gravitational constant is taken as $g=9.812$ in our computation.
We take the numerical solution obtained by the fifth-order finite volume WENOJS-XS method, with $N_x = 3000$ in the 1D case and $N_x = N_y = 1000$ in the 2D case, as the reference solution for comparisons, unless otherwise indicated.

To show the performance of the WENO-AO reconstruction, we also replace the WENO-AO reconstruction by WENO-ZQ \cite{zhu2016new} and WENO-MR \cite{zhu2018new} in the WENOAO-CST method (denoted as WENOZQ-CST and WENOMR-CST).
{ By comparing the results of the four methods}, we have found that the well-balanced WENOAO-CST method outperforms the other methods in terms of efficiency and resolution.
To be specific, the WENOAO-CST method is more efficient than the WENOJS-XS method in the sense that the former leads to a smaller error than the latter for a fixed amount of CPU time (cf. Example \ref{accuracy test} and Example \ref{accuracy test-2D}).
Although
the WENO-AO reconstruction needs to calculate a fourth-order polynomial for a big stencil which adds extra computational cost to the algorithm.
In addition, the WENOAO-CST method has better resolution than the WENOZQ-CST and WENOMR-CST methods, especially for small perturbation tests (cf. Fig.~\ref{figex3002} and  Fig.~\ref{figex3003} of Example \ref{small perturbation} in \S~\ref{num_sec}).

\vspace{8pt}

\begin{example}
{ (The exact C-property and positivity-preserving tests in 1D)}
\label{Well_balance}
\end{example}
This example is used to show the well-balanced and positivity-preserving properties of the proposed WENOAO-CST method.
The initial conditions are set to be the steady state solutions
\be
H(x, 0)=h(x, 0)+b(x)=10, \quad (h u)(x, 0)=0,\nonumber
\ee
on the computational domain $[0,10]$.
We use three different bottom topographies as follows
\begin{align}
&b(x)=5 e^{-\frac{2}{5}(x-5)^{2}},  \label{smooth}\\
&b(x)=
\left\{\begin{array}{ll}
4, &  4 \leqslant x \leqslant 8, \label{nonsmooth}\\
0, & \text {else},
\end{array}\right. \\
&b(x)=10e^{-\frac{2}{5}(x-5)^{2}}. \label{smooth1}
\end{align}
The first bottom is smooth, the second is discontinuous, and the last is a smooth bottom with a dry region.
The final time is $T=0.5$. The numerical errors ($N_x=200$)
of the water height $h$ and discharge $hu$
are shown in Table~\ref{tab:ex1:001} with three different bottom topographies (\ref{smooth})-\eqref{smooth1}. We can observe the WENOAO-CST method can achieve the round-off errors. Therefore the WENOAO-CST method is well-balanced.

It is worth pointing out that the WENOAO-CST method can maintain the positivity of $h$ for the bottom \eqref{smooth1} without PP limiter, while the WENOJS-XS method \cite{xing2006new} fails without PP limiter.
This result shows that the WENOAO-CST method is more robust than the WENOJS-XS method for this example.

\begin{table}[H]
\caption{$L^1$ and $L^\infty$ errors of the WENOAO-CST method for the 1D exact C-property tests. $N_x=200$ and $T=0.5$. }
\vspace{8pt}
\centering
\label{tab:ex1:001}
\begin{tabular}{ c  c    c c  c}
\hline
\multirow{2}{*}{bottom}&\multicolumn{2}{c}{$L^1$ error}&\multicolumn{2}{c}{$L^\infty$ error}\\
&\multicolumn{2}{c}{$h$~~~~~~~~~~$hu$}&\multicolumn{2}{c}{$h$~~~~~~~~~~$hu$}\\\hline
(\ref{smooth})& \multicolumn{2}{c}{4.07e-14 ~~~~~1.04e-13} &  \multicolumn{2}{c}{6.57e-14 ~~~~~3.85e-13} \\
(\ref{nonsmooth})& \multicolumn{2}{c}{3.97e-14  ~~~~~7.92e-14} &  \multicolumn{2}{c}{6.22e-14 ~~~~~2.54e-13} \\
(\ref{smooth1})& \multicolumn{2}{c}{4.47e-14  ~~~~~8.25e-14} &  \multicolumn{2}{c}{6.75e-14 ~~~~~2.33e-13} \\
 \hline
\end{tabular}
\end{table}

\begin{example}
(The accuracy test in 1D)\label{accuracy test}
\end{example}
The bottom topography and initial conditions are
$$
b(x)=\sin ^{2}(\pi x), \quad h(x, 0)=5+e^{\cos (2 \pi x)}, \quad(h u)(x, 0)=\sin (\cos (2 \pi x)),
$$
on the computational domain [0,1] with a periodic boundary condition.
The final time is $T = 0.1$ when the solutions remain smooth.
Since the exact solutions are not available, we adopt the WENOJS-XS method with $N_x=12800$ to compute reference solutions and treat them as the exact solutions to compute the numerical errors.
The $L^1$  and $L^\infty$ errors and orders for $h$ and $hu$ by the WENOAO-CST method are plotted in Table~\ref{tab:ex2:001}.
One can see that our method achieves the designed fifth-order accuracy.

In Fig.~\ref{figex2-1001}, we also exhibit the numerical errors versus the CPU time for $h$ and $hu$ by the WENOAO-CST and WENOJS-XS methods.
We can observe that the WENOAO-CST method outperforms the WENOJS-XS method since the former can obtain a smaller error for a fixed CPU time.

\begin{table}[H]
\caption{$L^1$ and $L^\infty$ errors and orders of the WENOAO-CST method for the 1D accuracy test.  $T= 0.1$.}
\vspace{8pt}
\centering
\label{tab:ex2:001}
\begin{tabular}{ c c c  c c c c c c c }
\toprule
\multirow{2}{*}{$N_x$}&\multirow{2}{*}{CFL}&$h$&&$hu$&&$h$&&$hu$\\

& &$L^1$ error& order& $L^1$ error& order&$L^\infty$ error& order& $L^\infty$ error& order\\
 \midrule
50	&0.6& 1.70e-03	&	    &1.80e-02	&    & 1.71e-02 &       &1.68e-01	&\\
100	&0.4& 2.43e-04	&2.80	&2.06e-03	&3.12& 3.67e-03	&2.22	&3.21e-02	&2.39\\
200	&0.3& 1.54e-05	&3.98	&1.31e-04	&3.97& 3.82e-04	&3.26	&3.27e-03	&3.30\\
400	&0.2& 5.95e-07	&4.70	&5.08e-06	&4.69& 1.99e-05	&4.26	&1.69e-04	&4.28\\
800	&0.1& 1.84e-08	&5.02	&1.57e-07	&5.01& 6.53e-07	&4.93	&5.53e-06	&4.93\\
\bottomrule
\end{tabular}
\end{table}

\begin{figure}[H]
\centering
\subfigure[$h$]{
\includegraphics[width=0.45\textwidth,trim=40 0 40 10,clip]{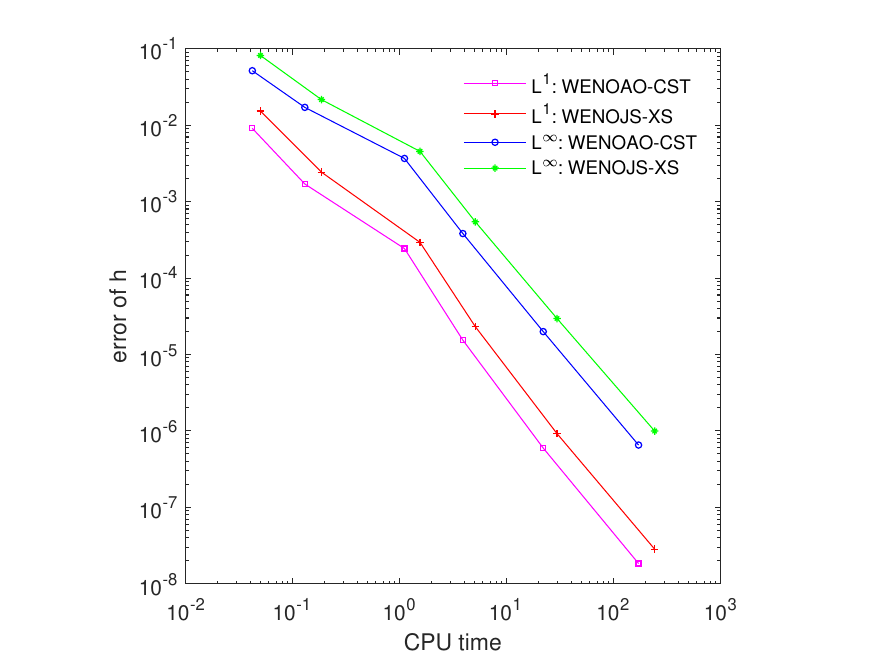}}
\subfigure[$hu$]{
\includegraphics[width=0.45\textwidth,trim=40 0 40 10,clip]{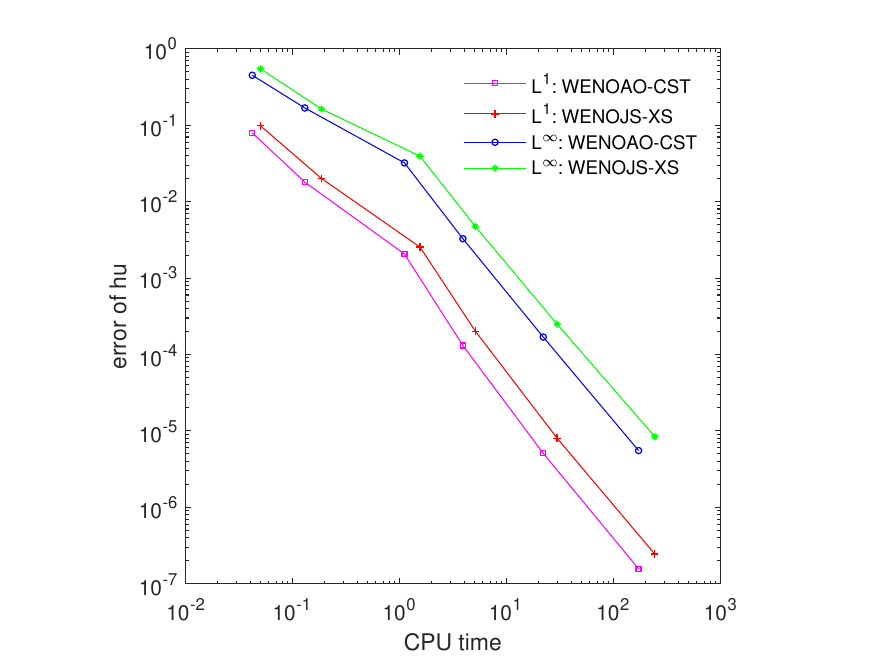}}
\caption{The $L^1$ and $L^\infty$ errors vs. CPU time of the water height $h$ and discharge $hu$ by the WENOAO-CST and WENOJS-XS methods.  $T= 0.1$.}
\label{figex2-1001}
\end{figure}

\begin{example}
(The tidal wave flow in 1D)\label{Tidal}
\end{example}
This example is taken from \cite{bermudez1994upwind}.
The initial conditions and the bottom topography are
\begin{align*}
&H(x, 0) = h(x,0) + b(x) \equiv 60.5, \quad(h u)(x, 0) \equiv 0,
\\&
b(x)=10+\frac{40 x}{L}+10 \sin \left(\frac{4 \pi x}{L}-\frac{\pi}{2}\right),
\end{align*}
on the computational domain $0 \leqslant x \leqslant L$ with  $L=14,000$.
The boundary conditions are taken as
\begin{equation*}
H(0, t)=64.5-4 \sin \left(\frac{4 \pi t}{86400}+\frac{\pi}{2}\right), \quad(h u)(L, t)=0.
\end{equation*}

This is a useful test since the exact solutions have a highly accurate asymptotic approximation
\begin{align*}\label{tidal}
H(x, t)=64.5-4 \sin \left(\frac{4 \pi t}{86400}+\frac{\pi}{2}\right), \,\, (h u)(x, t)=\frac{\pi(x-L)}{5400} \cos \left(\frac{4 \pi t}{86400}+\frac{\pi}{2}\right).
\end{align*}
The solutions are computed up to $T=7,552.13$ with $N_x = 200$.
In Fig.~\ref{figex3-1001}, the numerical solutions of $h$ and $u$ by the WENOAO-CST and WENOJS-XS methods are plotted with the exact solutions.
The numerical solutions of the two methods coincide with the exact solutions quite well.

\begin{figure}[H]
\centering
\subfigure[ $h$]{
\includegraphics[width=0.45\textwidth,clip]{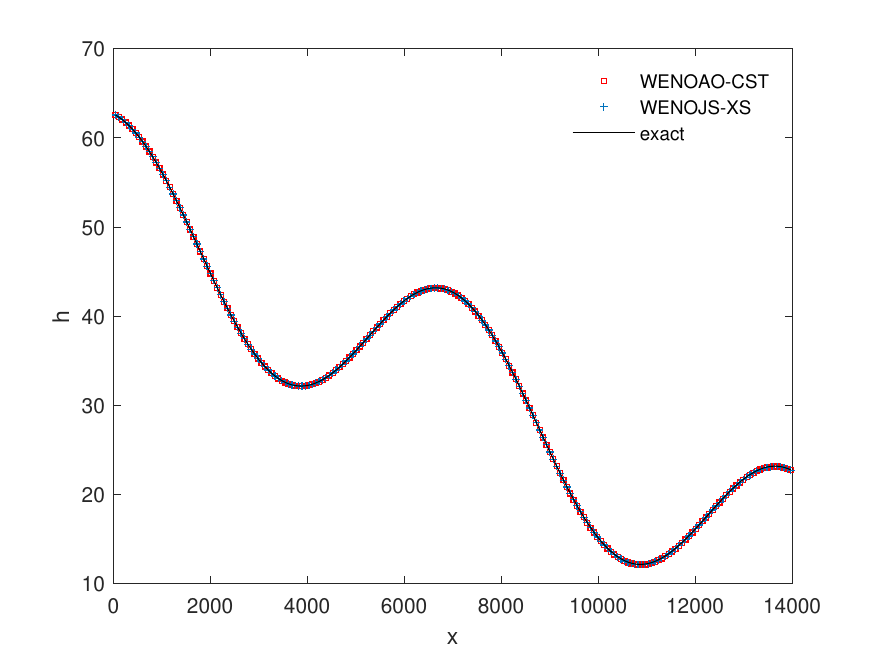}}
\subfigure[$ u$]{
\includegraphics[width=0.45\textwidth,clip]{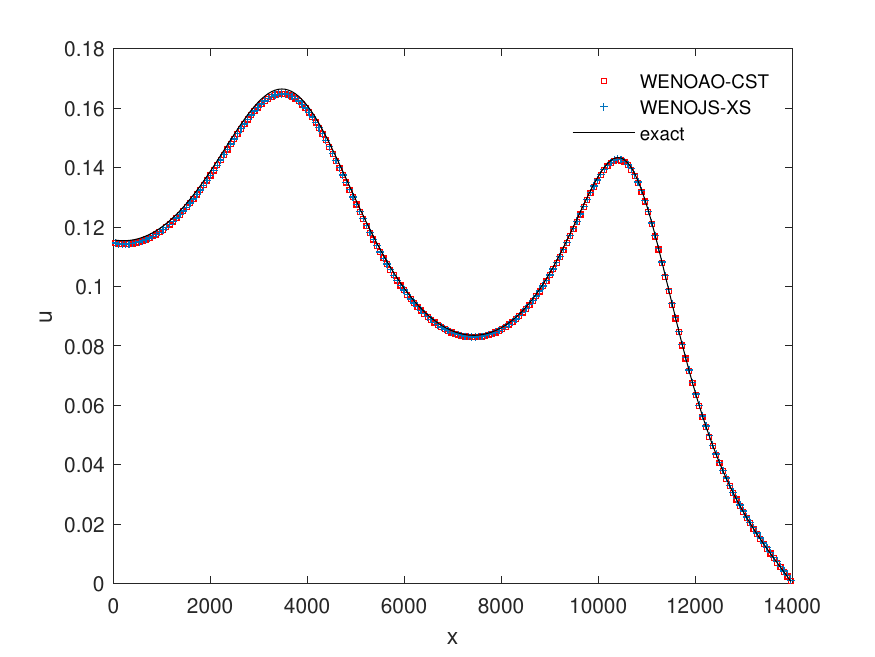}}
\caption{The water height $h$ and velocity $u$ for the tidal wave flow test. $T=7,552.13$ and $N_x=200$. Square: WENOAO-CST; plus: WENOJS-XS; solid line: exact solutions.}
\label{figex3-1001}
\end{figure}

\begin{example}
(The small perturbation of a steady-state water in 1D)
\label{small perturbation}
\end{example}
This example is used to demonstrate the capability of the WENOAO-CST  method for computing small perturbations of a steady-state flow over a non-flat bottom topography. The bottom topography with a hump is given by
\be
b(x)=\begin{cases}
0.25(\cos (10 \pi(x-1.5))+1), & 1.4 \leqslant x \leqslant 1.6,\\
0, & \text{else}.
\end{cases}
\ee
 The initial conditions on the computational domain $[0,2]$ are provided by
\be
(h u)(x, 0)=0, \,\, h(x, 0)=\begin{cases}
1-b(x)+\varepsilon, & 1.1 \leqslant x \leqslant 1.2, \\
1-b(x), & \text{else},
\end{cases}
\ee
where the perturbation constant $\varepsilon$ is non-zero, and the big pulse $\varepsilon = 0.2$ and small pulse $\varepsilon = 0.001$ are used in our computation. The transmissive boundary condition is applied on both ends.
Theoretically, the small disturbance should break into two waves moving left and right with characteristic speeds of approximately $\pm\sqrt{ gh}$, respectively.
{ {The simulation involving such a small perturbation of the sea surface is challenging for many non-well-balanced numerical methods.}}

The solutions are computed up to $T = 0.2$ with $N_x= 200$.
The numerical solutions by the WENOAO-CST, WENOZQ-CST, WENOMR-CST, and WENOJS-XS methods are plotted in Fig.~\ref{figex3002} (for big pulse $\varepsilon= 0.2$) and Fig.~\ref{figex3003} (for small pulse $\varepsilon = 0.001$).
We can observe that our WENOAO-CST method has slightly better resolution than the WENOJS-XS method around the discontinuities.
It is obvious that both WENO-ZQ and WENO-MR reconstructions generate spurious oscillations near shock, so we suggest using the WENO-AO reconstruction.

\begin{figure}[H]
\centering
\subfigure[$h + b$]{
\includegraphics[width=0.45\textwidth,clip]{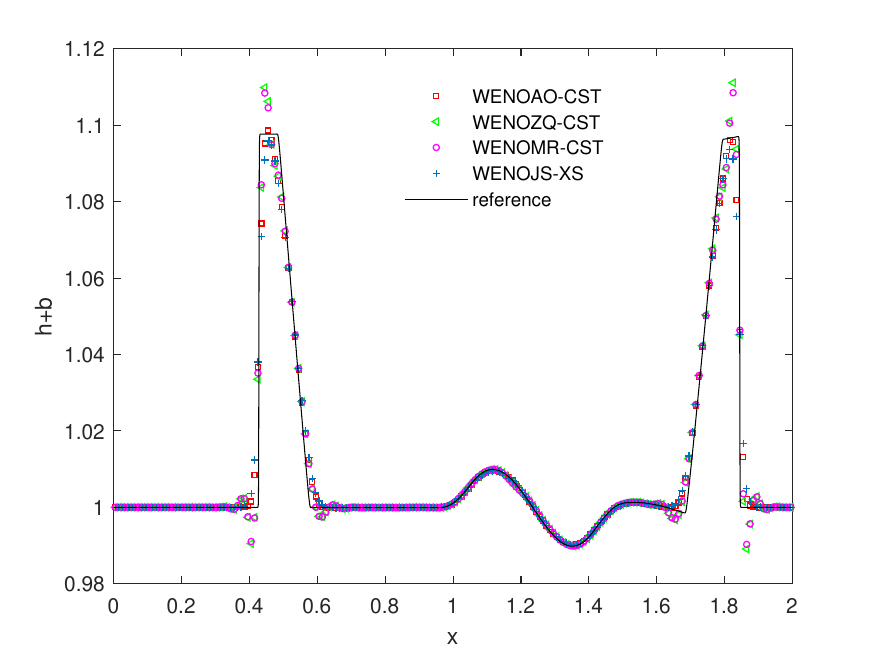}}
\subfigure[zoom-in of (a) at $x \in (0.3,0,7)$]{
\includegraphics[width=0.45\textwidth,clip]{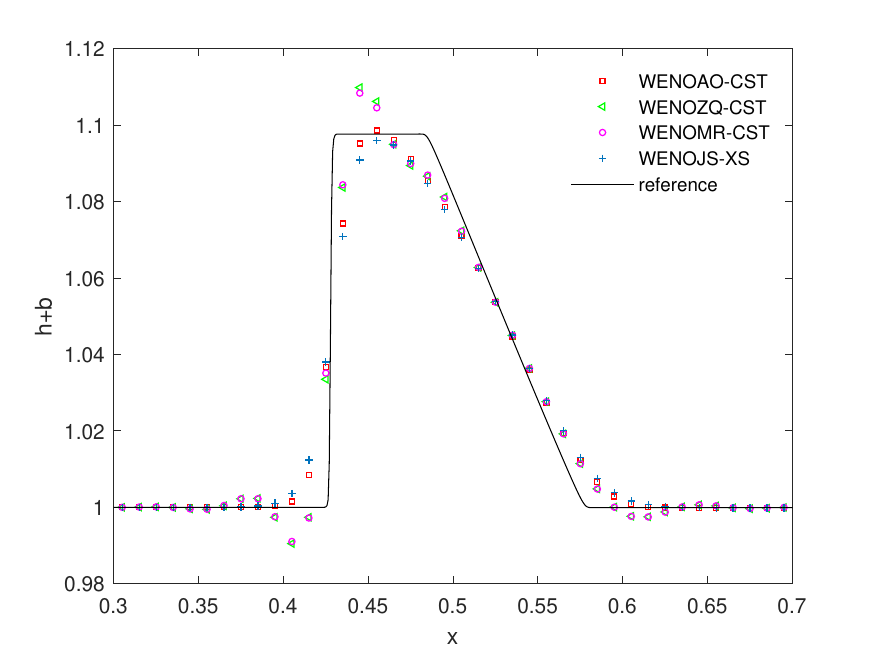}}
\subfigure[$hu$]{
\includegraphics[width=0.45\textwidth,clip]{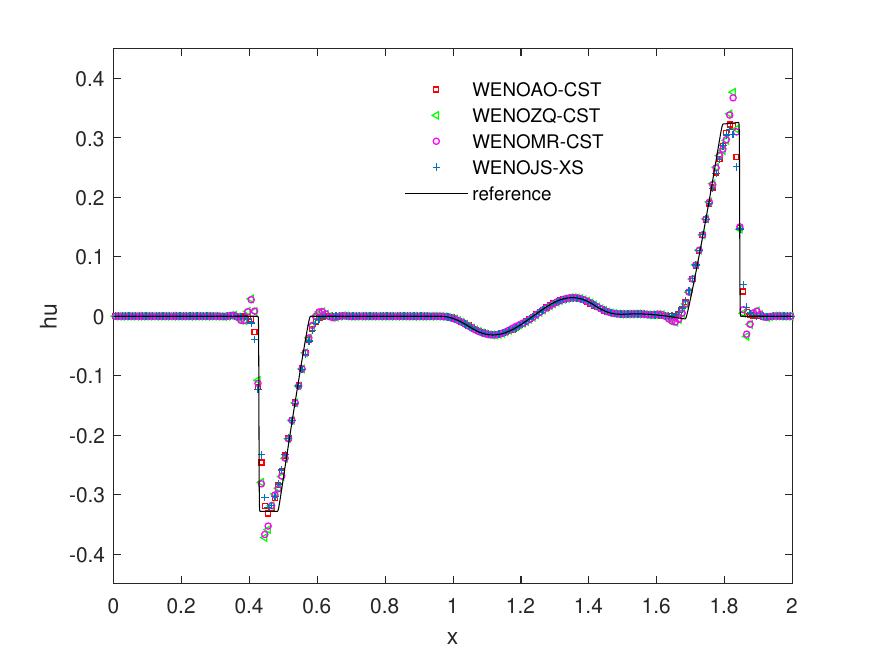}}
\subfigure[zoom-in of (c) at $x \in (0.3,0,7)$]{
\includegraphics[width=0.45\textwidth,clip]{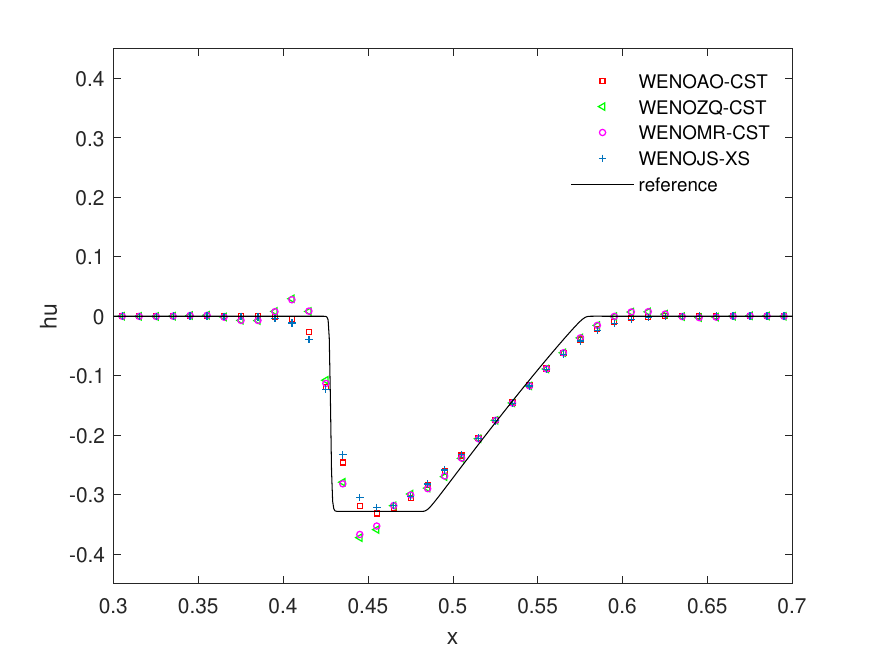}}
\caption{The surface level $h + b$ and discharge $hu$ for small perturbation of a steady-state water with a big pulse $\varepsilon=0.2$. $T=0.2$ and $N_x=200$. Square: WENOAO-CST; triangle: WENOZQ-CST; circle: WENOMR-CST; plus: WENOJS-XS;  solid line: reference solutions.}
\label{figex3002}
\end{figure}

\begin{figure}[H]
\centering
\subfigure[$h + b$]{
\includegraphics[width=0.45\textwidth,clip]{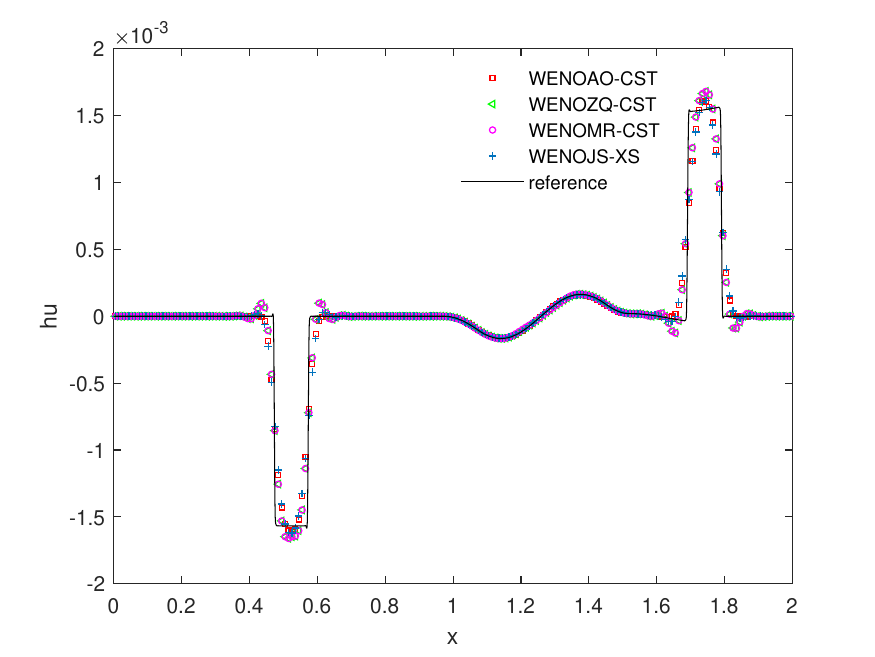}}
\subfigure[zoom-in of (a) at $x \in (0.3,0,7)$]{
\includegraphics[width=0.45\textwidth,clip]{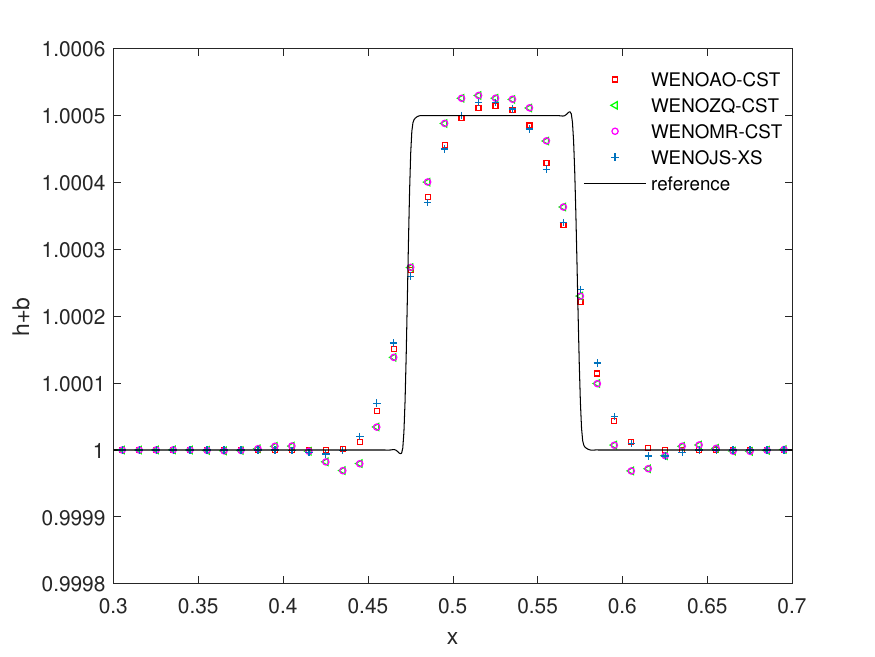}}
\subfigure[$hu$]{
\includegraphics[width=0.45\textwidth,clip]{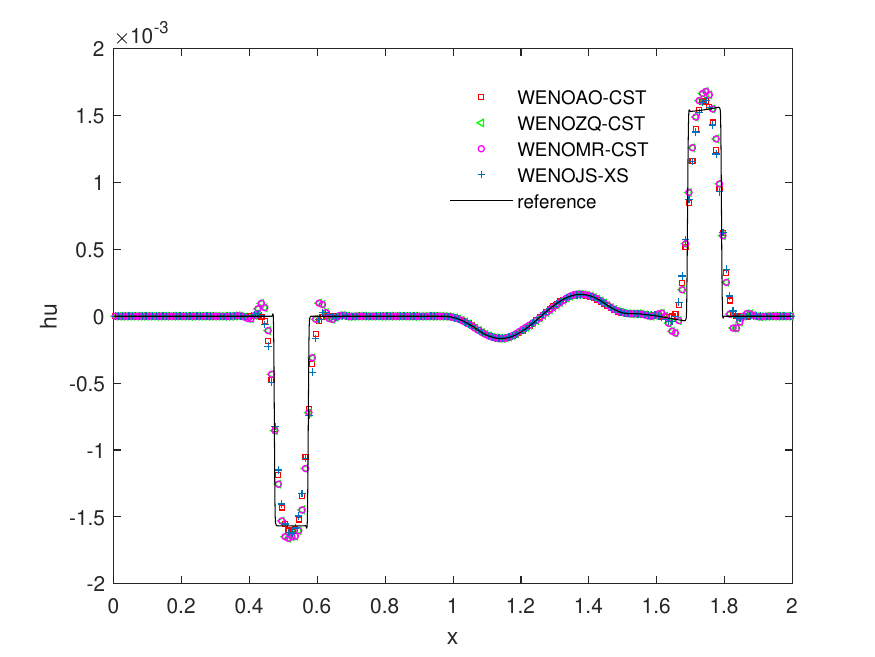}}
\subfigure[zoom-in of (c) at $x \in (0.3,0,7)$]{
\includegraphics[width=0.45\textwidth,clip]{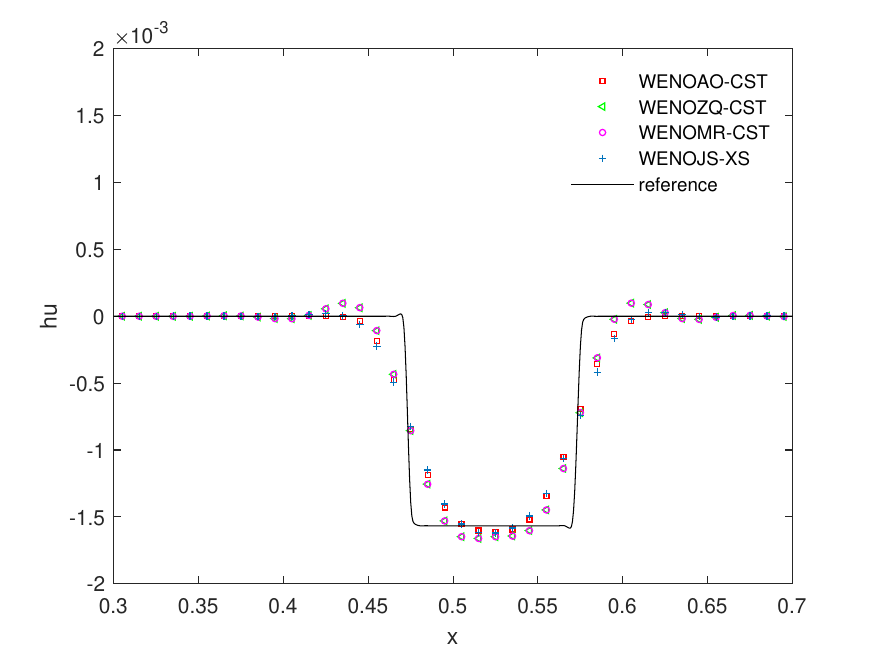}}
\caption{The surface level $h + b$ and discharge $hu$ for small perturbation of a steady-state water with a small pulse $\varepsilon=0.001$. $T=0.2$ and $N_x=200$. Square: WENOAO-CST; triangle: WENOZQ-CST; circle: WENOMR-CST; plus: WENOJS-XS; solid line: reference solutions.}
\label{figex3003}
\end{figure}

\begin{example}
(The dam breaking problem over a rectangular bump in 1D)
\label{dam breaking}
\end{example}
We solve the dam breaking problem over a rectangular bump, which produces a rapidly varying flow over a non-smooth bottom topography.  The computational domain is $[0, 1500]$, and
the bottom topography is a single rectangular hump
\be
b(x)=\left\{\begin{array}{ll}
8, & |x-750| \leqslant 1500 / 8, \\
0, & \text {else}.
\end{array}\right.\nonumber
\ee
 The initial conditions are
\be
(h u)(x, 0)=0,\,\quad  h(x, 0)=\left\{\begin{array}{ll}
20-b(x), & x \leqslant 750, \\
15-b(x), & \text {else}.
\end{array}\right.\nonumber
\ee

The solutions are calculated up to $T=15$ and $T=60$ with $N_x = 400$. At the initial time, the water height $h$ is discontinuous and the water surface level $h + b$ is smooth at the points $x = 562.5$ and $x = 937.5$.
The bottom topography $b$ and surface level $h+b$ at $T=0,~15,~60$ are shown in Figs.~\ref{figex4001} and \ref{figex4002}.
From the figures, we can see that both methods can solve this problem well and agree with the reference solutions. The results of the two methods are comparable.

\begin{figure}[H]
\centering
\subfigure[$ T=0,~15$]{
\includegraphics[width=0.45\textwidth,clip]{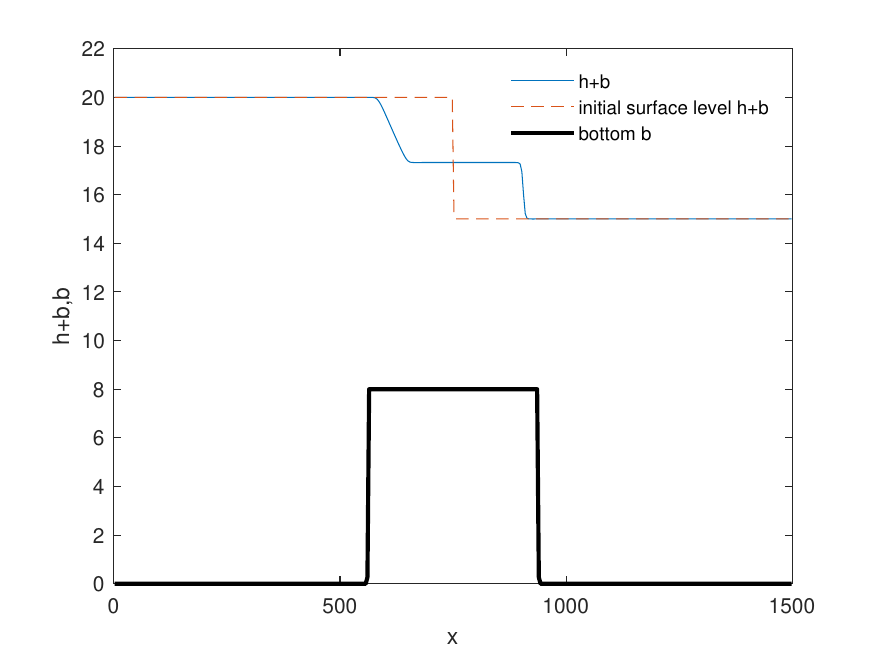}}
\subfigure[$T=0,~60$]{
\includegraphics[width=0.45\textwidth,clip]{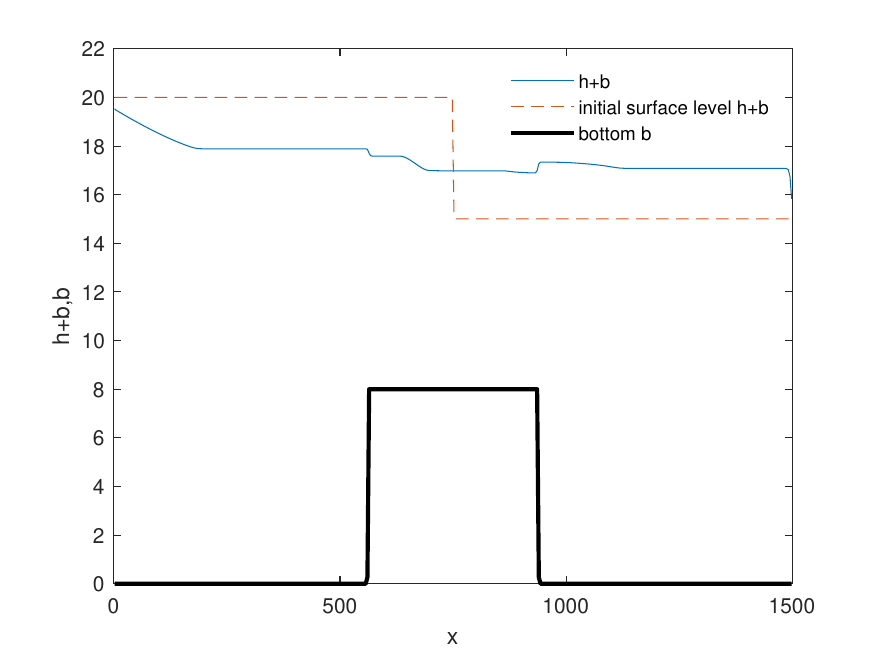}}
\caption{The bottom topography and the surface level $h + b$ at $T=0,~15,~60$ by the WENOAO-CST for the dam breaking problem. $N_x=400$. Left: $T=15$; right: $T=60$.}
\label{figex4001}
\end{figure}

\begin{figure}[H]
\centering
\subfigure[$ h+b$ at $T=15$]{
\includegraphics[width=0.45\textwidth,clip]{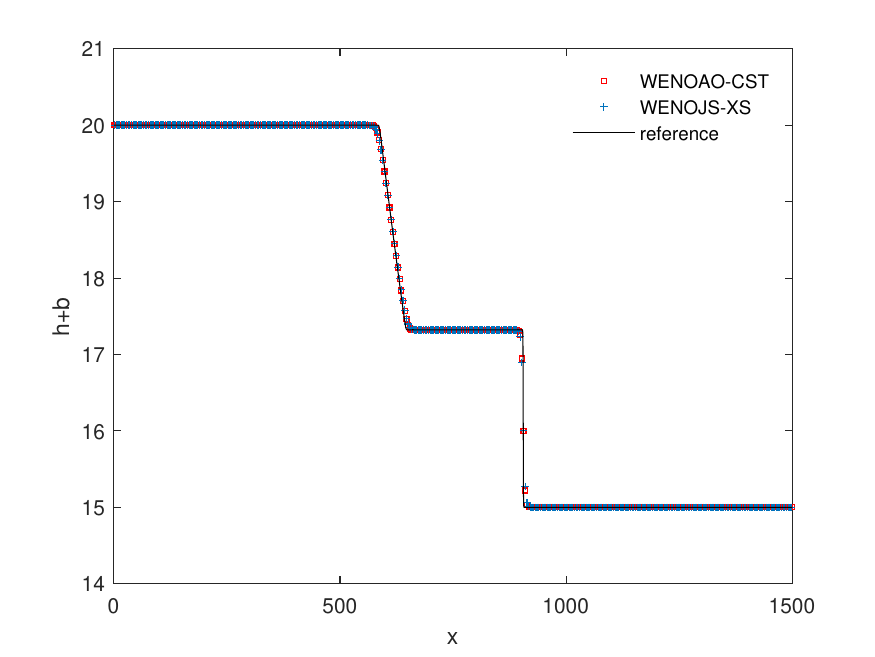}}
\subfigure[$ h+b$ at $T=60$]{
\includegraphics[width=0.45\textwidth,clip]{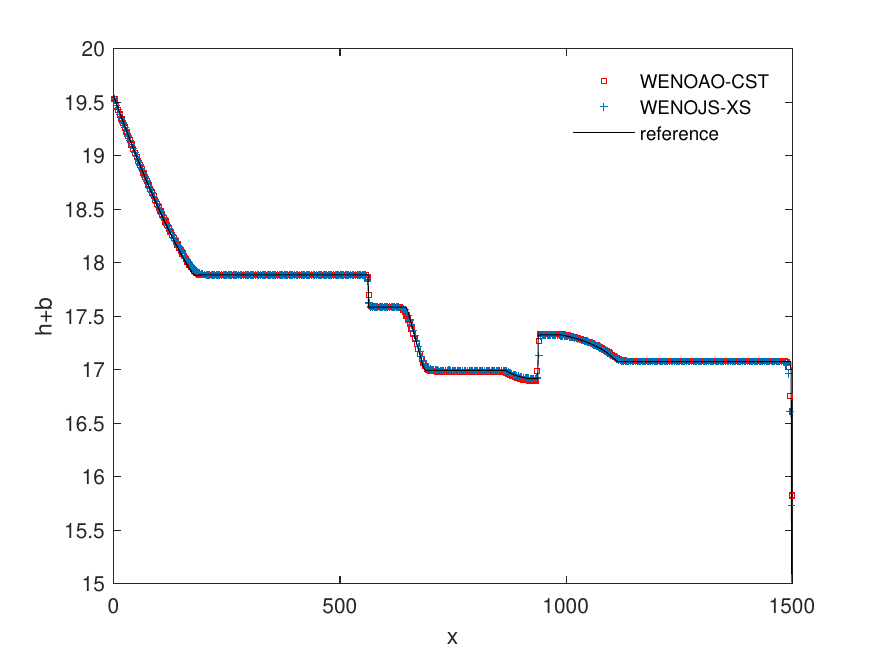}}
\caption{The surface level $h + b$ at $T=15$ and $60$ for the dam breaking problem. $N_x=400$. Square: WENOAO-CST; plus: WENOJS-XS; solid line: reference solutions.}
\label{figex4002}
\end{figure}

\begin{example}
(The steady flows over a hump in 1D)\label{Steady_flows}
\end{example}
This example is used to study the convergence in time toward steady flow over a hump.
Such characteristics of flows are determined by the free-stream Froude number, $ Fr=|u| /\sqrt{g h}$, and the bottom topography. The steady state solution will be smooth if $F r<1$ (subcritical flow) or $F r>1$ (supercritical flow) everywhere.
Otherwise, one of the eigenvalues of the Jacobian matrix, $ u \pm \sqrt{g h}$, goes through zero with transitions occurring at points where $ F r$ passes through $1$. In the transcritical flow, a steady shock may present in the steady state solution. To test the numerical methods for the SWEs, steady flows across a hump are common benchmarks for transcritical and subcritical steady flows.
The hump bottom topography is provided by
\be
b(x)=\left\{\begin{array}{ll}
0.2-0.05(x-10)^{2}, & 8 \leqslant x \leqslant 12, \nonumber\\
0, & 0 \leqslant x < 8\text{~and~} 12<x\leqslant 25.
\end{array}\right.
\ee
The initial conditions are
\be
H(x, 0) \equiv 0.5, \quad(h u)(x, 0) \equiv 0.\nonumber
\ee
The boundary conditions determine the type of flow, and the flow can be subcritical or transcritical with or without a steady shock.
The time at which all of the solutions reach their steady states is specified as $T=200$. The upstream and downstream boundary conditions of the three cases are given as follows.
\begin{itemize}
\item []{\bf Case} 1: subcritical flow
\begin{itemize}
\item upstream: $(h u)(0, t)=4.42$.
\item downstream: $H(25, t)=2$.
\end{itemize}
\item[]{\bf Case} 2: transcritical flow without steady shock
\begin{itemize}
\item upstream: $(h u)(0, t)=1.53$.
\item downstream: $H(25, t)=0.41$.
\end{itemize}
\item[]{\bf Case} 3: transcritical flow with a steady shock
\begin{itemize}
\item upstream: $(h u)(0, t)=0.18$.
\item downstream: $H(25, t)=0.33$.
\end{itemize}
\end{itemize}

The surface level $h + b$ and the discharge $hu$ by the WENOAO-CST and WENOJS-XS methods for {\bf Cases} 1, 2, and 3 are plotted in Fig \ref{figex5003}.
We observe that the numerical solutions are non-oscillatory and in good agreement with the reference solutions.

\begin{figure}[H]
\centering
\subfigure[{\bf Case} 1: $h+b,b$]{
\includegraphics[width=0.45\textwidth,clip]{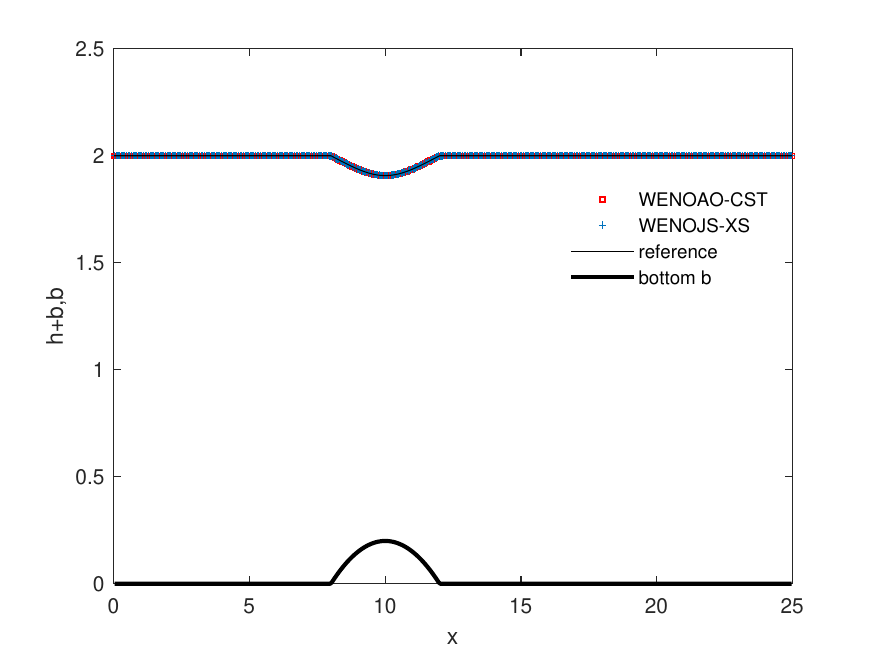}}
\subfigure[{\bf Case} 1: $hu$]{
\includegraphics[width=0.45\textwidth,clip]{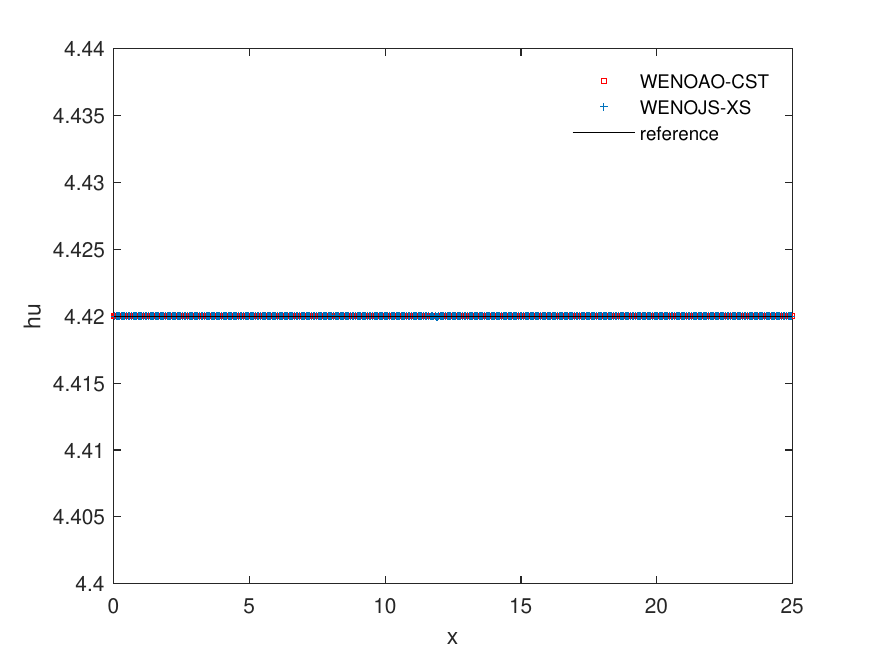}}
\subfigure[{\bf Case} 2: $h+b,~b$]{
\includegraphics[width=0.45\textwidth,clip]{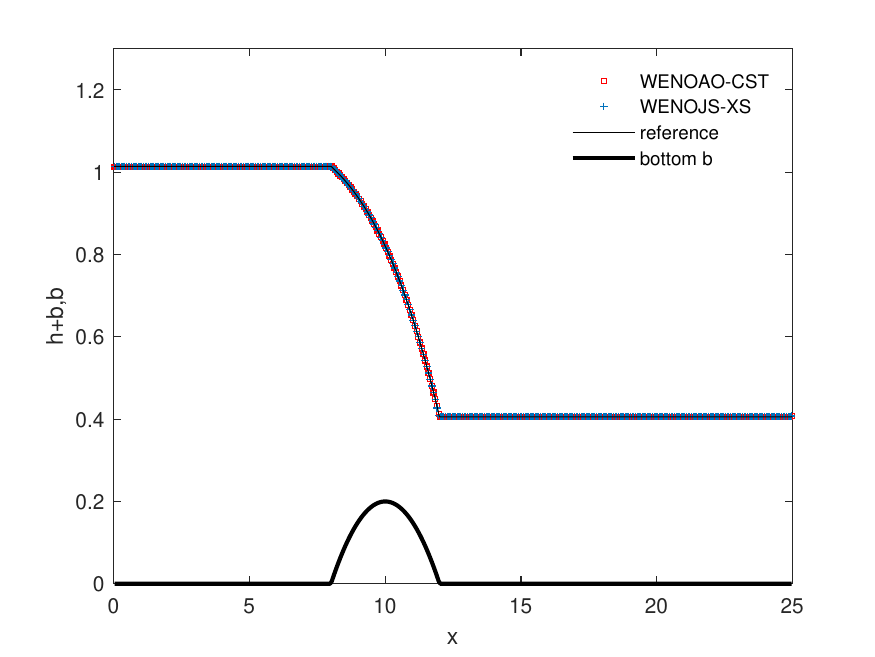}}
\subfigure[{\bf Case} 2: $hu$]{
\includegraphics[width=0.45\textwidth,clip]{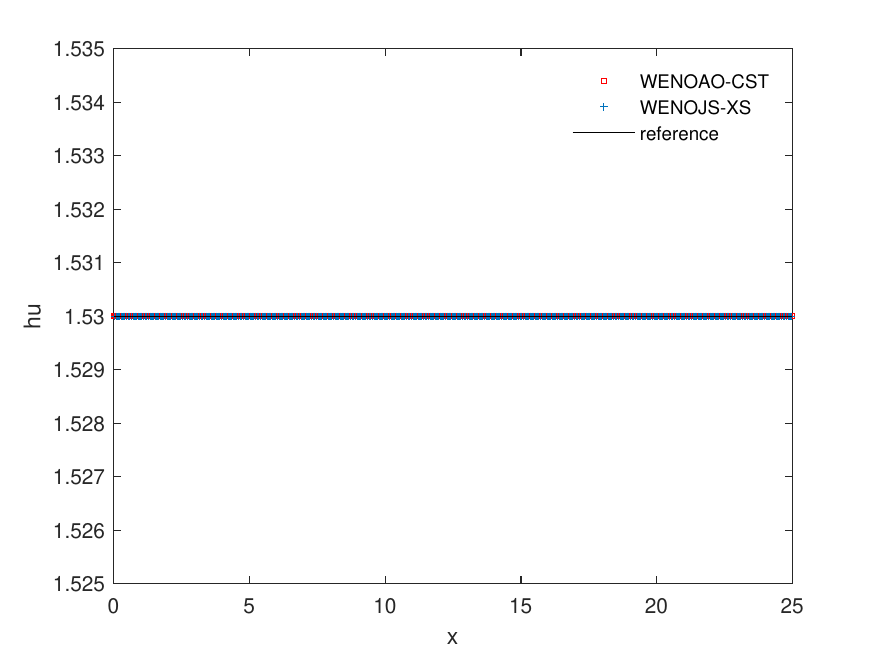}}
\subfigure[{\bf Case} 3: $h+b, ~b$]{
\includegraphics[width=0.485\textwidth,clip]{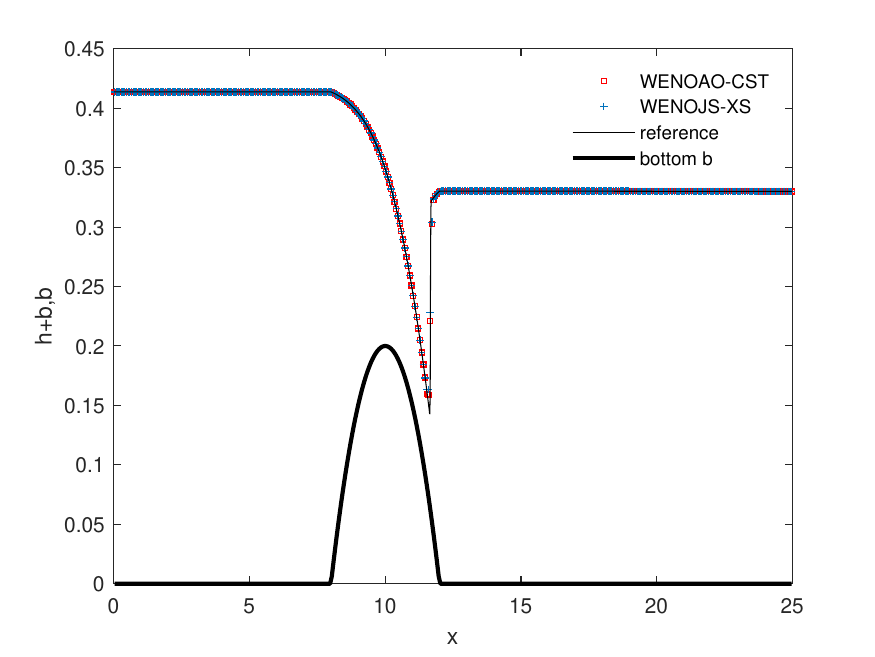}}
\subfigure[{\bf Case} 3: $hu$]{
\includegraphics[width=0.45\textwidth,clip]{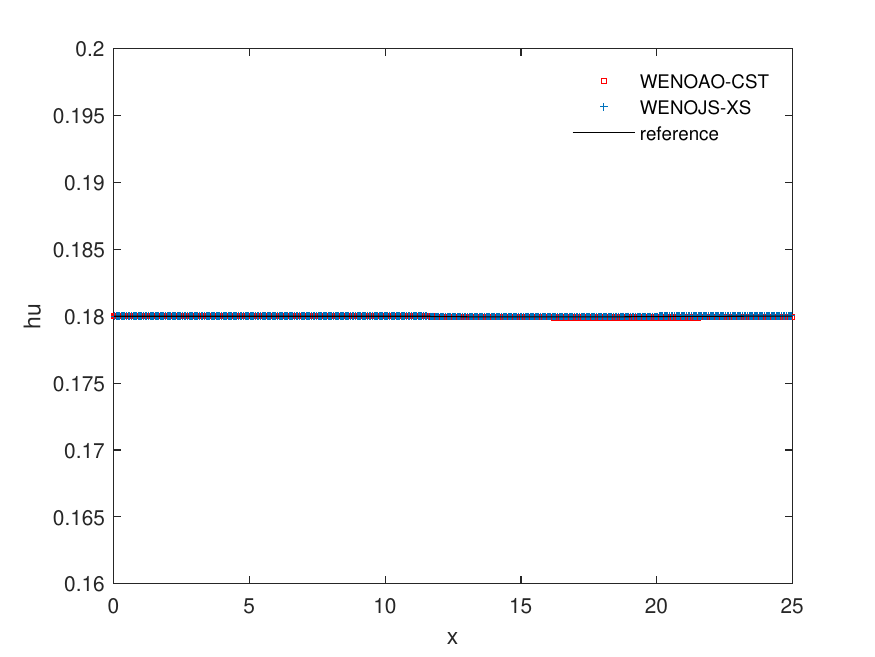}}
\caption{The surface level $h + b$ and discharge $hu$ for {\bf Cases} 1,~2 and 3. $T=200$ and $N_x=400$. Square: WENOAO-CST; plus: WENOJS-XS; solid line: reference solutions.}
\label{figex5003}
\end{figure}

\begin{example}
{ (The positivity-preserving tests over the flat and rectangle bottoms in 1D)}
\label{Riemann_problem}
\end{example}
{ We solve the 1D SWEs with a dry area over a flat bottom and a step bottom, respectively, to demonstrate the positivity-preserving property of our WENOAO-CST method. Specifically, we consider the bottom $b(x) \equiv 0$ in the first and second tests and a rectangle bottom in the third test.
These tests have been studied in \cite{bokhove2005flooding,bunya2009wetting,GALLOUET2003479,xing2011high,xing2010positivity}.
}

In the first test, the initial conditions are
\be \label{remain1}
hu(x, 0)=0, \quad h(x, 0)=\begin{cases}
10, & x \in [-300,~0], \\
0, & x \in (0,~300],
\end{cases}
\ee
with a transmissive boundary condition.
The left side is still water with a surface level of $10$, and the right is a dry region. The analytic solution for this problem can be found in \cite{bokhove2005flooding}.
The solutions are computed up to $T=4,~8,~12$ with $N_x=250$. The numerical solutions by the WENOAO-CST and WENOJS-XS methods (with PP limiter) are shown in Fig.~\ref{figex7-1001}.
No negative water height is generated during the computation, and a good agreement between the numerical and exact solutions is observed.

\vspace{8pt}

In the second test, the initial conditions with a nonzero velocity are
\be \label{remain2}
hu(x, 0)=\begin{cases}
0, &  x \in [-200,~0], \\
400, & x \in (0,~400],
\end{cases}
\quad
h(x, 0)=\begin{cases}
5, &  x \in [-200,~0],\\
10, & x \in (0,~400],
\end{cases}
\ee
with a transmissive boundary condition.
In the beginning, the solution does not have any dry areas. However, a dry region appears when the constant initial conditions meet the drying requirement $\sqrt{g h_{l}}+\sqrt{g h_{r}}+u_{l}-u_{r}<0$, and this makes the problem computationally challenging. After that, two expansion waves propagate away from each other. The analytical solution for this problem can be found in \cite{bokhove2005flooding}.
The solutions are computed up to $T=2,~4,~6$ with $N_x=250$. In Fig.~\ref{figex7-1003}, we plot the numerical solutions by WENOAO-CST and WENOJS-XS methods with the positivity-preserving limiter.
From these figures, we observe that the exact solutions are well captured by both two methods, and the water height is non-negative based on the positivity-preserving limiter.
The zoomed-in version near the dry region at these times is also shown in Fig.~\ref{figex7-1003}. We can see that the results of the WENOAO-CST method have slightly better resolution than that of the WENOJS-XS method around the dry region.

\vspace{8pt}

{
In the last test of this example, the initial conditions are
\begin{equation}\label{remain3}
\begin{split}
&
H(x,0) = 10,\quad b(x)=\begin{cases}
1, &  x\in [\frac{25}{3},~\frac{25}{2}],  \\
0, & x \in [0,~\frac{25}{3}]\cup[\frac{25}{2},~25],
\end{cases}\\
&h u(x, 0)=\begin{cases}
-350, &  x \in [0,~\frac{50}{3}], \\
350, &  x \in (\frac{50}{3},~25],
\end{cases}
\end{split}
\end{equation}
with a transmissive boundary condition.
The numerical results at different times  $T=0$, $0.25$, and $0.65$ with $N_x=250$ are shown in Fig.~\ref{figex8-1001} for the water surface and the discharge, respectively.  The water flows out of the domain and a dry region is developed. The results reflect this pattern well and agree with
those obtained in \cite{GALLOUET2003479,xing2011high}.
}

\begin{figure}[H]
\centering
\subfigure[$h$]{
\includegraphics[width=0.45\textwidth,clip]{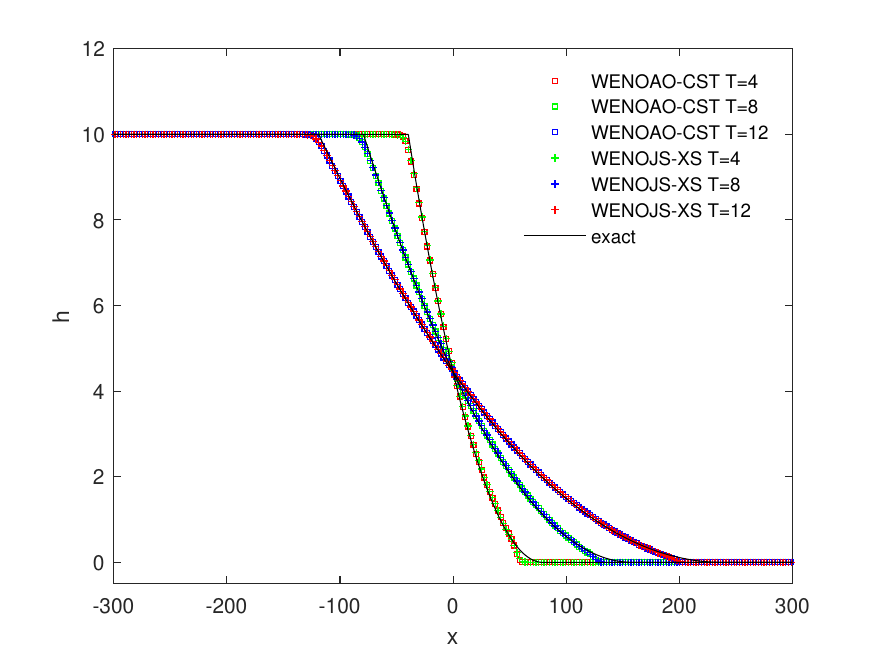}}
\subfigure[$ hu$]{
\includegraphics[width=0.45\textwidth,clip]{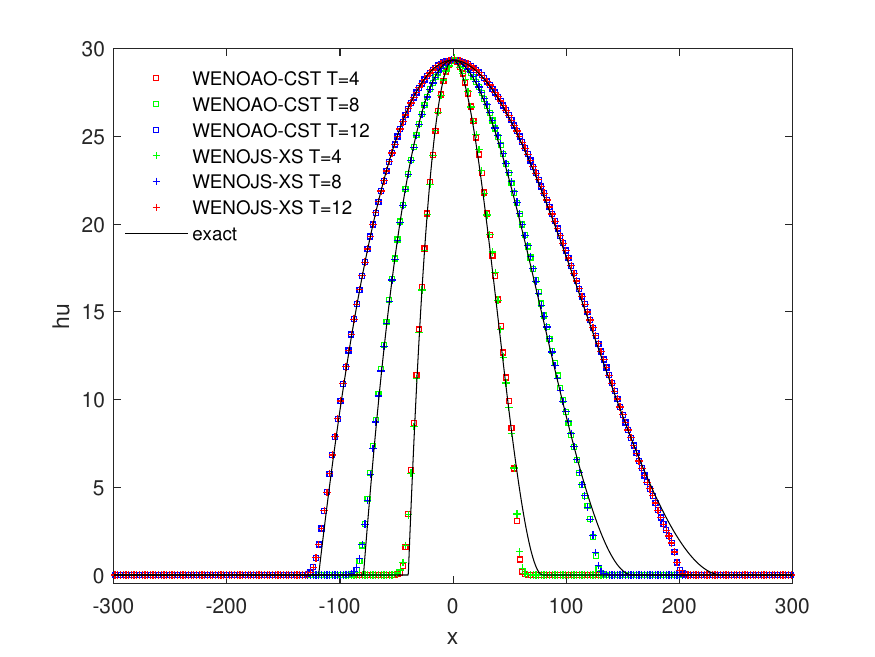}}
\caption{The water height $h$ and discharge $hu$ for the initial data \eqref{remain1}. $T=4,~8,~12$ and $N_x=250$. Square: WENOAO-CST; plus: WENOJS-XS; solid line: exact solutions. }
\label{figex7-1001}
\end{figure}

\begin{figure}[H]
\centering
\subfigure[$h$]{
\includegraphics[width=0.45\textwidth,clip]{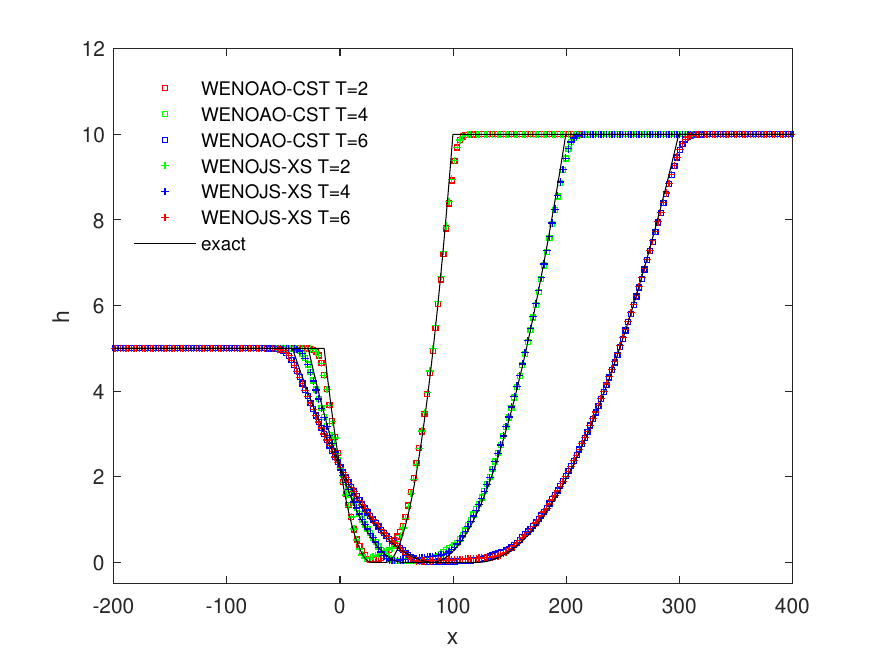}}
\subfigure[zoom-in of (a) around dry region]
{\includegraphics[width=0.45\textwidth,clip]{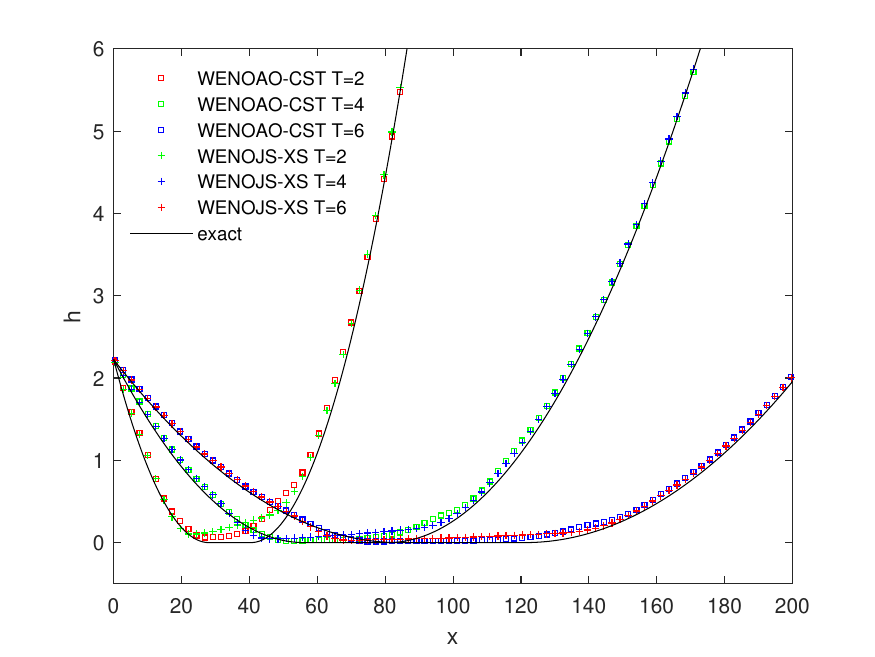}\label{c1}}
\subfigure[$ hu$]{
\includegraphics[width=0.45\textwidth,clip]{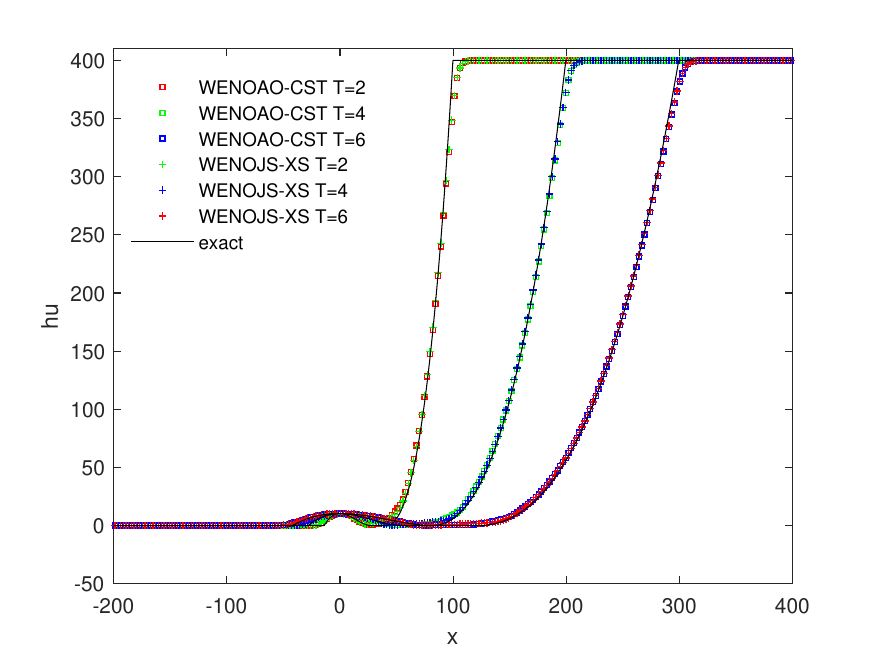}}
 \subfigure[zoom-in of (c) around dry region]
 {\includegraphics[width=0.45\textwidth,clip]{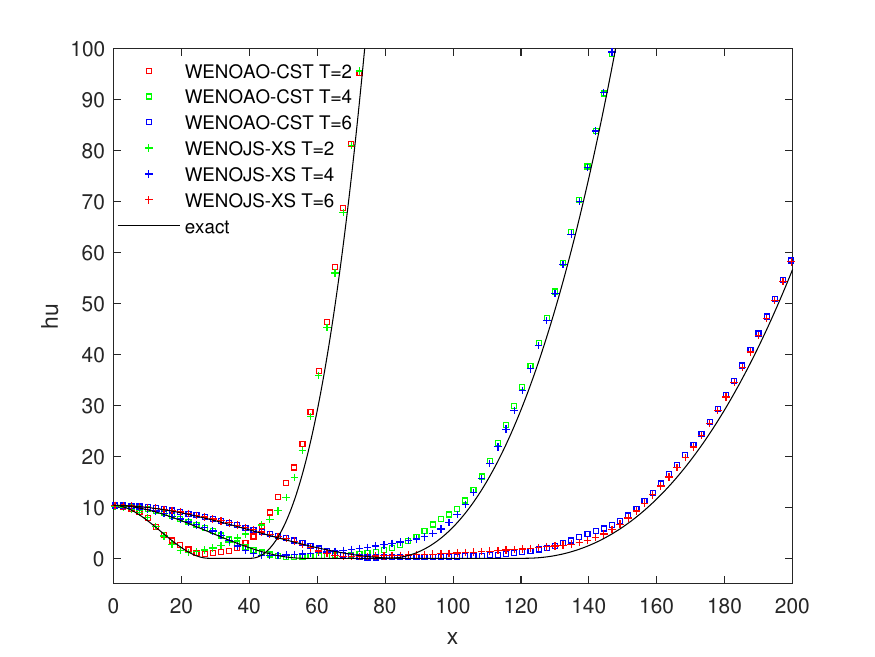}\label{d1}}
\caption{The water height $h$ and discharge $hu$ for the initial data \eqref{remain2}. $T=2,~4,~6$ and $N_x=250$. Square: WENOAO-CST; plus: WENOJS-XS; solid line: exact solutions.}
\label{figex7-1003}
\end{figure}

\begin{figure}[H]
\centering
\subfigure[$h+b,~b$]{
\includegraphics[width=0.45\textwidth,clip]{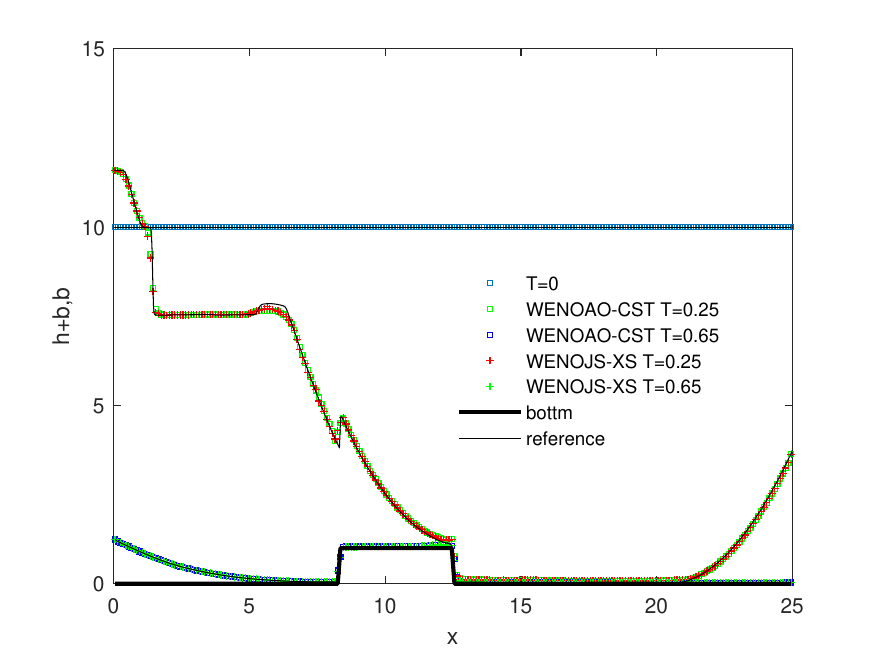}}
\subfigure[$ hu$]{
\includegraphics[width=0.45\textwidth,clip]{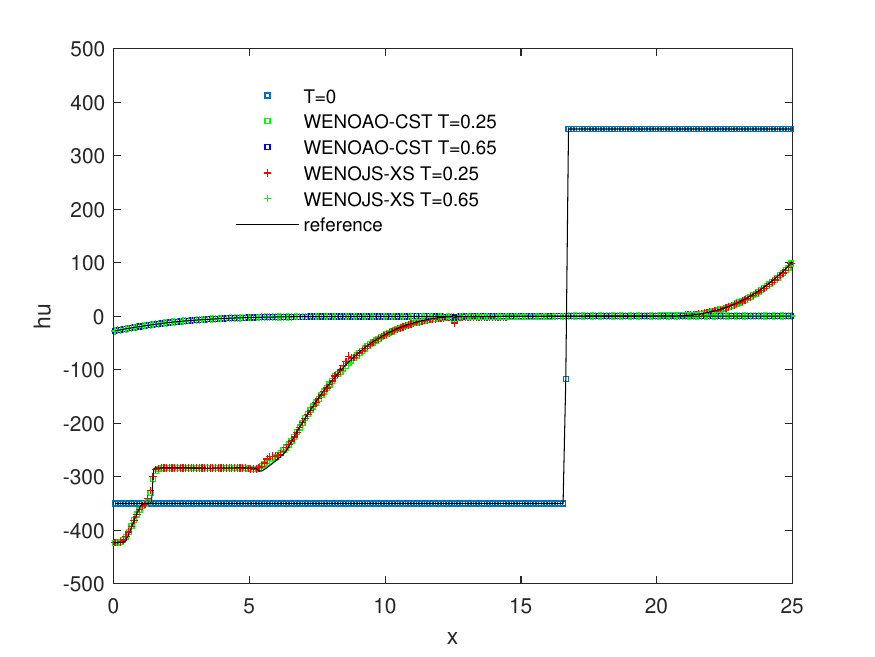}}
\caption{{ The surface level $h + b$ and discharge $hu$ for the initial data \eqref{remain3}. $T=0,~0.25,~0.65$ and $N_x=250$.  Square: WENOAO-CST; plus: WENOJS-XS; solid line: reference solutions.} }
\label{figex8-1001}
\end{figure}

\begin{example}
{ (The exact C-property and positivity-preserving tests in 2D)}\label{smooth-2D}
\end{example}
{ This example is used to demonstrate the well-balanced and positivity-preserving properties of the WENOAO-CST method for 2D SWEs over the non-flat bottoms.}
The computational domain is $[0,1]\times [0,1]$, and the initial conditions are set to be steady state solution
\ben
H(x,y,0)=h(x,y,0)+b(x,y)=1, \, (h u)(x,y,0)=0, \, (h v)(x,y,0)=0.\nonumber
\een
Two bottom topographies are considered
\begin{align}
b(x, y)&=0.8 e^{-50\left((x-0.5)^{2}+(y-0.5)^{2}\right)},\label{wb_2D1} \\
b(x, y)&=e^{-50\left((x-0.5)^{2}+(y-0.5)^{2}\right)}.\label{wb_2D2}
\end{align}
The surface level $h+b$ should remain flat.
The solutions are computed up to $T=0.1$ with $N_x\times N_y=100 \times 100$.
The $L^1$ and $L^\infty$ errors
of $h$, $h u$, and $h v$  are shown in Table \ref{tab:ex6:001}.
We can observe that the numerical errors of the WENOAO-CST method are at the level of round-off errors, confirming the well-balanced property.
{
Notice that the PP limiter is needed for the bottom topography (\ref{wb_2D2}), and no negative water height is generated during the computation.}

\begin{table}[H]
\caption{The $L^1$ and $L^\infty$ errors for the 2D exact C-property tests. $T=0.1$ and $N_x= N_y=100$. }
\vspace{5pt}
\centering
\label{tab:ex6:001}
\begin{tabular}{  c  c  c  c c c c }
\toprule
\multirow{2}{*}{bottom} &\multicolumn{3}{c}{$L^1$ error}&\multicolumn{3}{c}{$L^\infty$ error}\\
&$h$~~~&~~$hu$&~~$hv$&$h$~~~&~~$hu$&~~$hv$\\
\midrule
 (\ref{wb_2D1})&  9.11e-15   & 1.54e-15  & 1.54e-15 &1.12e-14& 8.97e-15 &8.72e-15\\
  (\ref{wb_2D2})& 4.03e-15   & 1.07e-15  & 1.40e-15 &1.08e-14& 1.28e-14 &5.68e-15\\
\bottomrule
\end{tabular}
\end{table}

\begin{example}
(The accuracy test in 2D)\label{accuracy test-2D}
\end{example}

We consider the 2D accuracy test with a non-flat bottom topography.
The computational domain is $[0,1]\times[0,1]$, and the initial conditions are
\begin{align*}
h(x, y, 0)&=10+e^{\sin (2 \pi x)} \cos (2 \pi y), \\
(h u)(x, y, 0)&=\sin (\cos (2 \pi x)) \sin (2 \pi y), \\
(h v)(x, y, 0)&=\cos (2 \pi x) \cos (\sin (2 \pi y)),\nonumber
\end{align*}
with periodic boundary conditions. The bottom topography is given by
\ben
\begin{array}{c}
b(x, y)=\sin (2 \pi x)+\cos (2 \pi y).
\end{array}
\een
We compute the solutions up to $T=0.05$ when they are smooth.
We adopt the WENOJS-XS method with $N_x = N_y = 1600$ to compute the reference solutions and treat them as the exact solutions to compute the numerical errors.
The $L^1$ and $L^\infty$ errors and orders by the  WENOAO-CST method are listed in Tables \ref{tab:ex9:001} and \ref{tab:ex9:002}.
In the 2D case, our WENOAO-CST method achieves the designed fifth-order accuracy.

To show the efficiency of the WENOAO-CST method, we present the $L^1$ and $L^\infty$ errors versus the CPU time for $h$, $hu$ and $hv$ by the WENOAO-CST and WENOJS-XS methods in Fig.~\ref{figex9-1001}. From the results, one can see that our WENOAO-CST method is slightly more efficient than the WENOJS-XS.

\begin{table}[H]
\caption{$L^1$ errors and orders of the WENOAO-CST method for the 2D accuracy test. $T=0.05.$}
\label{tab:ex9:001}
\vspace{5pt}
\centering
\begin{tabular}{ c  c  c  c  c }
\toprule
\multirow{2}{*}{$N_x = N_y$}&\multirow{2}{*}{CFL}&$h$&$hu$&$hv$\\
  & & error~~~~ order&  error~~~~ order& error~~~~ order\\
\midrule
25 &0.6 & 7.15e-03~~~~~~- & 1.82e-02 ~~~~~~- &5.89e-02~~~~~~- \\
50 &0.6 & 9.40e-04 ~2.93 &  1.44e-03  ~3.66 & 7.61e-03 ~2.95 \\
100 &0.4 & 6.80e-05 ~3.79 &  7.80e-05  ~4.21 &  5.51e-04 ~3.79\\
200 &0.3 & 3.06e-06   ~4.47 & 3.28e-06  ~4.57& 2.47e-05 ~4.48\\
400 &0.2 & 1.07e-07   ~4.84 & 1.17e-07    ~4.81 & 8.59e-07 ~4.84\\
800 &0.1 & 3.39e-09  ~4.98 & 3.68e-09      ~4.98 & 2.72e-08 ~4.98\\
\bottomrule
\end{tabular}
\end{table}

\begin{table}[H]
\caption{$L^\infty$ errors and orders of the WENOAO-CST method for the 2D accuracy test. $T=0.05.$ }
\label{tab:ex9:002}
\vspace{5pt}
\centering
\begin{tabular}{ c  c  c  c  c }
\toprule
\multirow{2}{*}{$N_x = N_y$}&\multirow{2}{*}{CFL}&$h$&$hu$&$hv$\\
  & & error~~~~ order&  error~~~~ order& error~~~~ order\\
\midrule
25 &0.6 & 6.74e-02~~~~~~- & 9.40e-02 ~~~~~~- &4.27e-02~~~~~~- \\
50 &0.6 & 1.97e-02 ~1.78 &  2.02e-02  ~2.22 & 1.57e-01 ~1.44 \\
100&0.4 & 2.61e-03 ~2.92 &  1.95e-03  ~3.37 &  1.87e-02 ~3.06\\
200&0.3 & 2.12e-04   ~3.62 & 1.22e-04  ~4.00& 1.57e-03 ~3.58\\
400&0.2 & 8.62e-06   ~4.62 & 4.86e-06    ~4.65 & 6.47e-05 ~4.60\\
800&0.1 & 2.73e-07  ~4.99 & 1.55e-07      ~4.97 & 2.04e-06 ~4.99\\
\bottomrule
\end{tabular}
\end{table}

\begin{figure}[H]
\centering
\subfigure[$h$]{
\includegraphics[width=0.31\textwidth,trim=52 0 52 10,clip]{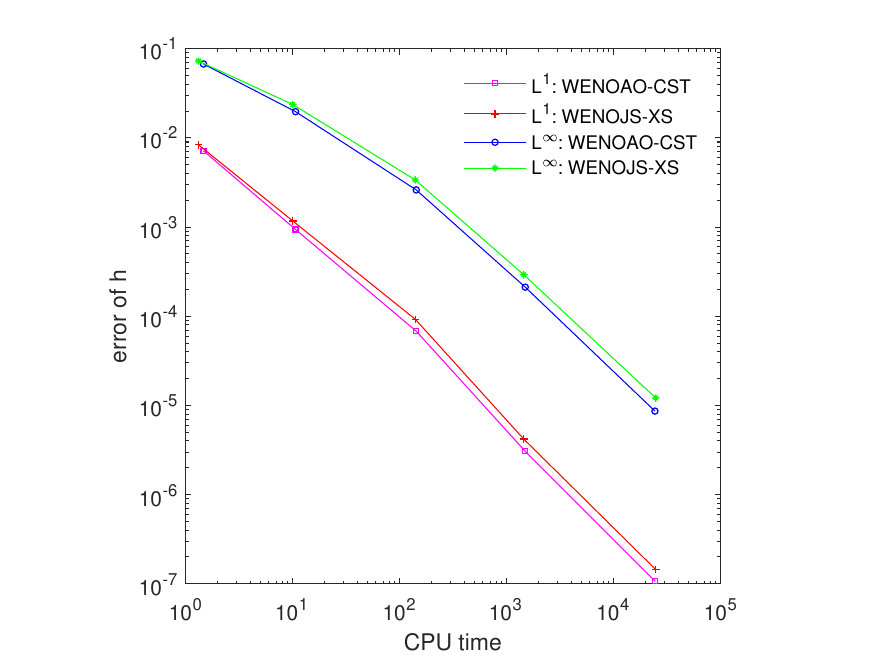}}
\subfigure[$ hu$]{
\includegraphics[width=0.31\textwidth,trim=52 0 52 10,clip]{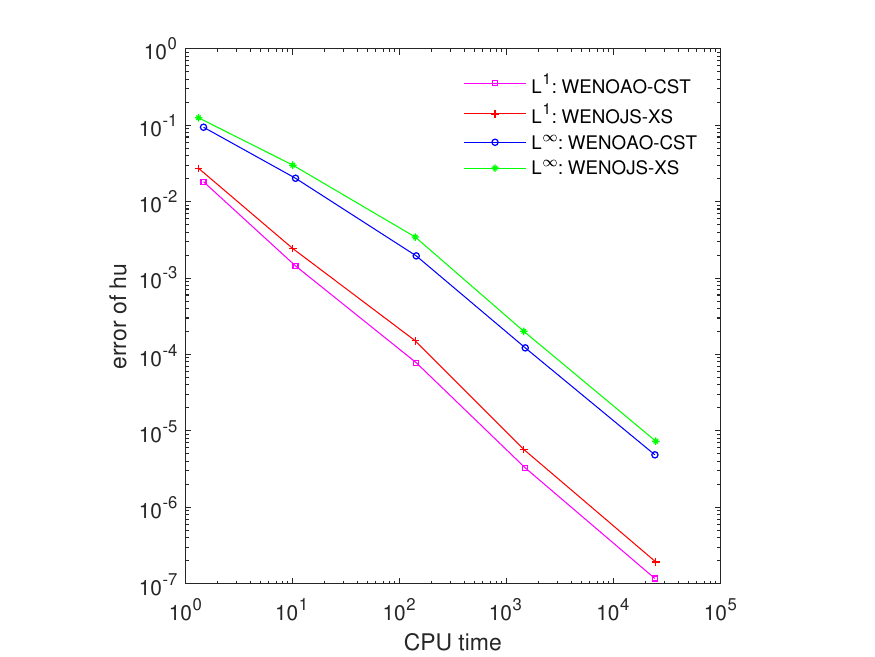}}
\subfigure[$ hv$]{
\includegraphics[width=0.31\textwidth,trim=52 0 52 10,clip]{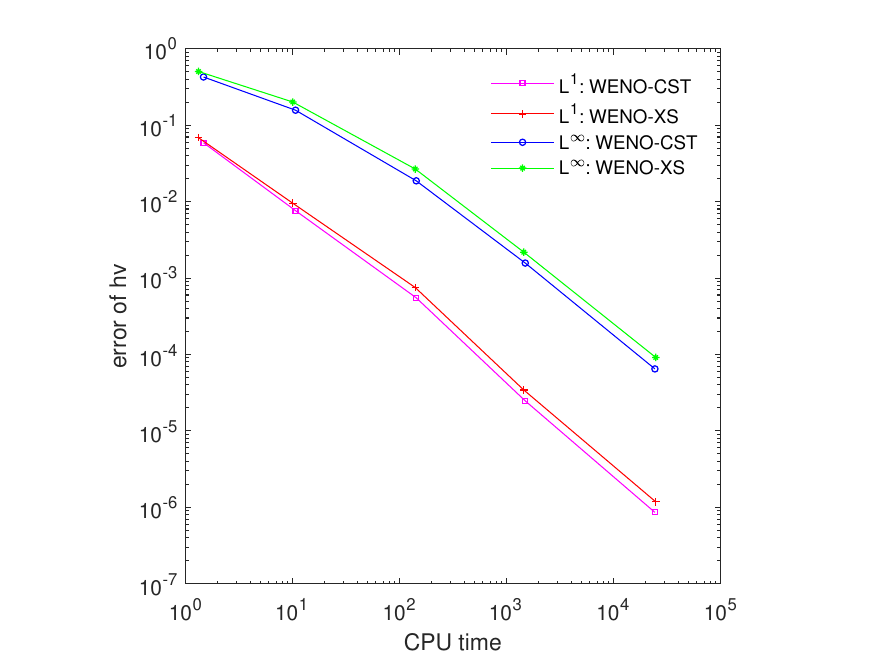}}
\caption{The numerical $L^1$ and $L^\infty$ errors vs. CPU time of the water height $h$ and discharges $hu$, $hv$ by the WENOAO-CST and WENOJS-XS methods.  $T= 0.05$.}
\label{figex9-1001}
\end{figure}

\begin{example}
(The small perturbation of a steady-state flow in 2D)\label{2D_small_per}
\end{example}
We consider this problem to show that the WENOAO-CST method can resolve the complex feature of the perturbation of a steady-state flow.
The computational domain is $[0,2] \times[0,1]$, and the
bottom topography is an isolated hump
$$
b(x, y)=0.8 e^{-5(x-0.9)^{2}-50(y-0.5)^{2}}.
$$
The initial conditions are
\begin{align*}
&h(x, y, 0)=\left\{\begin{array}{ll}
1-b(x, y)+0.01, &   x \in [0.05,~0.15], \\
1-b(x, y), & \text {else},\end{array}\right. \\
&h u(x, y, 0)=h v(x, y, 0)=0.
\end{align*}
The surface level is thus nearly flat except for $x \in [0.05,~0.15]$, where $h$ is perturbed upward by $0.01$. We study the propagation of the right-going disturbance over the hump, and
the surface level $h+b$ at different times $ T = 0.12,~0.24,~0.36,~0.48,~0.60$ by the WENOAO-CST method are presented in Fig.~\ref{figex6001} with  $N_x\times N_y=200 \times 100$ and $600 \times 300$, respectively, for comparison.
The results show that our method can capture the small features of this problem well.

\begin{figure}[H]
\centering
\subfigure[$N_x \times N_y=200\times 100, T=0.12$]{
\includegraphics[width=0.45\textwidth,clip]{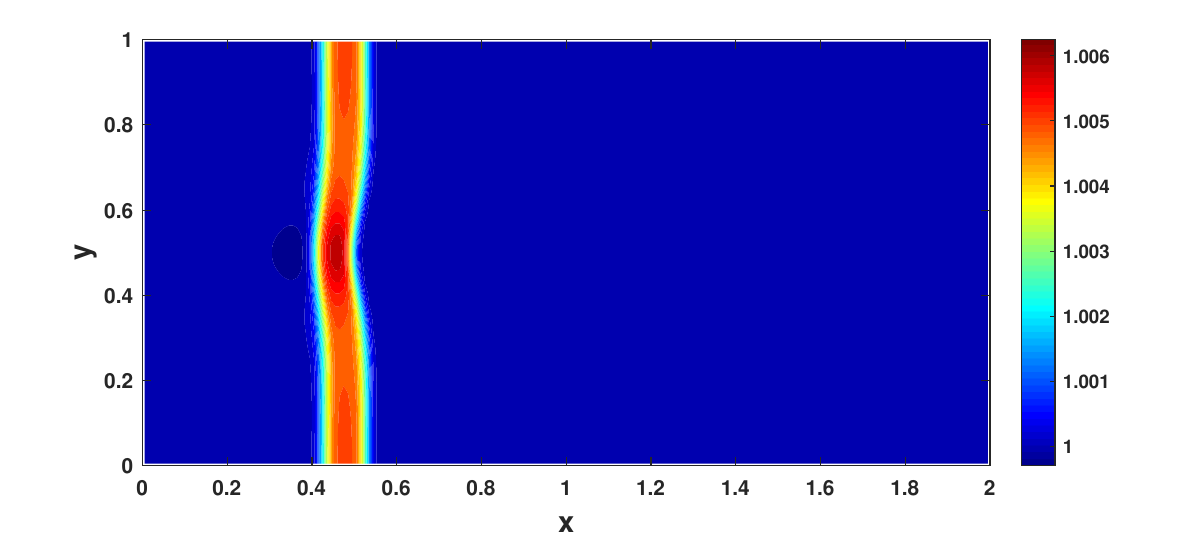}}
\subfigure[$N_x \times N_y=600\times 300, T=0.12$]{
\includegraphics[width=0.45\textwidth,clip]{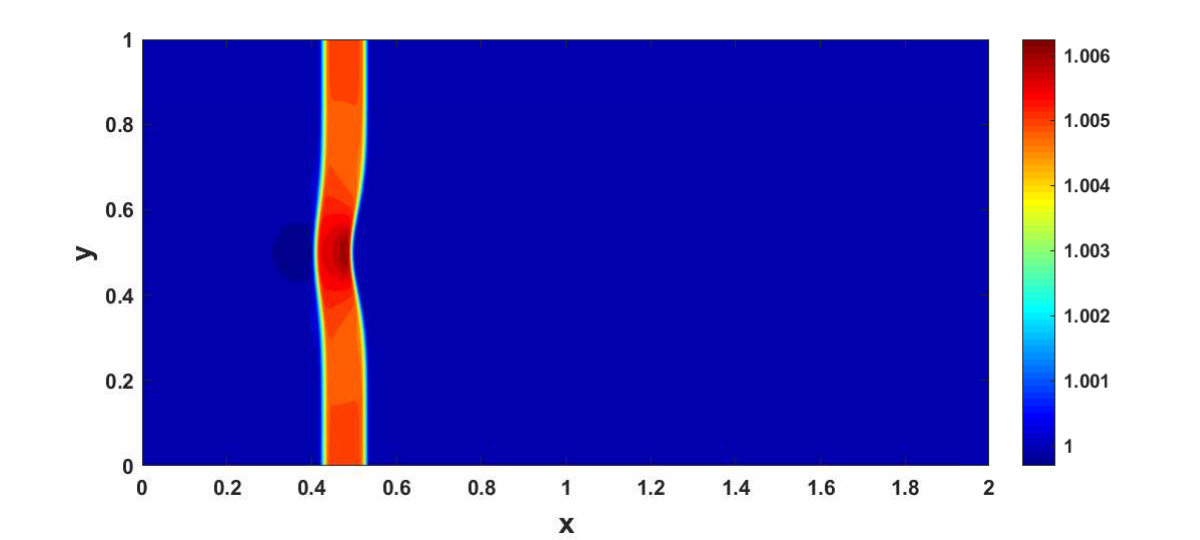}}\\
\subfigure[$N_x \times N_y=200\times 100, T=0.24$]{
\includegraphics[width=0.45\textwidth,clip]{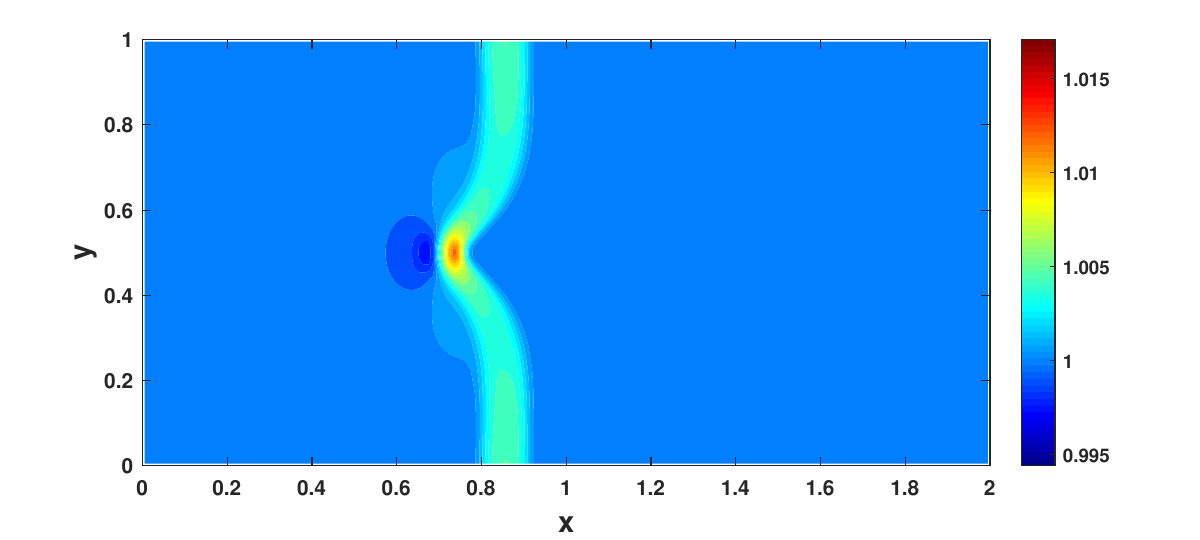}}
\subfigure[$N_x \times N_y=600\times 300, T=0.24$]{
\includegraphics[width=0.45\textwidth,clip]{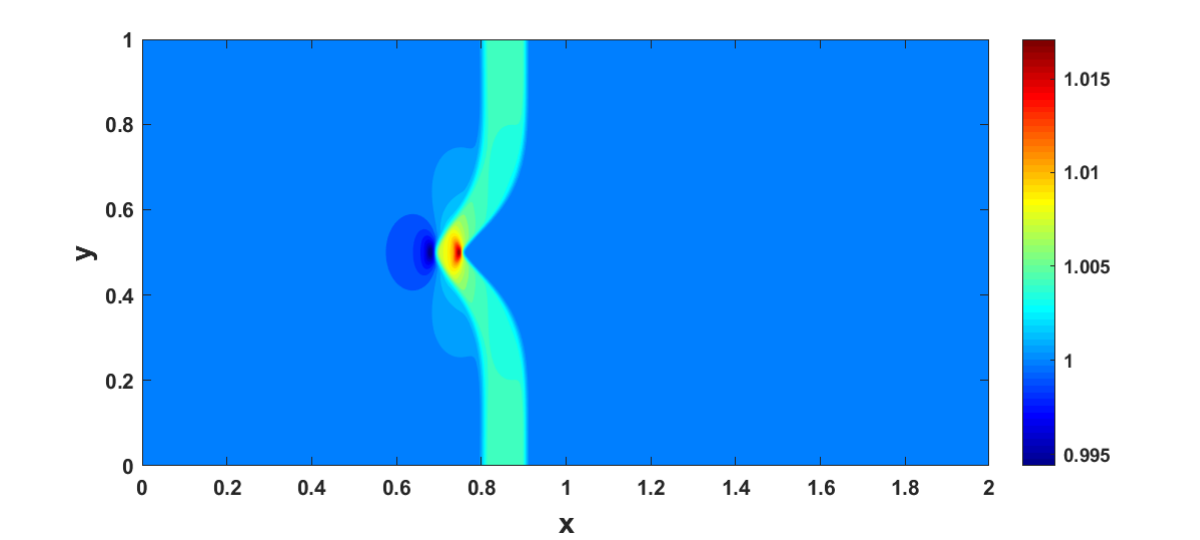}}\\
\subfigure[$N_x \times N_y=200\times 100, T=0.36$]{
\includegraphics[width=0.45\textwidth,clip]{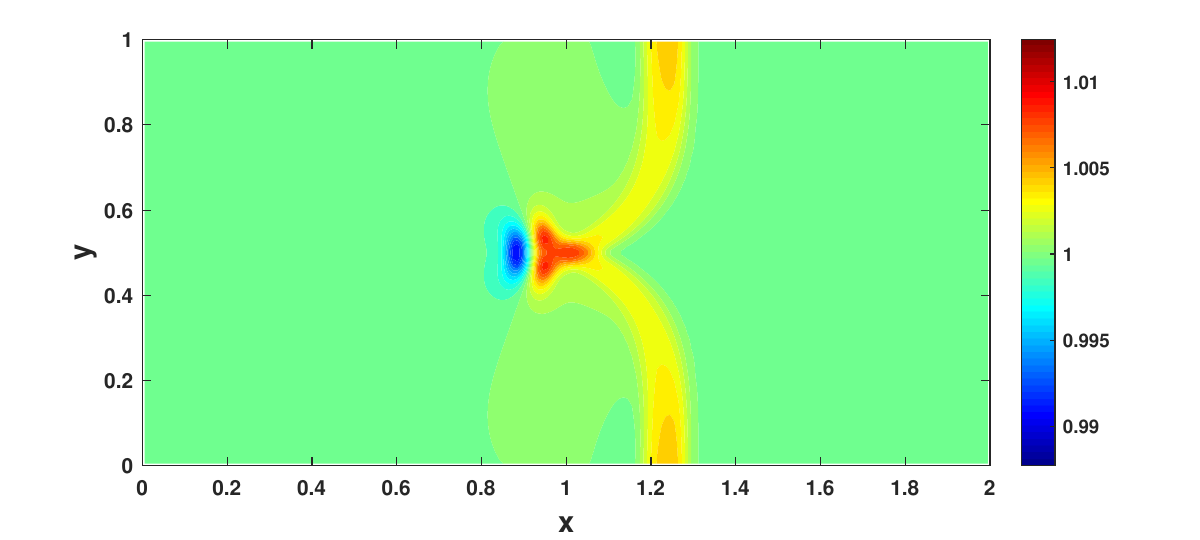}}
\subfigure[$N_x \times N_y=600\times 300, T=0.36$]{
\includegraphics[width=0.45\textwidth,clip]{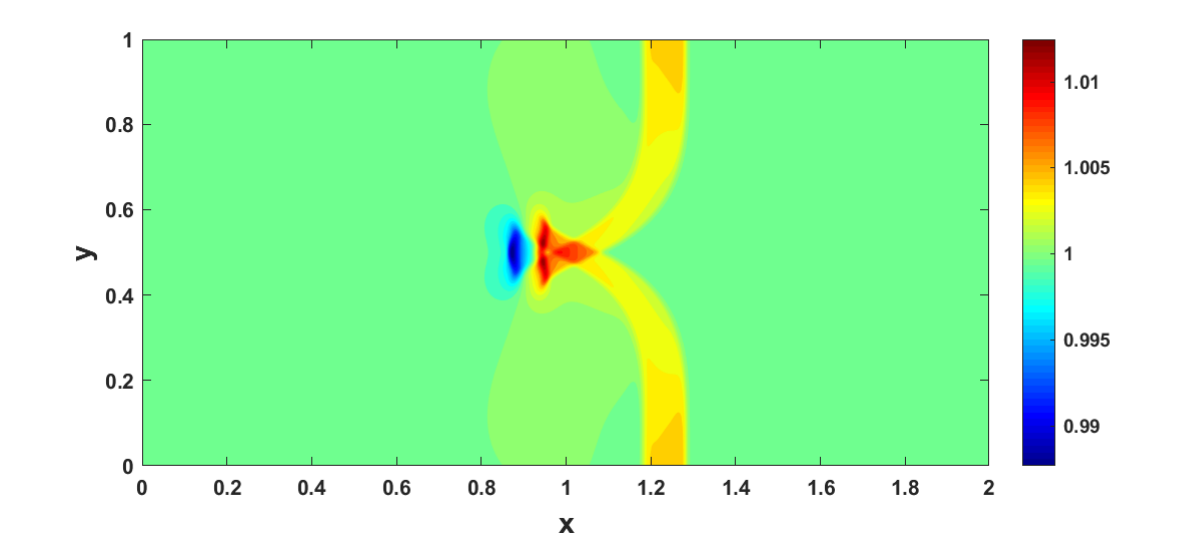}}\\
\subfigure[$N_x \times N_y=200\times 100, T=0.48$]{
\includegraphics[width=0.45\textwidth,clip]{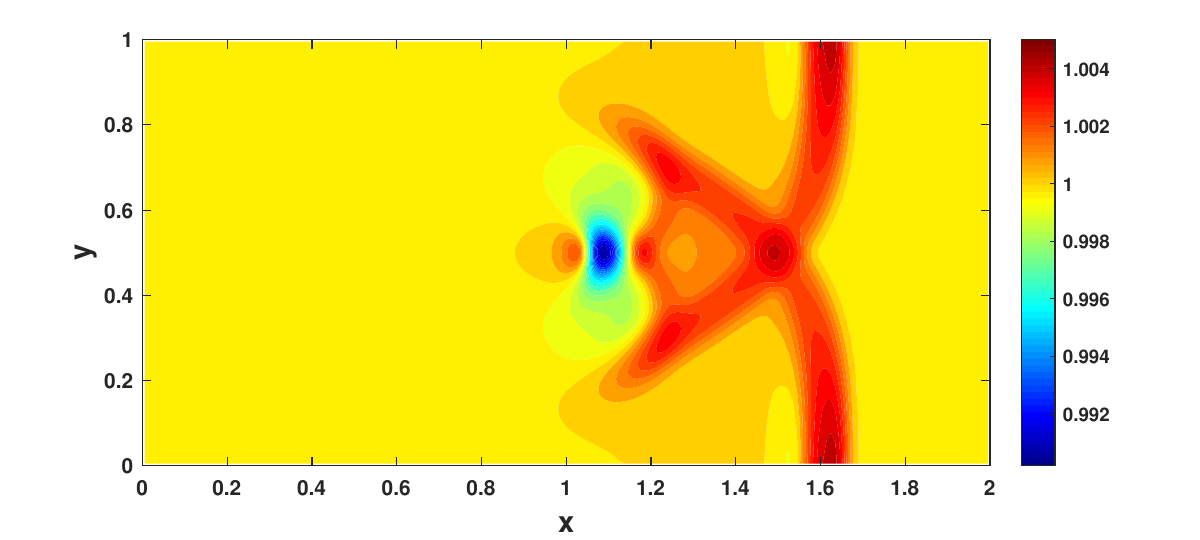}}
\subfigure[$N_x \times N_y=600\times 300, T=0.48$]{
\includegraphics[width=0.45\textwidth,clip]{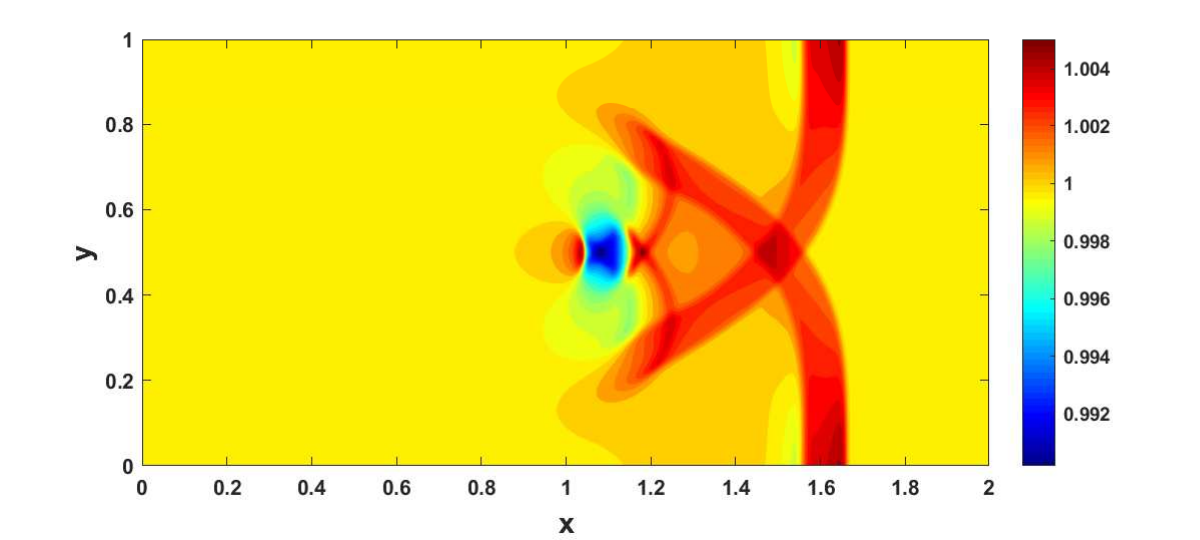}}\\
\subfigure[$N_x \times N_y=200\times 100, T=0.60$]{
\includegraphics[width=0.45\textwidth,clip]{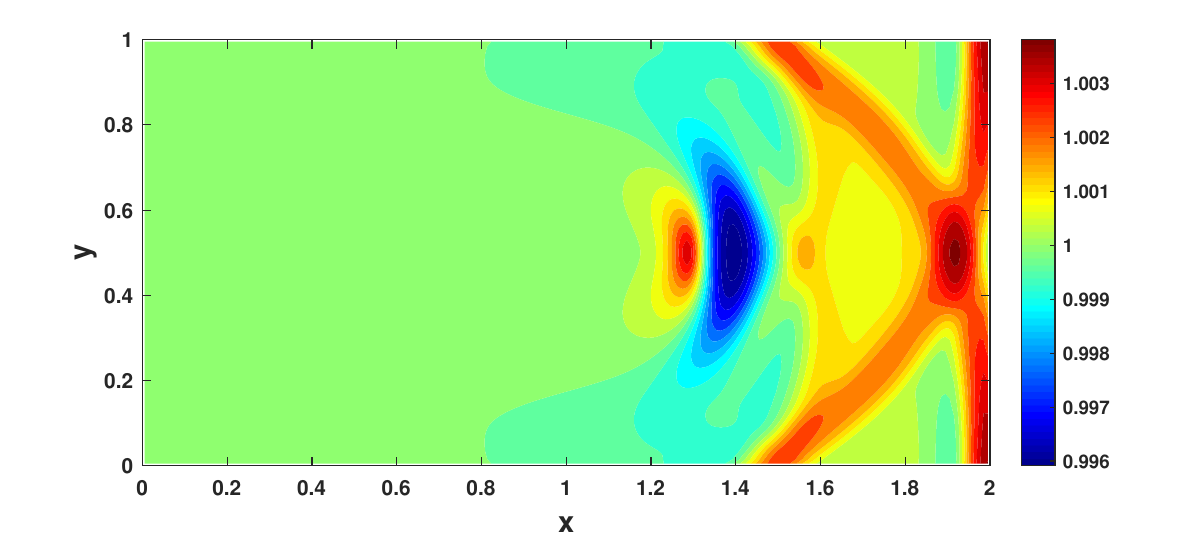}}
\subfigure[$N_x \times N_y=600\times 300, T=0.60$]{
\includegraphics[width=0.45\textwidth,clip]{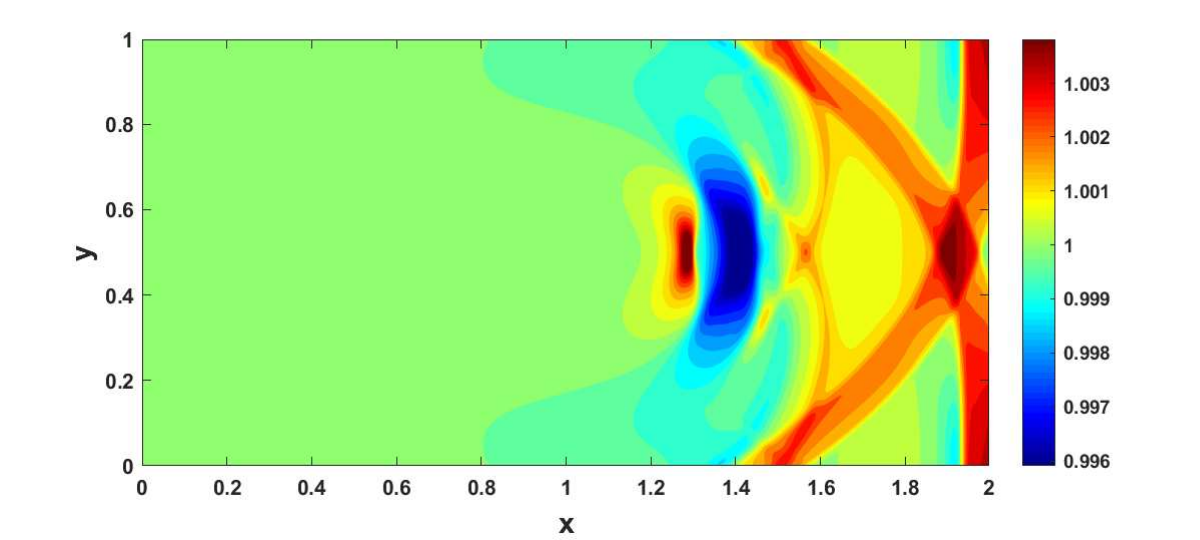}}\\
\caption{The contours of the surface level $h+b$ by the WENOAO-CST method for small perturbation of a steady-state flow in 2D. 30 uniformly spaced contour lines. }
\label{figex6001}
\end{figure}

\begin{example}
(The circular dam break in 2D)\label{circular_Dam}
\end{example}
We model the 2D circular dam break problem  \cite{capilla2013new,WANG20201387}. The computational domain is  $[0,2] \times[0,2]$ and the bottom topography and the initial conditions are
\begin{align*}
&b(x, y)=\begin{cases}
\frac{1}{8}(\cos (2 \pi(x-0.5))+1)(\cos (2 \pi y)+1), & \sqrt{(x-1.5)^{2}+(y-1)^{2}} \leqslant 0.5, \\
0, & \text {else. }
\end{cases}\\
&H(x, y, 0)=\begin{cases}
1.1, & \sqrt{(x-1.25)^{2}+(y-1)^{2}} \leqslant 0.1, \\
0.6, & \text {else, }
\end{cases}
\quad
hu(x, y, 0)=hv(x, y, 0)=0 .
\end{align*}
The solutions are calculated up to $T=0.15$ with $N_x=N_y=200$.  The surface level $h+b$ by the WENOAO-CST and WENOJS-XS methods, as well as the cut of the corresponding results along the line $y=1$, are shown in Fig.~\ref{figex12001}.
One can see that our WENOAO-CST method gives well-resolved and non-oscillatory solutions for this problem, and the results of the two methods are comparable.

\begin{figure}[H]
\centering
\subfigure[$h+b$]{
\includegraphics[width=0.45\textwidth,clip]{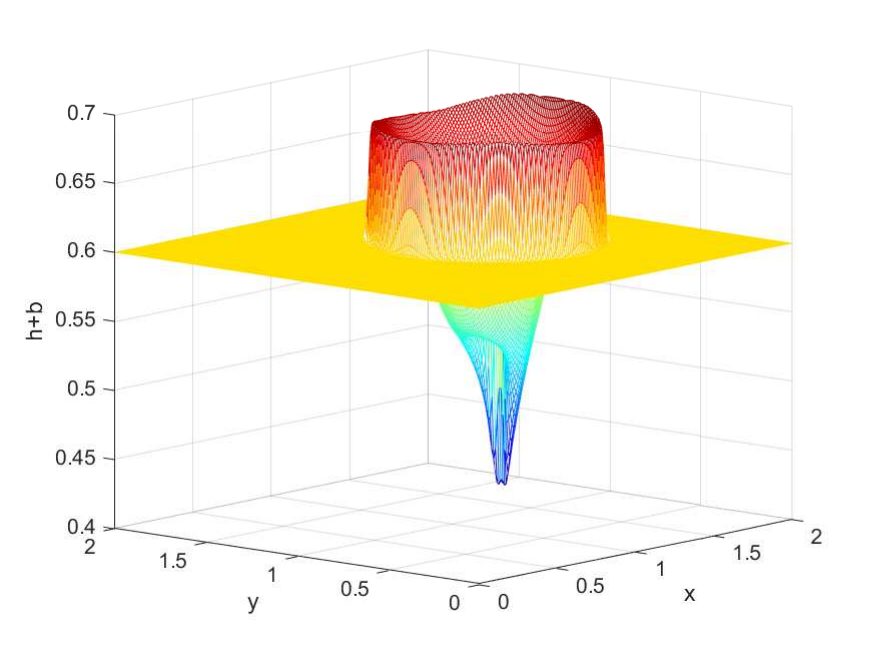}}
\subfigure[$h+b$]{
\includegraphics[width=0.45\textwidth,clip]{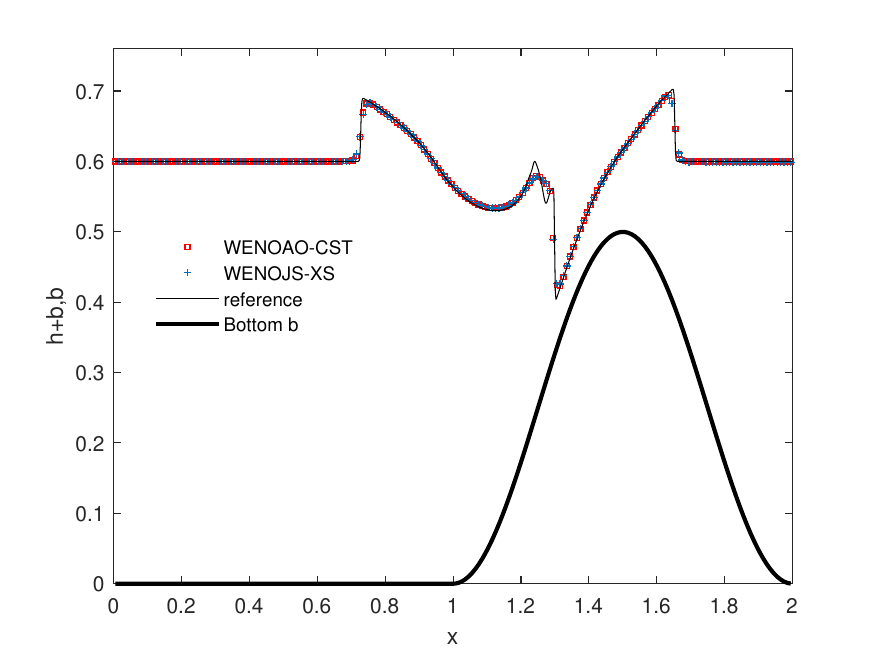}}
\caption{2D circular dam break problem.
The surface level $h +b$  and the cut of the
corresponding results along the line $y = 1$.
$T = 0.15$ and $N_x = N_y = 200 $.
Square: WENOAO-CST; plus: WENOJS-XS; solid line: reference solutions.}
\label{figex12001}
\end{figure}

\begin{example}
{ (The positivity-preserving tests over the flat, rectangle, and hump bottoms in 2D)}\label{Oblique_2D}
\end{example}
{
This example is used to demonstrate the positivity-preserving property of our WENOAO-CST method for the 2D SWEs over the flat, rectangle, and hump bottom topographies, respectively.}
We first investigate the evolution of water over a flat bottom \cite{berthon2008positive}, which results in a moving front with an inclination of $45^{\circ}$ for the boundary of the computational domain $[-0.5,0.5] \times[-0.5,0.5] $.
The initial conditions are provided by
\begin{equation}\label{oblique}
h(x, y, 0)=\begin{cases}
1, &  x+y \leqslant 0, \\
0, & \text {else, }
\end{cases}
\quad  h u(x, y, 0)=h v(x, y, 0)=0,
\end{equation}
with the transmissive boundary condition in each direction.
Half of the domain is dry area, and the other half has still water with a water height equal to 1.
The analytical solutions to this problem can be found in the literature \cite{bokhove2005flooding}.
At different times $T = 0,~ 0.02,~0.06,~ 0.1$, the water height $h$ with $N_x=N_y=100$ on the central cross section ($x = y$)  orthogonal to the propagating front are displayed in Fig.~\ref{figex11002} with the exact solutions.
A good agreement between the numerical and exact solutions is observed, and both methods can get the non-oscillatory and non-negative solutions for this problem.

\vspace{8pt}

{
Next, we consider the tests on the computational domain $[0, 25]\times[0, 25]$, and the initial conditions are set as
\begin{equation}\label{remain4}
\begin{split}
&H(x,y, 0) = 10,\quad
h u(x,y, 0)=\begin{cases}
-350, &  x \leqslant \frac{50}{3}, \\
350, & \text { else, }
\end{cases}
\quad hv(x,y,0)=0,
\end{split}
\end{equation}
with the rectangle and hump bottoms
\begin{align}
&b(x,y)=\begin{cases}
1, & \frac{25}{3} \leqslant x \leqslant \frac{25}{2}, \\
0, & \text { else,}
\end{cases}\label{b1}
\\&
b(x,y)=\begin{cases}
\sin\left(6\pi(\frac{x}{25}-\frac{1}{3})\right), & \frac{25}{3} \leqslant x \leqslant\frac{25}{2}, \\
0, & \text { else.}
\end{cases}\label{b2}
\end{align}
The transmissive boundary condition is applied in each direction.
The water surface  at different times  $T=0,~0.05,~0.25$,  and $0.65$ are shown in Fig.~\ref{figex12-1001} with $N_x=N_y =250$.
Furthermore, we present the water surface
in the central cross section ($x = y$) in Fig.~\ref{figex12-1002}.
We can observe that results of the two  schemes are comparable, and both schemes can
capture the solutions well in the dry region.
}

\begin{figure}[H]
\centering
\includegraphics[width=0.45\textwidth,trim=20 0 20 10,clip]{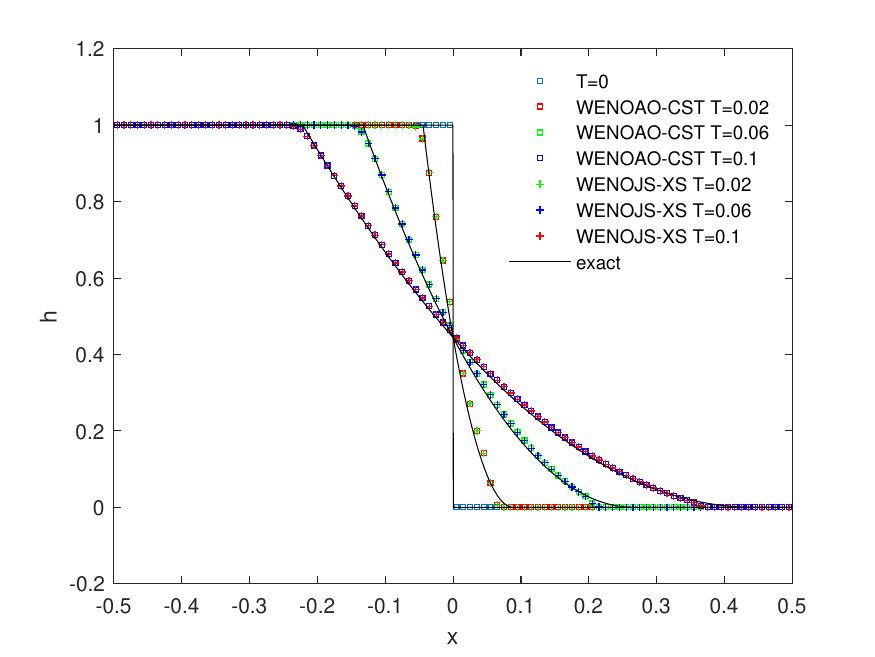}
\caption{
{
The water height $h$ at different times $T = 0,~ 0.02,~0.06,~ 0.1$ in the central cross-section ($x = y$) for 2D SWEs with initial condition \eqref{oblique}. $N_x= N_y=100$. Square: WENOAO-CST; plus: WENOJS-XS; solid line: exact solutions.}}
\label{figex11002}
\end{figure}

\begin{figure}[H]
\centering
\subfigure[bottom \eqref{b1}]{
\includegraphics[width=0.45\textwidth,clip]{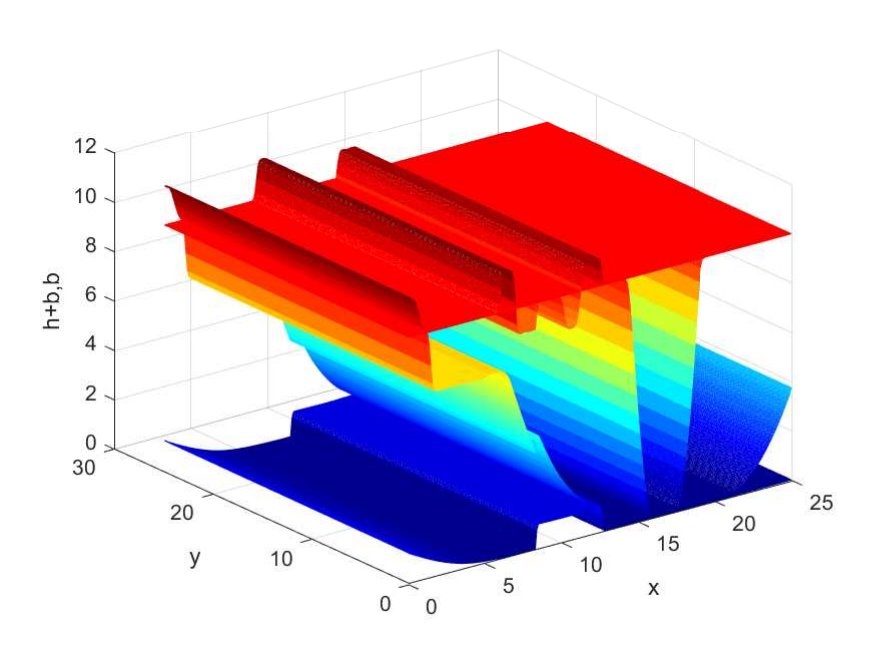}}
\subfigure[bottom \eqref{b2}]{
\includegraphics[width=0.45\textwidth,clip]{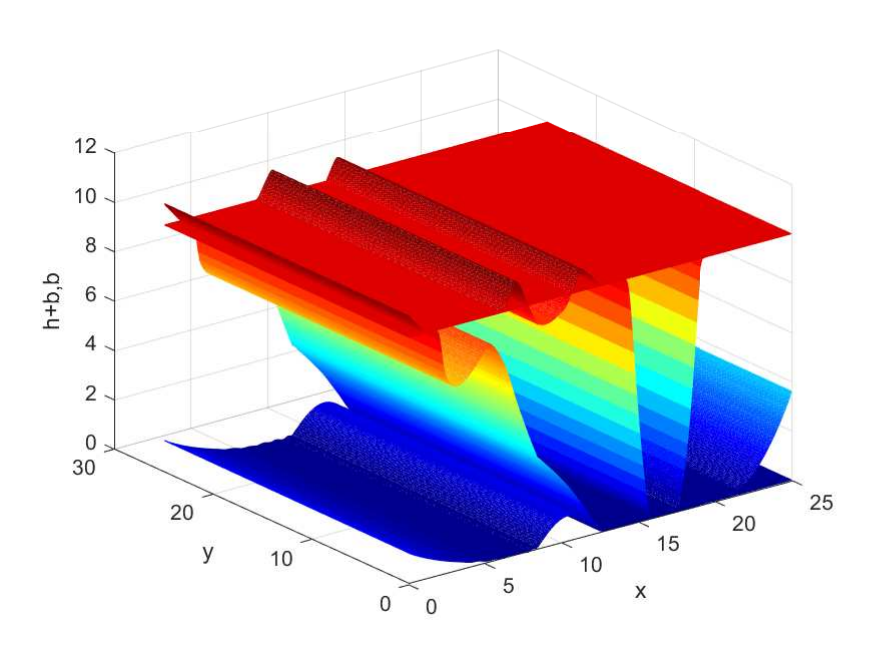}}
\caption{{ The surface level $h + b$ at $T=0,~0.05,~0.25,~0.65$ by the WENOAO-CST method for the 2D SWEs with the initial condition \eqref{remain4} for bottoms \eqref{b1} (left) and \eqref{b2} (right).
$N_x= N_y =250$.
} }
\label{figex12-1001}
\end{figure}

\section{Conclusions}
\label{sec_con}
In this paper, a fifth-order well-balanced finite volume WENOAO-CST method is developed to solve one- and two-dimensional SWEs with flat and non-flat bottom topographies { in the CST pre-balanced form}.
The WENO-AO approach is used in the reconstruction to ensure the high-order accuracy and non-oscillatory properties of the method.
Compared with other WENO reconstruction schemes( e.g., WENO-ZQ and WENO-MR reconstructions), the WENO-AO reconstruction has better resolution for some numerical examples.
Furthermore, a PP limiter is adopted for the computation of solutions with dry areas.
The proposed WENOAO-CST method outperforms the WENOJS-XS method \cite{xing2011high} in terms of accuracy and efficiency verified by the numerical results.
Future work also includes the application of the finite volume WENOAO-CST method for other related models, such as the SWEs in open channels with irregular geometry.

\section*{Acknowledgements}
{
The authors thank the anonymous referees for their valuable comments and suggestions that helped improve the quality of the paper.}
Min Zhang is partially supported by the National Natural Science Foundation of China (Grant Number: 12301493).

\section*{Data availability}
Data sharing not applicable to this article as no datasets were generated or analysed during the current study.

\begin{figure}[H]
\centering
\subfigure[$h+b,~b,~T=0.05$]{
\includegraphics[width=0.45\textwidth,clip]{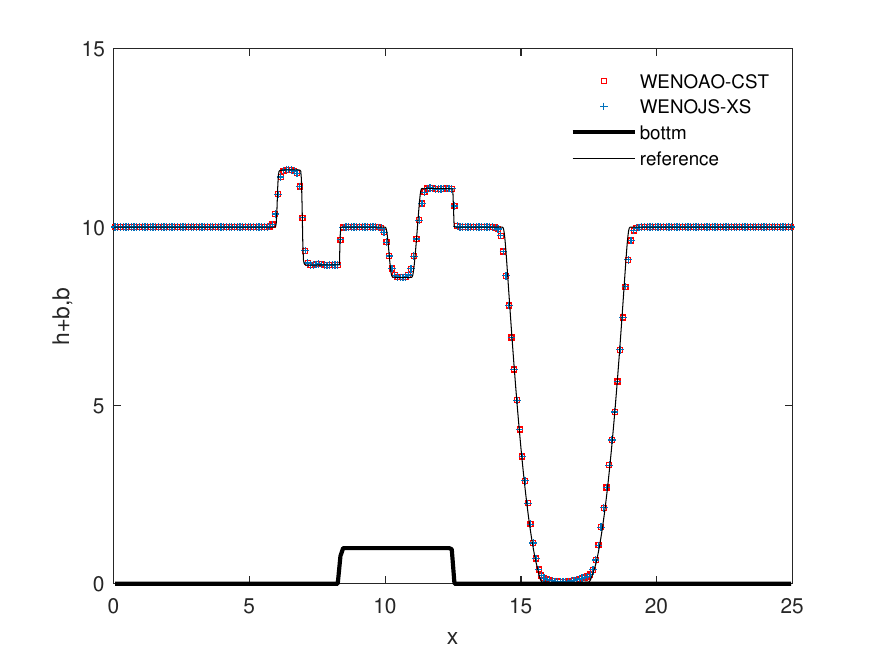}}
\subfigure[$h+b,~b,~T=0.05$]{
\includegraphics[width=0.45\textwidth,clip]{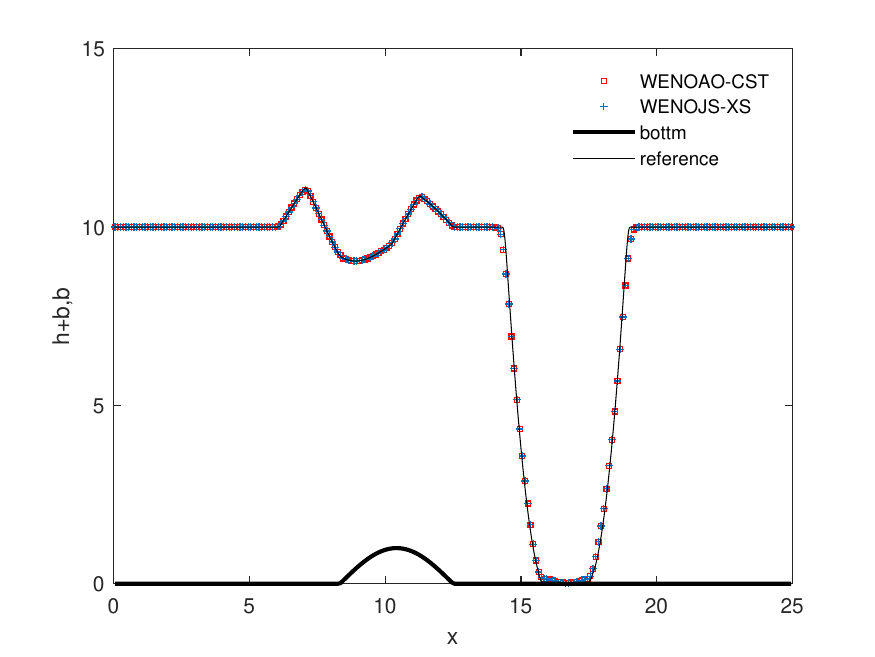}}
\subfigure[$h+b,~b,~T=0.25$]{
\includegraphics[width=0.45\textwidth,clip]{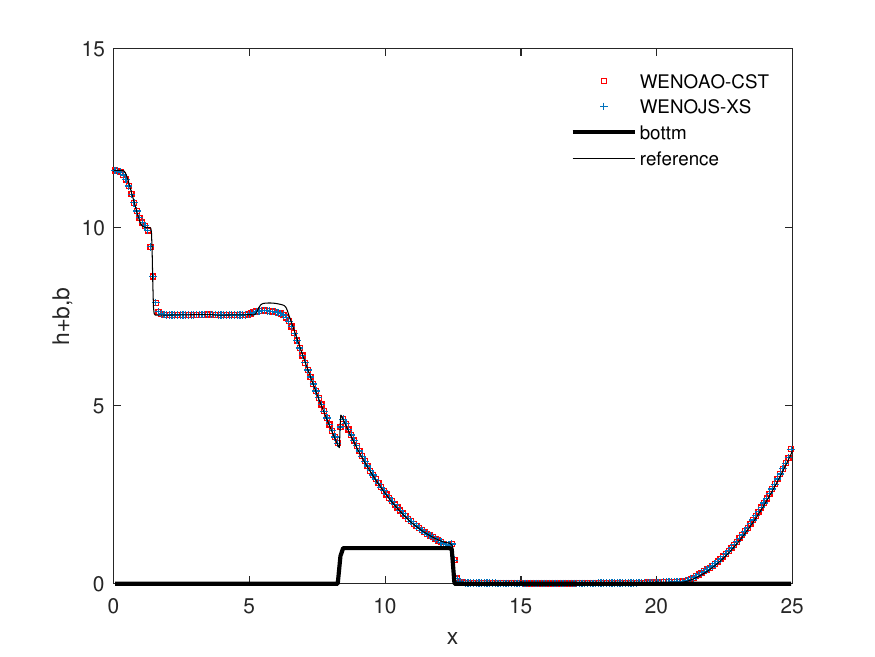}}
\subfigure[$h+b,~b,~T=0.25$]{
\includegraphics[width=0.45\textwidth,clip]{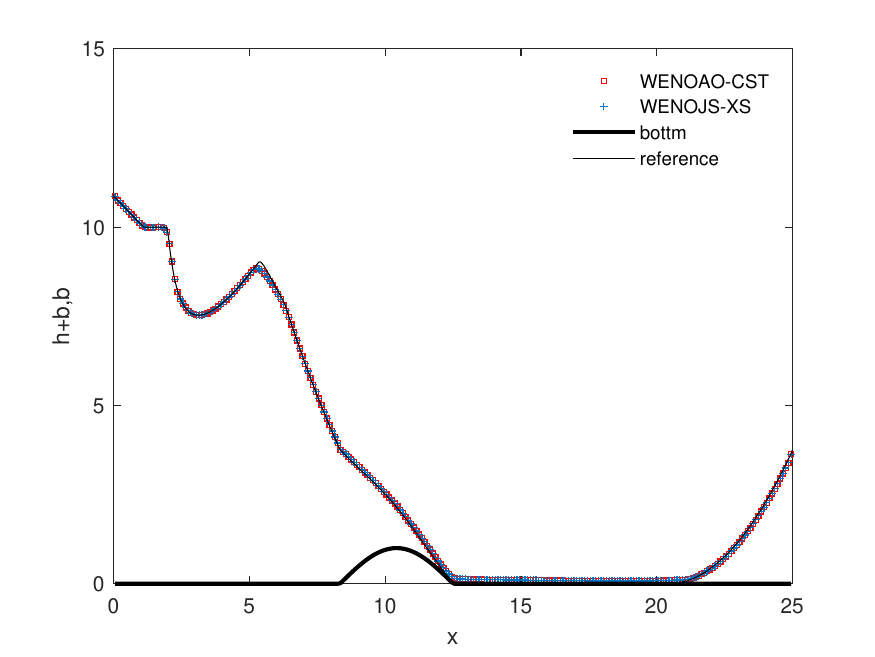}}
 \subfigure[$h+b,~b,~T=0.65$]{
\includegraphics[width=0.45\textwidth,clip]{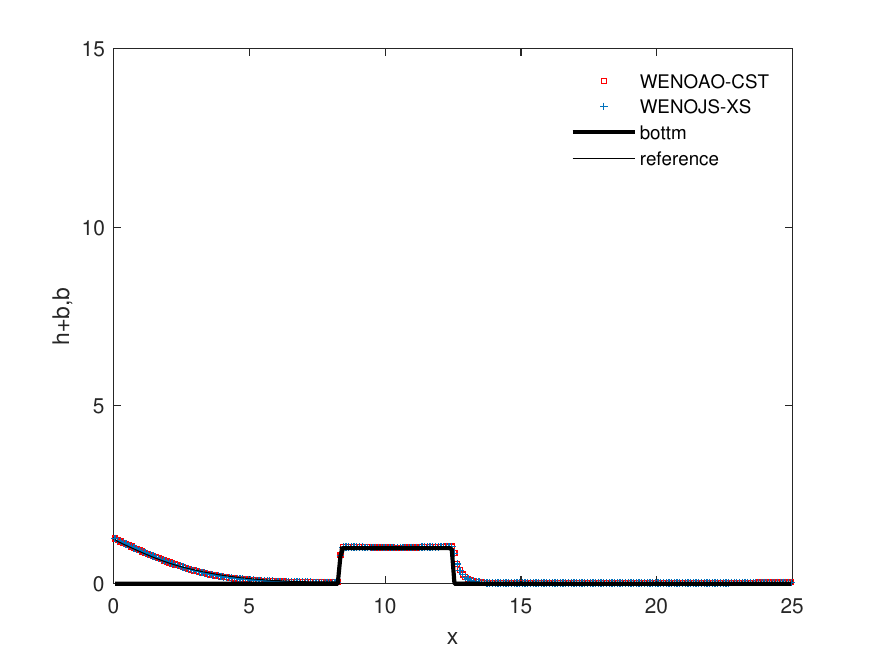}}
 \subfigure[$h+b,~b,~T=0.65$]{
\includegraphics[width=0.45\textwidth,clip]{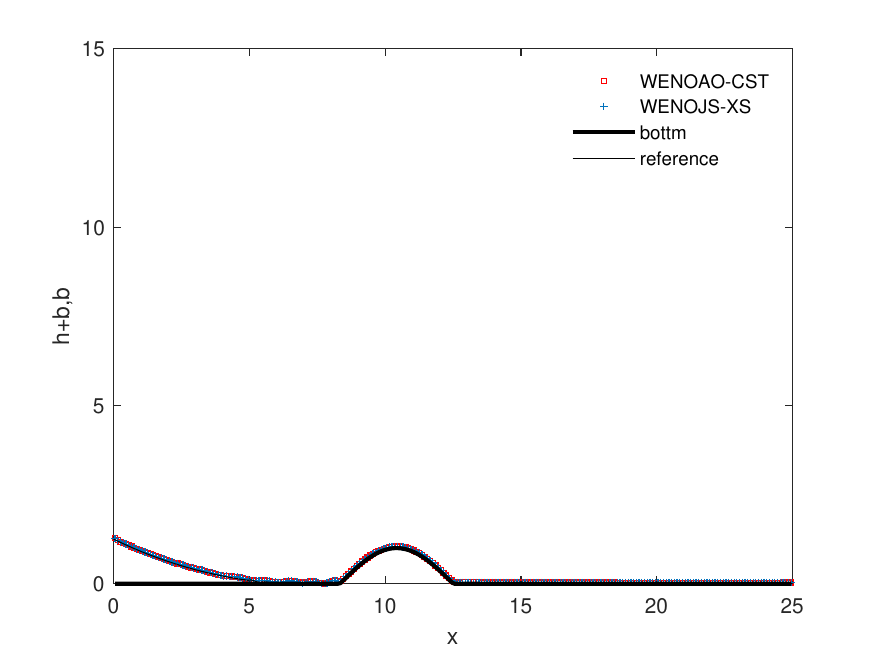}}
\caption{{ The surface level $h + b$ at $T=0.05,~0.25,~0.65$ in the central cross-section for the 2D SWEs with the initial condition \eqref{remain4} for bottoms \eqref{b1} (left) and \eqref{b2} (right). $N_x=N_y=250$. Square: WENOAO-CST; plus: WENOJS-XS; solid line: reference solutions.}}
\label{figex12-1002}
\end{figure}

\section*{Declaration of competing interest}
The authors declare that they have no known competing financial interests or personal relationships that could have appeared to influence the work reported in this paper.

\section*{Appendix A.\ The fifth-order WENO-AO reconstruction}
\label{app-A}
In this section, we introduce the fifth-order WENO-AO reconstruction and refer to \cite{BALSARA2016780} for further details.
Without loss of generality, we here only describe the detailed WENO-AO reconstruction of $u_{i +\ha}^{-}$ and $u_{i -\ha}^{+}$ based on the values $\left\{\overline{u}_{j}\right\}_{j=i-2}^{i+2}$.
We reconstruct three quadratic polynomials $p_{k}(x)~(k=1,2,3)$, and one quartic polynomial $p_{4}(x)$ satisfying the following conditions
\begin{flalign}
\frac{1}{\Delta x} \int_{I_{i+j}} p_{1}(x) d x=\overline{u}_{i+j}, \quad& j=-2,-1,0 , \nonumber\\
\frac{1}{\Delta x} \int_{I_{i+j}} p_{2}(x) d x=\overline{u}_{i+j},\quad & j=-1,0,1,  \nonumber \\
\frac{1}{\Delta x} \int_{I_{i+j}} p_{3}(x) d x=\overline{u}_{i+j}, \quad& j=0,1,2, \nonumber\\
\frac{1}{\Delta x} \int_{I_{i+j}} p_{4}(x) d x=\overline{u}_{i+j}, \quad& j=-2,-1,0,1,2. \nonumber
\end{flalign}
For the fifth-order WENO-AO reconstruction, the linear weights  ${\gamma}_{k}~(k=1,2,3,4)$ can be any positive numbers with the only requirement that their sum equals one, and the suggested  $\left(\gamma_{1}, \gamma_{2}, \gamma_{3}, \gamma_{4}\right)=(0.01125,~0.1275,~0.01125,~0.85)$ is used in this paper. The nonlinear weights are calculated as follows
\ben
\begin{array}{l}
\omega_{k}=\frac{\alpha_{k} }{\sum\limits_{k=1}^{4} \alpha_{k}}, \quad \alpha_{k}=\gamma_{k}\left[1+\left(\displaystyle\frac{\tau}{\beta_{k}+10^{-12}}\right)^{2}\right], \quad
\tau=\frac{1}{3}\sum\limits_{k=1}^{3}\left|\beta_{4}-\beta_{k}\right|,
\end{array}
\een
where $\beta_{k}$ is the smoothness indicator of  $p_{k}(x)$ on  $ I_{i}$ given by
\begin{align*}
\beta_{k}=\sum\limits_{\alpha=1}^{\kappa}\int_{I_i}\Delta x^{2\alpha-1}\left(\frac{d^\alpha p_{k}(x)}{dx^\alpha}\right)^2dx,
\end{align*}
with $\kappa=2$ for $k=1,\ 2,\ 3$, and $\kappa=4$ for $k=4$.

Finally, the WENO-AO reconstructions of $u(x)$ and $u'(x)$ are
\begin{flalign*}
u(x) \approx P(x) = \frac{\omega_{4}}{\gamma_{4}}\left[p_{4}(x)-\sum_{k=1}^{3} \gamma_{k} p_{k}(x)\right]+\sum_{k=1}^{3} \omega_{k} p_{k}(x),\label{WENO-AO1}\\
u'(x) \approx P'(x) = \frac{\omega_{4}}{\gamma_{4}}\left[p'_{4}(x)-\sum_{k=1}^{3} \gamma_{k} p'_{k}(x)\right]+\sum_{k=1}^{3} \omega_{k} p'_{k}(x),
\end{flalign*}
Then, we can obtain the fifth-order approximations
$$u_{i +\ha}^{-}=P(x_{i +\ha}),\quad u_{i -\ha}^{+}=P(x_{i -\ha}).$$

\section*{Appendix B.\ The proof of Proposition \ref{prop_ahpha_1D}}
In this section, we give a proof of Proposition \ref{prop_ahpha_1D} based on the idea of \cite{xing2011high,xing2010positivity}.
We start by showing the following lemma on the positivity of a first-order Lax-Friedrichs scheme with a well-balanced flux.
\begin{lemma}\label{lemma1}
Consider the first-order scheme with the Lax-Friedrichs flux
\begin{align*}
{h}_{i}^{n+1}
={h}_{i}^{n}&-\lambda\left(\hat{\bF}^h\left(h_{i}^{*,+},u_{i}^n;
h_{i+1}^{*,-},u_{i+1}^n\right)-\hat{\bF}^h\left(h_{i-1}^{*,+},u_{i-1}^n;
h_{i}^{*,-},u_{i}^n\right)\right),
\end{align*}
where
\begin{align*}
&h^{*,+}_{i}=\max(0,h^n_{i}+b_{i}-\max(b_{i},b_{i+1})),
\\&h^{*,-}_{i}=\max(0,h^n_{i}+b_{i}-\max(b_{i-1},b_{i})).
\end{align*}
If ${h}_{i}^{n},~h_{i\pm1}^n$ are non-negative, then ${h}_{i}^{n+1}$ is also non-negative under the CFL condition $\lambda\alpha\leqslant1$, with
$\alpha=\max\limits_i(|u_i^n|+\sqrt{gh_i^n})$.
\end{lemma}
\begin{proof}
\begin{align*}
{h}_{i}^{n+1}&={h}_{i}^{n}-\frac{\lambda}{2}\left(h_i^{*,+}u_i^n+h_{i+1}^{*,-}u_{i+1}^n-
 \alpha(h_{i+1}^{*,-}-h^{*,+}_i)\right) \\
 & \quad \quad \,\, +\frac{\lambda}{2}\left(h_{i-1}^{*,+}u_{i-1}^n+h_{i}^{*,-}u_{i}^n-
 \alpha(h_{i}^{*,-}-h^{*,+}_{i-1})\right) \\
 &=\left(1-\frac{\lambda}{2}(\alpha+u_i^n)\frac{h_i^{*,+}}{h_i^n}
-\frac{\lambda}{2}(\alpha-u_i^n)\frac{h_i^{*,-}}{h_i^n}\right)h_i^n\\
& \quad \,
+\left[\frac{\lambda}{2}(\alpha+u_{i-1}^n)\frac{h_{i-1}^{*,+}}{h_{i-1}^n}\right]h_{i-1}^n
+\left[\frac{\lambda}{2}(\alpha-u_{i+1}^n)\frac{h_{i+1}^{*,-}}{h_{i+1}^n}\right]h_{i+1}^n
\end{align*}
One can see that ${h}_{i}^{n+1}$ is a linear combination of ${h}_{i}^{n},~h_{i-1}^n$ and $h_{i+1}^n$ and all the
coefficients are non-negative. Thus, ${h}_{i}^{n+1}\geq0$.
\end{proof}

\vspace{8pt}

Now, we provide a complete proof of Proposition \ref{prop_ahpha_1D}.

\vspace{8pt}

\begin{proof}
From (\ref{Xi}), we have
\begin{equation}\label{simple-eq}
\overline{h}_{i}^{n}=\left(1-w_{1}-w_{4}\right) \xi_{i}+w_{1} h_{i-\frac{1}{2}}^{+}+w_{4} h_{i+\frac{1}{2}}^{-}.
\end{equation}
Plug \eqref{simple-eq} into \eqref{discrete_10} and rewrite it as
\begin{align*}
\overline{h}_{i}^{n+1}=
& \left(1-w_{1}-{w}_{4}\right) \xi_{i}+w_{1} h_{i-\frac{1}{2}}^{+}+w_{4} h_{i+\frac{1}{2}}^{-}
\\
&-\lambda\left[\hat{\bF}^h\left( h_{i+\frac{1}{2}}^{*,-} ,u_{i+\frac{1}{2}}^{-};  h_{i+\frac{1}{2}}^{*,+},u_{i+\frac{1}{2}}^{+}\right)\right.
-\hat{\bF}^h\left( h_{i-\frac{1}{2}}^{*,+} ,u_{i-\frac{1}{2}}^{+};  h_{i+\frac{1}{2}}^{*,-},u_{i+\frac{1}{2}}^{-}\right)
\\&~~~~+\hat{\bF}^h\left(h_{i-\frac{1}{2}}^{*,+} ,u_{i-\frac{1}{2}}^{+};  h_{i+\frac{1}{2}}^{*,-},u_{i+\frac{1}{2}}^{-}\right)
- \left.\hat{\bF}^h\left( h_{i-\frac{1}{2}}^{*,-} ,u_{i-\frac{1}{2}}^{-}; h_{i-\frac{1}{2}}^{*,+},u_{i-\frac{1}{2}}^{+}\right)\right]\\
=&\left(1-w_{1}-w_{4}\right) \xi_{i}+w_{1} Z_{1}+w_{4} Z_{4},
\end{align*}
where
\begin{align*}
Z_{1}=&  h_{i-\frac{1}{2}}^{+}
-\frac{\lambda}{w_{1}}\left[\hat{\bF}^h\left( h_{i-\frac{1}{2}}^{*,+},u_{i-\frac{1}{2}}^{+} ;h_{i+\frac{1}{2}}^{*,-}, u_{i+\frac{1}{2}}^{-}\right)-\hat{\bF}^h\left(h_{i-\frac{1}{2}}^{*,-},u_{i-\frac{1}{2}}^{-}  ; h_{i-\frac{1}{2}}^{*,+},u_{i-\frac{1}{2}}^{+}\right)\right] \\
=& h_{i-\frac{1}{2}}^{+}
-\frac{\lambda}{2w_{1}}\left[h_{i-\frac{1}{2}}^{*,+}u_{i-\frac{1}{2}}^{+}
+h_{i+\frac{1}{2}}^{*,-}u_{i+\frac{1}{2}}^{-}
-\alpha(h_{i+\frac{1}{2}}^{*,-}-h_{i-\frac{1}{2}}^{*,+} )\right]\\
&~~~~~~~+\frac{\lambda}{2w_{1}}\left[h_{i-\frac{1}{2}}^{*,-}u_{i-\frac{1}{2}}^{-}
+h_{i-\frac{1}{2}}^{*,+}u_{i-\frac{1}{2}}^{+}
-\alpha(h_{i-\frac{1}{2}}^{*,+}-h_{i-\frac{1}{2}}^{*,-} )\right]\\
=& \left[1-\frac{\lambda}{2w_{1}}(\alpha+u_{i-\frac{1}{2}}^{+})\frac{h_{i-\frac{1}{2}}^{*,+}}{h_{i-\frac{1}{2}}^{+}}
-\frac{\lambda}{2w_{1}}(\alpha-u_{i-\frac{1}{2}}^{+})\frac{h_{i-\frac{1}{2}}^{*,+}}{h_{i-\frac{1}{2}}^{+}}\right]h_{i-\frac{1}{2}}^{+}\nonumber\\
&+\left[\frac{\lambda}{2w_{1}}(\alpha+u_{i-\frac{1}{2}}^{-})\frac{h_{i-\frac{1}{2}}^{*,-}}{h_{i-\frac{1}{2}}^{-}}\right]h_{i-\frac{1}{2}}^{-}
+\left[\frac{\lambda}{2w_{1}}(\alpha-u_{i+\frac{1}{2}}^{-})\frac{h_{i+\frac{1}{2}}^{*,-}}{h_{i+\frac{1}{2}}^{-}}\right]h_{i+\frac{1}{2}}^{-},\nonumber\\
Z_{4}= & h_{i+\frac{1}{2}}^{-}
 -\frac{\lambda}{w_{4}}\left[\hat{\bF}^h\left( h_{i+\frac{1}{2}}^{*,-},u_{i+\frac{1}{2}}^{-}; h_{i+\frac{1}{2}}^{*,+},u_{i+\frac{1}{2}}^{+}\right)-\hat{\bF}^h\left( h_{i-\frac{1}{2}}^{*,+},u_{i-\frac{1}{2}}^{+} ; h_{i+\frac{1}{2}}^{*,-},u_{i+\frac{1}{2}}^{-}\right)\right].
\end{align*}
It follows from Lemma \ref{lemma1} that $Z_1\geq0$ and $Z_4\geq0$, and we can get $\overline{h}_{i}^{n+1}\geq0$.
\end{proof}

\end{document}